\documentclass[11pt,leqno]{article}
\usepackage{amsmath,amssymb,mathrsfs,xy,amsthm,ddag,bm,url}
\input{xy}
\xyoption{all}

\makeatletter
 \def\@seccntformat#1{\csname the#1\endcsname.\hspace{2ex}}

 \newcommand{\nsubsection}%
  {\@startsection{subsection}%
  {2}%
  {\z@}%
  {-3.5ex plus -1ex minus -.2ex}%
  {-0ex}%
  {\reset@font\normalsize\bfseries}}%

  \renewcommand{\subsection}%
  {\@startsection{subsection}%
  {2}%
  {\z@}%
  {-3.5ex plus -1ex minus -.2ex}%
  {0ex}
  {\reset@font\normalsize\bfseries}}%

  \renewcommand{\subsubsection}%
  {\@startsection{subsubsection}%
  {3}%
  {\z@}%
  {-1ex}
  {0ex}
  {\reset@font\normalsize\bfseries}}%

 \newcommand{\nnsubsection}%
  {\@startsection{subsection}%
  {2}%
  {\z@}%
  {-3ex}%
  {1ex}%
  {\reset@font\normalsize\bfseries}}%

 \renewcommand{\subsection}%
  {\@startsection{subsection}%
  {2}%
  {\z@}%
  {-3.5ex plus -1ex minus -.2ex}%
  {0ex}
  {\reset@font\normalsize\bfseries}}%

 \@addtoreset{subsection}{section}
 \@addtoreset{equation}{subsection}

 \newcommand{\nsection}{\@startsection{section}{1}{\z@}%
     {-5ex}
     {1ex}
     {\reset@font\center\large\sc}}

 \renewenvironment{thebibliography}[1]
 {\nsection*{\refname\@mkboth{\refname}{\refname}}%
   \list{\@biblabel{\@arabic\c@enumiv}}%
        {\settowidth
	\labelwidth{\@biblabel{#1}}%
         \leftmargin
	 \labelwidth
         \advance
	 \leftmargin
	 \labelsep
         \@openbib@code
         \usecounter{enumiv}%
         \let\p@enumiv\@empty
	 \parskip=0pt
	 \itemsep=1pt
	 \parsep=1pt
	 \itemindent=\z@
         \renewcommand\theenumiv{\@arabic\c@enumiv}}%
   \sloppy
   \clubpenalty4000
   \@clubpenalty\clubpenalty
   \widowpenalty4000%
   \footnotesize
   \sfcode`\.\@m}
  {\def\@noitemerr
    {\@latex@warning{Empty `thebibliography' environment}}%
   \endlist}
\makeatother

\newtheoremstyle{thm}
 {1em}
 {3pt}
 {\itshape}
 {}
 {\bf}
 {. ---}
 {0.5em}
 {}
\newtheoremstyle{dfn}
 {1em}
 {3pt}
 {}
 {}
 {\bf}
 {. {---}}
 {0.5em}
 {}

\swapnumbers
\theoremstyle{thm}
\newtheorem{thm}[subsection]{Theorem}
\newtheorem*{thm*}{Theorem}
\newtheorem{lem}[subsection]{Lemma}
\newtheorem*{lem*}{Lemma}
\newtheorem{cor}[subsection]{Corollary}
\newtheorem*{cor*}{Corollary}
\newtheorem{prop}[subsection]{Proposition}
\newtheorem*{prop*}{Proposition}
\newtheorem*{conj*}{Conjecture}
\newtheorem*{thmM}{Theorem \ref{main}}

\theoremstyle{dfn}
\newtheorem{dfn}[subsection]{Definition}
\newtheorem*{dfn*}{Definition}
\newtheorem{ex}[subsection]{Example}
\newtheorem*{ex*}{Example}
\newtheorem{rem}[subsection]{Remark}
\newtheorem*{rem*}{Remark}
\newtheorem*{cl}{Claim}

\renewcommand{\qedsymbol}{$\blacksquare$}

\usepackage{color}

\newsavebox{\circlebox}
\savebox{\circlebox}{\fontencoding{OMS}\selectfont\large\char13}
\newlength{\circleboxwdht}
\newcommand{\ccirc}[1]{%
  \setlength{\circleboxwdht}{\wd\circlebox}%
  \addtolength{\circleboxwdht}{\dp\circlebox}%
  \raisebox{0.4\dp\circlebox}{%
    \parbox[][\circleboxwdht][c]{\wd\circlebox}{\centering\small #1}}%
  \llap{\usebox{\circlebox}}%
}

\newenvironment{meta}{
\noindent \color{red}
\sffamily[}{\upshape]}

\newcommand{\Gm}[1]{\mb{G}_{\mr{m},{#1}}}
\newcommand{\Thetatil}[1]{\Theta^{#1}}
\newcommand{\Epr}{\mathbb{E}}

\begin{document}
\title{Rings of microdifferential operators for arithmetic
$\ms{D}$-modules.\\{\normalsize --- Construction and an application to
the characteristic varieties for curves}}
\author{Tomoyuki Abe}
\date{}
\maketitle
\begin{abstract}
 One aim of this paper is to develop a theory of microdifferential
 operators for arithmetic $\ms{D}$-modules.
 We first define the rings of microdifferential operators of arbitrary
 levels on arbitrary smooth formal schemes. A difficulty lies in the
 fact that there is no homomorphism between rings of
 microdifferential operators of different levels. To remedy this,
 we define the intermediate differential operators, and using
 these, we define the ring of microdifferential operators for
 $\ms{D}^\dag$. We conjecture that the characteristic variety of a
 $\ms{D}^\dag$-module is computed as the support of the
 microlocalization of a $\ms{D}^\dag$-module, and prove it in the curve
 case.
\end{abstract}
\renewcommand{\abstractname}{R\'{e}sum\'{e}}

\tableofcontents

\section*{Introduction}
This paper is aimed to construct a theory of rings of
microdifferential operators for arithmetic
$\ms{D}$-modules. Let $X$ be a smooth variety over $\mb{C}$. Then the
sheaf of rings of microdifferential operators denoted by $\ms{E}_X$ is
defined on the cotangent bundle $T^*X$ of $X$. This ring is one of
basic tools to study $\ms{D}$-modules microlocally, and it is used in
various contexts. One of the most important and fundamental properties
is
\begin{equation}
 \label{fundamentalproperty}
  \tag{ˆë}
 \mr{Char}(\ms{M})=\mr{Supp}(\ms{E}_X\otimes_{\pi^{-1}\ms{D}_X}\ms{M})
\end{equation}
for a coherent $\ms{D}_X$-module $\ms{M}$, where $\pi:T^*X\rightarrow X$
is the projection. One goal of this study is to find an analogous
equality in the theory of arithmetic $\ms{D}$-modules.

We should point out two attempts to construct rings of microdifferential
operators. The first attempt was made by R. G. L\'{o}pez in
\cite{Lop}. In there, he constructed the ring of microdifferential
operators of {\em finite order on curves}. However, the relation between
his construction and the theory of arithmetic $\ms{D}$-modules was not
clear as he pointed out in the last remark of \cite{Lop} The second
construction was carried out by A. Marmora in \cite{Mar}. Our work can
be seen as a generalization of this work, and we explain the
relation with our construction in the following.

Now, let $R$ be a complete discrete valuation ring of mixed
characteristic $(0,p)$. Let $\ms{X}$ be a
smooth formal scheme over $\mr{Spf}(R)$, and we denote the special fiber
of $\ms{X}$ by $X$. For an integer $m\geq 0$, P. Berthelot defined the
ring of differential operators of level $m$
denoted by $\DcompQ{m}{\ms{X}}$. He also defined the characteristic
varieties for coherent $\DcompQ{m}{\ms{X}}$-modules in almost the
same way we define the characteristic varieties for {\em
analytic} $\ms{D}$-modules. It is natural to
hope that there exists a theory of microdifferential operators, and that
we can define the ring of microdifferential operators
$\EcompQ{m}{\ms{X}}$ of level $m$ associated with $\DcompQ{m}{\ms{X}}$
satisfying an analog of (\ref{fundamentalproperty}). When $\ms{X}$ is a
curve, this was done by Marmora in \cite{Mar}, in his study of Fourier
transform. He fixed a system of local coordinates, constructed
the ring of microdifferential operators using  explicit descriptions
as in \cite[Chapter VIII]{Bjo}, and proved that the construction does
not depend on the choice of local coordinates. In this paper, we use a
general technique of G. Laumon of formal microlocalization of certain
filtered rings (cf.\ \cite{Lau}) to define the ring of naive
microdifferential operators of level $m$ denoted by $\EcompQ{m}{\ms{X}}$
(cf.\ \ref{defofnaivemic}).
One advantage of this construction is that we do not need to choose
coordinates. It follows also formally using the
result of Laumon that for a coherent $\DcompQ{m}{\ms{X}}$-module
$\ms{M}$, we get
\begin{equation*}
 \mr{Char}^{(m)}(\ms{M})=\mr{Supp}(\EcompQ{m}{\ms{X}}\otimes_
  {\pi^{-1}\DcompQ{m}{\ms{X}}}\pi^{-1}\ms{M})
\end{equation*}
in $T^{(m)*}X:=\mr{Spec}(\mr{gr}(\Dmod{m}{X}))$, where $\pi\colon
T^{(m)*}X\rightarrow X$ is the projection, and $\mr{Char}^{(m)}$ denotes
the characteristic variety (cf.\ \ref{charvardef}).

Before explaining the construction of sheaves of microdifferential
operators associated with $\Ddag{\ms{X},\mb{Q}}$, let us review the
theory of Berthelot, and see why we need to consider
$\DdagQ{\ms{X}}$-modules. Berthelot proved that
many fundamental theorems in the theory of analytic $\ms{D}$-modules
hold also for $\DcompQ{m}{\ms{X}}$-modules.
For example, he defined pull-backs and push-forwards, and proved
that push-forwards of coherent modules by proper morphisms remain
coherent (cf.\ \cite{BerInt}).
However, the analogue of Kashiwara's theorem, which states an
equivalence
between the category of coherent $\DcompQ{m}{\ms{X}}$-modules which are
supported on a smooth closed formal subscheme $\ms{Z}$ of $\ms{X}$ and
the category of coherent $\DcompQ{m}{\ms{Z}}$-modules, does not
hold. This failure makes it difficult to define a suitable subcategory of
holonomic modules in the category of $\DcompQ{m}{\ms{X}}$-modules. To
remedy this, Berthelot took inductive
limit on the levels to define the ring $\DdagQ{\ms{X}}$, and proved
an analogue of Kashiwara's theorem for coherent $\DdagQ{\ms{X}}$-modules
(cf.\ \cite[5.3.3]{BerInt}). As in the analytic $\ms{D}$-module theory,
we need to consider holonomic modules
to deal with push-forwards along open immersions, and we need to define
characteristic varieties to define holonomic modules. When a coherent
$\DdagQ{\ms{X}}$-module possesses a {\em Frobenius structure} ({\it
i.e.}\ an isomorphism $\ms{M}\xrightarrow{\sim}F^*\ms{M}$), Berthelot
defined the characteristic variety. He reduced the definition to a
finite level situation using a marvelous theorem of Frobenius descent,
and proved Bernstein's inequality by using the analogue of Kashiwara's
theorem. However, in the absence of Frobenius, the situation is
mysterious.

In this paper, we propose a new formalism which allows us at least
conjecturally to interpret this characteristic varieties by means of
microlocalizations, and use them to define the characteristic varieties
for general coherent $\DdagQ{\ms{X}}$-modules which may not carry
Frobenius structures. We also prove the conjecture in the case of
curves (cf.\ Theorem \ref{main}). Let us describe a more precise
statement and difficulties to carry this out.

One of the difficulties in defining microdifferential operators
associated with $\ms{D}^\dag$ is that there are {\em no}
transition homomorphism (cf.\ \ref{countcharnotstable})
\begin{equation*}
 \EcompQ{m}{\ms{X}}\rightarrow\EcompQ{m+1}{\ms{X}}
\end{equation*}
compatible with
$\DcompQ{m}{\ms{X}}\rightarrow\DcompQ{m+1}{\ms{X}}$. This makes it hard
to define the ring of microdifferential operators
corresponding to $\DdagQ{\ms{X}}$ in a naive way. Let
$\pi:T^*\ms{X}\rightarrow\ms{X}$ be the projection. To remedy this, we
define a $\pi^{-1}\DcompQ{m}{\ms{X}}$-algebra $\EcompQ{m,m'}{\ms{X}}$
for any integer $m'\geq m$ called the ``intermediate ring of
microdifferential operators of level $(m,m')$'' so that there exist
homomorphisms of $\pi^{-1}\DcompQ{m}{\ms{X}}$-algebras
\begin{equation*}
 \EcompQ{m,m'+1}{\ms{X}}\rightarrow\EcompQ{m,m'}{\ms{X}},\qquad
  \EcompQ{m,m'}{\ms{X}}\rightarrow\EcompQ{m+1,m'}{\ms{X}},
\end{equation*}
and $\EcompQ{m,m}{\ms{X}}=\EcompQ{m}{\ms{X}}$. We define
\begin{equation*}
 \Emod{m,\dag}{\ms{X},\mb{Q}}:=\invlim_{m'}\EcompQ{m,m'}{\ms{X}}.
\end{equation*} 
On this level, we have a transition homomorphism
$\Emod{m,\dag}{\ms{X},\mb{Q}}\rightarrow\Emod{m+1,\dag}{\ms{X},\mb{Q}}$
compatible with $\DcompQ{m}{\ms{X}}\rightarrow\DcompQ{m+1}{\ms{X}}$. We
define
\begin{equation*}
 \EdagQ{\ms{X}}:=\indlim_m\Emod{m,\dag}{\ms{X,\mb{Q}}}.
\end{equation*}
Unfortunately, we no longer have the equality
\begin{equation*}
 \mr{Char}^{(m)}(\ms{M})=\mr{Supp}(\Emod{m,\dag}{\ms{X},\mb{Q}}\otimes_
  {\pi^{-1}\DcompQ{m}{\ms{X}}}\pi^{-1}\ms{M})
\end{equation*}
for a coherent $\DcompQ{m}{\ms{X}}$-module $\ms{M}$ in general (cf.\
\ref{countexchsupp}). However, we conjecture the following.

\begin{conj*}
 Let $\ms{X}$ be a quasi-compact smooth formal scheme over $R$, and
 $\ms{M}$ be a coherent $\DcompQ{m}{\ms{X}}$-module. Then there exists
 $N>m$ such that for any $m'\geq N$,
 \begin{equation*}
  \mr{Char}^{(m')}(\DcompQ{m'}{\ms{X}}\otimes_{\DcompQ{m}{\ms{X}}}
   \ms{M})=\mr{Supp}(\Emod{m',\dag}{\ms{X},\mb{Q}}\otimes_{\pi^{-1}
   \DcompQ{m}{\ms{X}}}\pi^{-1}\ms{M}).
 \end{equation*}
\end{conj*}
This conjecture implies that $\mr{Car}(\ms{M})=\mr{Supp}(\EdagQ{\ms{X}}
\otimes\ms{M})$ for a coherent $F$-$\DdagQ{\ms{X}}$-module $\ms{M}$
where $\mr{Car}$ denotes the characteristic variety defined by
Berthelot. It is also worth noticing here that if this conjecture is
true, the characteristic varieties for coherent
$\DcompQ{m}{\ms{X}}$-modules stabilize when we raise the level
$m$, and in particular we are able to define characteristic varieties
for coherent $\DdagQ{\ms{X}}$-modules even without Frobenius
structures. In the last part of this paper, we prove the following.

\begin{thmM}
 When $\ms{X}$ is a curve, the conjecture is true.
\end{thmM}

Finally, let us point out one of the most important applications of
this theorem. The construction of the ring of microdifferential
operators in this paper and Theorem \ref{main} are crucial technical
tools for the proof of the product formula for $p$-adic epsilon factors
in \cite{AM}. Especially, the result of this paper is
used to establish the theory of $p$-adic local Fourier transform and the
``principle of stationary phase''.
This product formula is expected to be used to establish the
``Langlands correspondence'' of overconvergent $F$-isocrystals. 
In the celebrated paper of Lafforgue \cite{Laf}, Langlands
correspondence for function field was established. This is a certain
correspondences between $\ell$-adic Galois representations and cuspidal
automorphic forms. A natural question is if there is any analogous
correspondences for overconvergent $F$-isocrystals. We believe that
there are similar correspondences, and we observe that the
conjecture of Deligne \cite[1.2.10 (vi)]{De} (or more precisely
\cite[4.13]{Cr}) can be understood as a consequence of these unknown
correspondences. See \cite{AbeL} for more details.
\bigskip

To conclude the introduction, let us see the structure of this paper. In
\S1, we review the theory of formal microlocalization of
certain filtered rings, and single out some basic cases where these
rings are noetherian (according to Definition \ref{defnoethrigsh}).
Using these results, we define the naive
ring of microdifferential operators $\EcompQ{m}{\ms{X}}$, and prove some
basic facts in the next section \S2. Before proceeding to the
definition of the intermediate rings of microdifferential operators, we
study some properties of $\mr{gr}(\Emod{m}{\ms{X},\mb{Q}})$ in \S
3. These are used to study the intermediate rings, which are defined in
\S4. In \S5, we prove the flatness of transition homomorphisms and
related results. One of the most important properties of
$\Emod{m,\dag}{\ms{X},\mb{Q}}$ is that its
sections over a strict affine open subscheme form a Fr\'{e}chet-Stein
algebra. In \S6, we prove a finiteness property of certain sheaves of
modules, which may be useful to deal with sheaves on formal schemes. In
the last section, we formulate the conjecture, and prove it in the case
of curves.

\nnsubsection*{Acknowledgments}
A large part of the work was done while the author was
visiting IRMAR at {\it l'Universit\'{e} de Rennes I}. He would like to
thank Professor Pierre Berthelot for giving him a lot of advice, and
Professor Ahmed Abbes for his kind hospitality and help him to improve
the paper. He is also grateful to Yoichi Mieda for answering many
questions, and reading some proofs. He would like to thank Adriano
Marmora for fruitful discussions. Finally he would like to thank
Professors Shuji Saito and Atsushi Shiho for continuous support and
encouragement. This work was supported by Grant-in-Aid for JSPS Fellows
20-1070 and partially by JSPS Core-to-Core 18005.

\nnsubsection*{Conventions}
In this paper, all rings are assumed to be associative with
unity. Filtered groups are assumed to be exhaustive (cf.\
\ref{termfiltsoon}), and modules are left modules unless otherwise
stated.
In principle, we use Raman fonts ({\it e.g.}\ $X$) for
schemes, and script fonts ({\it e.g.}\ $\mathscr{X}$) for formal
schemes.

\section{Preliminaries on filtered rings}
The aim of this section is to review the formal construction of the
microlocalization of certain filtered rings due to O.\ Gabber and G.\
Laumon. To fix notation and terminology, we begin by reviewing
well-known definitions and properties of filtered modules.

\subsection{}
\label{prelimfiltterm}
The reader can refer to \cite[III, \S2]{Bour} and \cite{HO} for more
details.

\subsubsection{}
\label{termfiltsoon}
An increasing sequence $\{G_n\}_{n\in\mb{Z}}$ of subgroups of a group
$G$ is called an increasing filtration on $G$. The filtration is said to
be {\em positive} if $G_n=0$ for all $n<0$. We say that the filtration
is {\em separated} if $\bigcap_nG_n=\{e\}$ where $e$ is the unit. If
$G_n$ is a normal subgroup of $G$ for any $n$, the filtration defines a
canonical topology, which makes $G$ a topological group (cf.\ \cite[III,
\S2.5]{Bour}).

Let $A$ be a ring (not necessary commutative), and
$\{A_i\}_{i\in\mb{Z}}$ be a filtration of the additive group $A$. We
say that the couple $(A,\{A_i\}_{i\in\mb{Z}})$ is a {\em filtered ring}
if $A_i\cdot A_j\subset A_{i+j}$, and $1\in A_0$. We always assume that
the filtration is {\em exhaustive} ({\it i.e.\
}$\bigcup_iA_i=A$). If there is no possible confusion, we abbreviate it
by $(A,A_i)$. Let $M$ be an $A$-module, and $\{M_i\}_{i\in\mb{Z}}$ be a
filtration of the additive group $M$ such that $A_i\cdot M_j\subset
M_{i+j}$ for any $i,j\in\mb{Z}$. Then the couple
$(M,\{M_i\}_{i\in\mb{Z}})$ is said to be a {\em filtered
$(A,A_i)$-module}. We often denote $(M,\{M_i\}_{i\in\mb{Z}})$ by
$(M,M_i)$ for short.

\subsubsection{}
\label{adicfiltex}
Let $A$ be a ring, and $I$ be a two-sided ideal. We put $A_n:=I^{-n}$
for $n\leq0$, and $A_n:=A$ for $n>0$. The couple
$(A,\{A_n\}_{n\in\mb{Z}})$ is a filtered ring, and the filtration is
called the {\em $I$-adic filtration}.

Let $(M,M_i)$ be a filtered $(A,A_i)$-module.
We say that the filtration $\{M_i\}_{i\in\mb{Z}}$ of $M$ is
{\em good} if there exist $m_1,\dots,m_s\in M$ and
$k_1,\dots,k_s\in\mb{Z}$ such that $M_n=\sum_{i=1}^sA_{n-k_i}\cdot m_i$
for any $n$.

\subsubsection{}
A {\em filtered homomorphism} $f\colon(A,A_i)\rightarrow(B,B_i)$ is a
ring homomorphism $f\colon A\rightarrow B$ such that there exists an
integer $n$ satisfying $f(A_i)\subset B_{i+n}$ for any integer $i$.
Such a homomorphism is continuous with respect to the topology defined
by the filtration on $A$ and $B$.
The filtered homomorphism $f$ is said to be {\em strict} if
$f(A_i)=f(A)\cap B_i$ for any $i\in\mb{Z}$.

\subsubsection{}
\label{princisymbdfn}
For a filtered ring $(A,A_i)$, we put $\mr{gr}_i(A):=A_i/A_{i-1}$,
and $\mr{gr}(A):=\bigoplus_{i}\mr{gr}_i(A)$. The module $\mr{gr}(A)$ is
naturally a graded ring, and it is called the {\em associated graded
ring}. We define the {\em principal symbol map} $\sigma\colon
A\rightarrow\mr{gr}(A)$ in the following way: let $x\in A$. If
$x\in\bigcap_iA_i$, then we put $\sigma(x)=0$. Otherwise there exists an
integer $i$ such that $x\in A_{i}$ and $x\not\in A_{i-1}$. We define
$\sigma(x)$ to be the image of $x$ in $\mr{gr}_i(A)\subset\mr{gr}(A)$.

\subsubsection{}
We introduce the completion of a filtered ring. Let $(A,A_i)$
be a filtered ring. We refer to \cite[Ch.I, \S3]{HO} for the
details.
Let $A[\nu,\nu^{-1}]$ be the ring of Laurent
polynomials with one variable $\nu$ over $A$, graded by the degree of
$\nu$. Here, the element $\nu$ is in the center by definition.
We define the graded subalgebra of $A[\nu,\nu^{-1}]$ denoted by
$A_{\bullet}$, called the {\em Rees ring} of $(A,A_i)$ by the formula
\begin{equation*}
 A_{\bullet}:=\bigoplus_{i\in\mb{Z}}A_i\cdot\nu^i.
\end{equation*}
For an integer $n\geq 1$, we define a graded ring
$A_{\bullet,n}:=A_{\bullet}/\nu^n
A_{\bullet}\cong\bigoplus_{i\in\mb{Z}} A_i/A_{i-n}\cdot\nu^i$.
For $i\in\mb{Z}$, we put $A_{i,n}:=A_i/A_{i-n}$, the part of degree
$i$ of $A_{\bullet,n}$.
We get a projective system of graded rings
\begin{equation*}
 \rightarrow A_{\bullet,n+1}\rightarrow A_{\bullet,n}\rightarrow\dots
  \rightarrow A_{\bullet,1}\cong\mr{gr}(A).
\end{equation*}
We define a module and a ring by
\begin{equation*}
 \widehat{A}_i:=\invlim_{n\rightarrow\infty}A_{i,n}\,
  \bigl(\cong\invlim_{n\rightarrow\infty}A_i/A_{i-n}\bigr),
  \qquad\widehat{A}:=\indlim_{i\rightarrow\infty}\,
  \widehat{A}_i.
\end{equation*}
The couple $(\widehat{A},\{\widehat{A}_i\}_{i\in\mb{Z}})$ is a filtered
ring, and is called the {\em completion of $(A,A_i)$}. This definition
coincides with \cite[Ch.I, 3.4.1]{HO} by \cite[Ch.I, 3.5, (d) and
(e)]{HO}. We note that the completion is separated, and the
canonical homomorphism $\mr{gr}(A)\xrightarrow{}\mr{gr}(\widehat{A})$ is
an isomorphism by \cite[Ch.I, 4.2.2]{HO}. We say that the filtered ring
$(A,A_i)$ is {\em complete} if the canonical homomorphism
$A\rightarrow\widehat{A}$ is an isomorphism of filtered rings.

\subsubsection{}
\label{noethfiltdef}
We say that a filtered ring $(A,A_i)$ is {\em left} (resp.\ {\em right},
{\em two-sided}) {\em noetherian filtered}
if the Rees ring $A_{\bullet}$ is left (resp.\ right, two-sided)
noetherian. If $A$ is a noetherian filtered ring, the associated graded
ring $\mr{gr}(A)$ is noetherian since $\mr{gr}(A)\cong A_\bullet/\nu
A_\bullet$. If $A$ is a complete filtered ring, $A$ is noetherian
filtered if and only if $\mr{gr}(A)$ is a noetherian
ring (cf.\ \cite[Ch.II, 1.2.3]{HO}). This shows that the completion of a
noetherian filtered ring is noetherian filtered. Moreover, if $(A,A_i)$
is a noetherian filtered ring, the canonical homomorphism
$A\rightarrow\widehat{A}$ is flat by \cite[Ch.I, 1.2.1]{HO}. We say that
a noetherian filtered ring is {\em Zariskian} if any good filtered
module is separated. Any noetherian filtered complete ring is known to
be Zariskian (cf.\ \cite[Ch.II, 2.2.1]{HO}).

\subsubsection{}
Let $X$ be a topological space or, more generally, topos. The
terminologies defined so far {\em except for those defined in
{\normalfont\ref{noethfiltdef}} and principal symbol in
{\normalfont\ref{princisymbdfn}}} can be defined also in the language of
sheaves by replacing ``ring'' by ``sheaf of rings on $X$'' and so
on. See \cite[A.III.2]{Bjo} for more details.

\subsection{}
\label{microlocaldef}
Let $(A,A_i)$ be a filtered ring which is complete and the associated
graded ring $\mr{gr}(A)$ is {\em commutative}. Let
$S_1\subset\mr{gr}(A)$ be a homogeneous multiplicative set ({\it i.e.\
}a multiplicative set consisting of homogeneous elements). Let
$c_n\colon A_{\bullet,n}\rightarrow A_{\bullet,1}\cong\mr{gr}(A)$ be the
canonical homomorphism. We put
\begin{equation*}
 S_n:=\bigl\{x\in A_{\bullet,n}\,\big|\,c_n(x)\in S_1\bigr\}.
\end{equation*}
By \cite[A.2.1]{Lau}, the multiplicative set $S_n$ satisfies the
two-sided Ore condition (cf.\ \cite[4.\ \S10A]{Lam}). We define a {\em
graded ring} by
\begin{equation*}
 A'_{\bullet,n}:=S_n^{-1}A_{\bullet,n}\cong A_{\bullet,n}S_n^{-1}.
\end{equation*}
This defines a projective system of graded rings
$\{A'_{\bullet,n}\}$. Let us denote by $A'_{i,n}$ the part of degree $i$
of $A'_{\bullet,n}$. We define
\begin{equation*}
 A'_i:=\invlim_{n\rightarrow\infty}A'_{i,n},\qquad
  A':=\indlim_{i\rightarrow\infty}A'_i.
\end{equation*}
The filtered ring $(A',A'_i)$ is complete. We denote this filtered ring
by $(A,A_i)_{S_1}$, and we call it the {\em microlocalization of
$(A,A_i)$ with respect to $S_1$}. If $\mr{gr}(A)$ is noetherian, then
$A'_0$ and $A'$ are noetherian and the canonical homomorphism
$A\rightarrow A'$ is flat by \cite[Corollaire A.2.3.4]{Lau}.

Let us describe elements of $A'$ concretely. Put $S:=\bigl\{a\in A\mid
\sigma(a)\in S_1\bigr\}$, and take $s\in
S$. By definition, $s$ is invertible in $(A,A_i)_{S_1}$. Given an
element $a\in A'_i$, there exist $a_k\in A_{l_k}$ and $s_k\in S\cap
A_{l_k-k}$ (thus $a_k\,s_k^{-1}\in A'_{k}$) for each integer $k\leq i$,
such that $a=\sum_{k\leq i}a_k\,s_k^{-1}$, where the sum is infinite and
we consider the topology defined by the filtration of $A$ (cf.\
\ref{termfiltsoon}). Moreover,
assume that $\sigma(s)\in\mr{gr}_N(A)$ and
$S_1=\bigl\{\sigma(s)^n\bigr\}_{n\geq0}$. Then for any $s'\in S$, there
exists an integer $l$ such that $\sigma(s')=\sigma(s)^l$. Since
$a:=s^l-s'\in A_{Nl-1}$, $s'^{-1}=s^{-l}\cdot\sum_{k\geq0}(as^{-l})^k$.
Thus for any $a'\in A'_i$, there exist an integer $n_k\geq0$ and
$a'_k\in A_{k+Nn_k}$ such that $a'=\sum_{k\leq
i}a'_k\,s^{-n_k}$.

\subsection{}
Let $(A,A_i)$ be a complete filtered ring whose associated graded ring
is commutative. The constructions in the previous subsection can be
carried out in almost the same way also for filtered
$(A,A_i)$-modules. For the details see \cite[A.2]{Lau}. For example, for
a filtered $A$-module $(M,M_i)$ and a homogeneous multiplicative system
$S_1\subset\mr{gr}(A)$, we are able to define the microlocalization of
$(M,M_i)$ with respect to $S_1$ denoted by $(M,M_i)_{S_1}$, which is
complete.

\subsection{}
\label{subsecmicloc}
Let us sheafify the results. Let $(A,\{A_i\}_{i\in\mb{Z}})$ be a
positively filtered ring such that the associated graded ring
$\mr{gr}(A)$ is commutative.
Let $R:=\mr{gr}(A)$ be the positively graded commutative
ring. Note that $A_0=R_0$ is a commutative ring by assumption. We let
$X:=\mr{Spec}(R_0)$, $V:=\mr{Spec}(R)$,
$P:=\mr{Proj}(R)$. Let $s\colon X\rightarrow V$ be the morphism defined
by the canonical projection $R\rightarrow R_0$. We put
$\mathring{V}:=V\setminus s(X)$. We have the following
canonical commutative diagram (cf. \cite[II, 8.3]{EGA}).
\begin{equation*}
 \xymatrix@C=40pt@R=10pt{
  V\ar[rd]|p&\mathring{V}\ar[r]^<>(.5){q}\ar[l]\ar[d]&P\ar[dl]\\
 &X\ar@/^1pc/[ul]|s&}
\end{equation*}
We define a topological space $V'$ in the following way: as a set,
$V':=V$. The topology of $V'$ is generated by the basis of open sets
\begin{equation*}
 \bigl\{D(f)\ \big|\ \mbox{$f\in R$ and $f$ is homogeneous}\bigr\}.
\end{equation*}
We denote by $\epsilon\colon V\rightarrow V'$ the identity map as sets,
which is continuous. In the sequel, various sheaves are defined
naturally on $V'$. A general strategy of \cite{Lau} is to define sheaves
on $V$, which are final outputs, by taking $\epsilon^{-1}$. Now, let us
denote by $\mf{O}(T)$, for a topological space $T$, the category of open
sets of $T$. The canonical functor
$\epsilon^{-1}\colon\mf{O}(V')\rightarrow\mf{O}(V)$ admits a left
adjoint denoted by $\epsilon_\cdot$.
For $U\in\mf{O}(V)$, this functor can be described as
$\epsilon_\cdot(U)=\bigcup_{\lambda\in\Gm{X}}\lambda\cdot U$ where
$\Gm{X}$ acts on $V$ naturally. Using this
functor, $\epsilon^{-1}$ can clearly be calculated: let
$\mc{F}'$ be a sheaf on $V'$. Then we have
\begin{equation}
 \label{noprimecalcfromprima}
 (\epsilon^{-1}\mc{F}')(U)=\mc{F}'(\epsilon_\cdot(U))
\end{equation}
by \cite[A.3.0.2]{Lau}.
Thanks to this equality, many properties of $\mc{F}'$ also automatically
apply to that of $\epsilon^{-1}\mc{F}'$.

Having this strategy in mind, let us define first important sheaves on
$V$. Let $\mc{O}_{V'}:=\epsilon_*\mc{O}_V$. For
$n\in\mb{Z}$, we denote by $\mc{O}_{V'}(n)$, the
subsheaf of $\mc{O}_{V'}$ consisting of the homogeneous sections of
degree $n$. We put $\mc{O}_V(n):=\epsilon^{-1}(\mc{O}_{V'}(n))$, and
$\mc{O}_V(*):=\bigoplus_{n\in\mb{Z}}\mc{O}_V(n)$. We note that
$\mc{O}_V(*)\cong\epsilon^{-1}\epsilon_*\mc{O}_V$. We get
$\mc{O}_V(n)|_{\mathring{V}}\cong q^{-1}\mc{O}_P(n)$ for any integer
$n$ by \cite[A.3.0.5]{Lau}. By \cite[A.3.0.5]{Lau}, we also have
\begin{equation}
 \label{zerosectionpullcallau}
 p_*\mc{O}_V(*)\cong s^{-1}\mc{O}_V(*)\cong\widetilde{R}
\end{equation}
where $\widetilde{\cdot}$ denotes
the associated quasi-coherent $\mc{O}_X$-module.

\subsection{}
\label{micshdfe}
Let $(\mc{A},\mc{A}_i)$ be the filtered quasi-coherent
$\mc{O}_X$-algebra associated to $(A,A_i)$ on $X$. Let $f$ be a
homogeneous element of $\mr{gr}(A)$, and we put
$S_1(f):=\{f^m\}_{m\geq0}\subset\mr{gr}(A)$. Let $S_n(f)$ be the
multiplicative set of $A_{\bullet,n}$ constructed from $S_1(f)$ (see
\ref{microlocaldef}), and define
$A'_{\bullet,n}(f):=S_n(f)^{-1}A_{\bullet,n}$.
We define a sheaf $\mc{B}'_{\bullet,n}$ on $V'$ to be the sheaf
associated to the presheaf $D(f)\mapsto A'_{\bullet,n}(f)$
over the open basis of $V'$ consisting of $D(f)$ with a homogeneous
element $f$ in $\mr{gr}(A)$. By \cite[A.3.1.1]{Lau}, we know that
\begin{equation}
 \label{quoticalclau}
 \Gamma(D(f),\mc{B}'_{\bullet,n})=A'_{\bullet,n}(f).
\end{equation}
We define
\begin{equation*}
 \mc{B}'_{i}:=\invlim_{n\rightarrow\infty}\mc{B}'_{i,n},\qquad
  \mc{B}':=\indlim_{i\rightarrow\infty}\mc{B}'_i.
\end{equation*}
Then we have an isomorphism of complete filtered rings
$(A,A_i)_{S_1(f)}\xrightarrow{\sim}
\Gamma(D(f),(\mc{B}',\mc{B}'_i))$ for a homogeneous element $f$ of
$\mr{gr}(A)$ by \cite[(A.3.1.2)]{Lau}. Now, let us use the general
machinery to define a filtered sheaf of rings on $V$ as follows:
\begin{equation*}
 (\mc{B},\mc{B}_i):=\epsilon^{-1}(\mc{B}',\mc{B}'_i).
\end{equation*}
There is a canonical homomorphism of filtered rings $\varphi\colon
p^{-1}(\mc{A},\mc{A}_i)\rightarrow(\mc{B},\mc{B}_i)$ on $V$.
The filtered ring $(\mc{B},\mc{B}_i)$ is called the {\em
microlocalization of $(\mc{A},\mc{A}_i)$}. By \cite[A.3.1.6]{Lau}, we
have canonical isomorphisms of graded rings
\begin{equation}
 \label{gradedisomicdi}
  \mr{gr}_n(\mc{B})\cong\mc{O}_V(n),\qquad
  \mr{gr}(\mc{B})\cong\mc{O}_V(*).
\end{equation}

\begin{rem*}
 Note that $q_*(\mc{O}_{\mathring{V}}(n))\cong\mc{O}_P(n)$ (resp.\
 $\mc{O}_{\mathring{V}}(n)$) is a quasi-coherent $\mc{O}_P$-module
 (resp.\ $\mc{O}_{\mathring{V}}(0)$-module). However, caution that,
 {\it a priori}, $q_*(\mc{B}|_{\mathring{V}})$ and
 $q_*(\mc{B}_i|_{\mathring{V}})$ {\em do not} have $\mc{O}_P$-module
 structure, nor do $\mc{A}|_{\mathring{V}}$ and
 $\mc{A}_i|_{\mathring{V}}$ have $\mc{O}_{\mathring{V}}$-module or
 $\mc{O}_{\mathring{V}}(0)$-module structure.
\end{rem*}

\subsection{}
Let $(M,M_i)$ be a filtered $(A,A_i)$-module such that $M_i=0$ for $i\ll
0$ (and $\bigcup_{i\in\mb{Z}}M_i=M$). Let $(\mc{M},\mc{M}_i)$ be the
quasi-coherent $\mc{O}_X$-module associated to the filtered module
$(M,M_i)$. This is a filtered $(\mc{A},\mc{A}_i)$-module.

Using exactly the same construction (cf.\ \cite[A.3.2]{Lau}), we are
able to define a $(\mc{B}',\mc{B}'_i)$-module $(\mc{N}',\mc{N}'_i)$ on
$V'$ such that we have an isomorphism of complete filtered modules
$(M,M_i)_{S_1(f)}\xrightarrow{\sim}\Gamma(D(f),(\mc{N}',\mc{N}'_i))$
for a homogeneous element $f$ of $\mr{gr}(A)$. We define a filtered
$(\mc{B},\mc{B}_i)$-module by
\begin{equation*}
 (\mc{N},\mc{N}_i):=\epsilon^{-1}(\mc{N}',\mc{N}'_i).
\end{equation*}
There is a homomorphism $\varphi_M\colon
p^{-1}(\mc{M},\mc{M}_i)\rightarrow(\mc{N},\mc{N}_i)$ over $\varphi$.

\begin{lem}
 \label{charandsuppformal}
 Let $\widetilde{\mr{gr}}(\mc{M})$ be the quasi-coherent
 $\mc{O}_V$-module such that
 $p_*(\widetilde{\mr{gr}}(\mc{M}))\cong\mr{gr}(\mc{M})$.
 Suppose $\mr{gr}(M)$ is finitely presented over $\mr{gr}(A)$. Then we
 have the following equalities in $V$:
 \begin{equation*}
  \mr{Supp}(\widetilde{\mr{gr}}(\mc{M}))=\mr{Supp}(\mr{gr}(\mc{N}))
   =\mr{Supp}(\mc{B}\otimes_{p^{-1}\mc{A}}p^{-1}\mc{M}).
 \end{equation*}
\end{lem}
\begin{proof}
 The first equality follows from \cite[Proposition A.3.2.4
 (i)]{Lau}. Let us show the second one. Let $U:=D(f)$ with a homogeneous
 element $f$ of $\mr{gr}(A)$. Since $\Gamma(U,\mc{N})$ is complete, it
 is in particular separated with respect to the filtration. Thus,
 $\Gamma(U,\mc{N})=0$ if and only if
 $\Gamma(U,\mr{gr}(\mc{N}))\cong\mr{gr}\bigl(\Gamma(U,\mc{N})\bigr)=0$
 where the first isomorphism follows from \cite[A.3.2]{Lau}.
 Combining this with \cite[Proposition A.3.2.4
 (ii)]{Lau}, the lemma follows.
\end{proof}

\begin{dfn*}
 Assume $\mc{M}$ is an $\mc{A}$-module of finite type. Then there exists
 a good filtration $\{\mc{M}_i\}_{i\in\mb{Z}}$
 (cf.\ \cite[A.III 2.15]{Bjo}) of $\mc{M}$.
 Suppose $\mr{gr}(A)$ is noetherian. The above lemma implies that
 $\mr{Supp}(\widetilde{\mr{gr}}(\mc{M}))$ does not depend on the
 choice of a good filtration. We call this the
 {\em characteristic variety of $\mc{M}$} and denote by
 $\mr{Char}(\mc{M})$.
\end{dfn*}

\begin{rem}
 These construction localize and are functorial. In particular, we may
 globalize the definitions of microlocalizations on schemes not
 necessary affine (cf.\ \cite[A.3.3]{Lau}).
\end{rem}

\subsection{}
Now, we collect some basic facts on noetherian conditions.

\begin{dfn*}
 \label{defnoethrigsh}
 Let $X$ be a topological space, $\mc{A}$ be a sheaf of rings on
 $X$, and $\mf{B}$ be an open basis of the topology.

 (i) The ring $\mc{A}$ is said to be {\em left noetherian with respect
 to $\mf{B}$} if it satisfies the following conditions.
 \begin{enumerate}
  \item It is a left coherent ring ({\it i.e.\ }locally, any finitely
	generated left ideal of $\mc{A}$ is finitely presented).
  \item For any point $x\in X$, the stalk $\mc{A}_x$ is a left
	noetherian ring.
  \item For any	$U\in\mf{B}$, $\Gamma(U,\mc{A})$ is a left noetherian
	ring.
 \end{enumerate}
 In the same way, we define a {\em right} (resp.\ {\em two-sided}) {\em
 noetherian ring with respect to $\mf{B}$}. When there is no possible
 confusion, we abbreviate two-sided noetherian sheaf of rings with
 respect to $\mf{B}$ as noetherian ring.

 (ii) A filtered ring $(\mc{A},\mc{A}_i)$ is said to be {\em pointwise
 left} (resp.\ {\em right, two-sided}) {\em Zariskian} if the stalk
 $\mc{A}_x$ is left (resp.\ right, two-sided) Zariskian for any $x\in
 X$.

 (iii) An $\mc{A}$-algebra $\mc{B}$ is said to be {\em of finite type}
 over $\mc{A}$ if for any $x\in X$, there exists an open neighborhood
 $U$ of $x$ and a surjection
 $\mc{A}[T_1,\dots,T_n]|_U\twoheadrightarrow\mc{B}|_U$.
\end{dfn*}

\begin{rem*}
\label{laumonnoeth}
 This definition of noetherian ring is slightly different from that of
 \cite[Definition 1.1.1]{KK}, who replaced 3 by Condition (c):
 for any open set $U$ of $X$, a sum of left coherent $\mc{A}|_U$-ideals
 are also coherent. In \S\ref{snoethcond}, we show that a stronger
 condition than Condition (c) holds for some of the noetherian rings
 defined in this paper.
\end{rem*}

\begin{ex*}
 Let $X$ be a noetherian scheme. Let $\mf{B}$ be the open basis
 consisting of affine open subschemes of $X$. Then $\mc{O}_X$ is a
 noetherian ring with respect to $\mf{B}$. More generally, let $\ms{X}$
 be a locally noetherian adic formal scheme (cf.\ \cite[I,
 10.4.2]{EGA}) and $\mf{C}$ be the open basis consisting of affine open
 formal subschemes of $\ms{X}$. Then $\mc{O}_{\ms{X}}$ is noetherian
 with respect to $\mf{C}$ by \cite[I, 10.1.6]{EGA}.
\end{ex*}

\subsection{}
The following lemma is a generalization of \cite[3.3.6]{Ber1} to
filtered rings.
\begin{lem*}
 \label{noetherianlemma}
 Let $(\mc{A},\{\mc{A}_i\}_{i\in\mb{Z}})$ be a filtered ring on a
 topological space $X$. Let $\mf{B}$
 be an open basis of the topological space $X$. Suppose that the
 following conditions hold.
 \begin{enumerate}
  \item For any $U\in\mf{B}$, the filtered ring
	$\bigl(\Gamma(U,\mc{A}),\Gamma(U,\mc{A}_i)\bigr)$ is
	complete.
  \item The graded ring $\mr{gr}(\mc{A})$ is left noetherian with
	respect to $\mf{B}$.
  \item For $V,U\in\mf{B}$ such that $V\subset U$, the restriction
	homomorphism
	$\Gamma(U,\mr{gr}(\mc{A}))\rightarrow\Gamma(V,\mr{gr}(\mc{A}))$
	is right flat.
  \item For any $U\in\mf{B}$, the canonical homomorphism
	$\mr{gr}(\Gamma(U,\mc{A}))\rightarrow\Gamma(U,\mr{gr}(\mc{A}))$
	is an isomorphism.
\end{enumerate}
 Then, for any $x\in X$, the canonical homomorphism
 \begin{equation}
  \label{faithflathom}
  \mc{A}_x\rightarrow\mc{A}^\wedge_x
 \end{equation}
 is right faithfully flat, where $^{\wedge}$ denotes the completion with
 respect to the filtration on $\mc{A}_x$. Moreover, $(\mc{A},\mc{A}_i)$
 is pointwise left Zariskian, and $\mc{A}$ is left noetherian with
 respect to $\mf{B}$. The statement is also valid if we replace left
 (resp.\ right) by right (resp.\ left).
\end{lem*}

\begin{proof}
 We only deal with the left case, and modules are always
 assumed to be left modules. Let $x\in X$, and take $U\in\mf{B}$ such
 that $x\in U$. Let us check that the restriction homomorphism
 \begin{equation}
  \label{isthisflat}
  r\colon\Gamma(U,\mc{A})\rightarrow\mc{A}_x^{\wedge}
 \end{equation}
 is flat. Indeed, consider the following commutative diagram
 \begin{equation*}
  \xymatrix{
   &\mr{gr}(\Gamma(U,\mc{A}))\ar[r]^{\sim}\ar[dl]_{\mr{gr}(r)}\ar[d]&
   \Gamma(U,\mr{gr}(\mc{A}))\ar[d]\\
  \mr{gr}(\mc{A}_x^\wedge)&\mr{gr}(\mc{A}_x)\ar[l]^<>(.5){\sim}
   \ar[r]_<>(.5){\sim}&\mr{gr}(\mc{A})_x
   }
 \end{equation*}
 where the vertical homomorphisms are the restriction homomorphism
 of $\mc{A}$ and $\mr{gr}(\mc{A})$. The upper horizontal homomorphism is
 an isomorphism by condition 4.
 The right vertical homomorphism is flat by condition 3, and thus
 $\mr{gr}(r)$ is flat as well. Since the filtered rings
 $\Gamma(U,\mc{A})$ and $\mc{A}^\wedge_x$ are complete
 and their associated graded rings are noetherian by conditions 2 and 4,
 these filtered rings are in fact noetherian filtered (cf.\
 \ref{noethfiltdef}). Since the source and the target of $r$ are
 noetherian filtered complete rings, $r$ is flat by the flatness of
 $\mr{gr}(r)$ and \cite[Ch.II, 1.2.1]{HO}. By taking the inductive
 limit over $U$, (\ref{faithflathom}) is flat.

 We say that an $\mc{A}_x$-module $M$ is {\em monogenic of finite
 presentation} if there exists a surjection $\mc{A}_x\rightarrow M$
 such that the kernel is a finitely generated ideal of $\mc{A}_x$.
 By \cite[3.3.5]{Ber1}, to check that (\ref{faithflathom}) is faithful,
 it suffices to show that for any $\mc{A}_x$-module $M$ monogenic of
 finite presentation such that $\mc{A}_x^\wedge\otimes{M}=0$, we get
 $M=0$. So assume $M$ to be monogenic of finite presentation such that
 $\mc{A}_x^\wedge\otimes{M}=0$.
 By this assumption, there exist $U\in\mf{B}$ and a
 $\Gamma(U,\mc{A})$-module ${M}_U$ monogenic of finite presentation such
 that $\mc{A}_x\otimes M_U\cong M$. We fix a surjection
 $\phi\colon\mc{A}_U:=\Gamma(U,\mc{A})\rightarrow M_U$. This induces a
 good filtration on $M_U$. We define an $\mc{A}_U$-ideal $K$
 by the following short exact sequence
 \begin{equation*}
  0\rightarrow K\rightarrow\mc{A}_U\xrightarrow{\phi}M_U\rightarrow 0.
 \end{equation*}
 We consider the induced filtration from $\mc{A}_U$ on $K$. Then we have
 the following exact sequence
 \begin{equation*}
  0\rightarrow\mr{gr}(K)\rightarrow\mr{gr}(\mc{A}_U)
   \xrightarrow{\mr{gr}(\phi)}\mr{gr}(M_U)\rightarrow 0.
 \end{equation*}
 Since $\mr{gr}(\mc{A}^\wedge_x)$ is flat over
 $\mr{gr}(\mc{A}_U)\cong\Gamma(U,\mr{gr}(\mc{A}))$, the sequence
 \begin{equation}
  \label{shortexactfilt}
  0\rightarrow\mr{gr}(\mc{A}^\wedge_x)\otimes_{\mr{gr}(\mc{A}_U)}
  \mr{gr}(K)\rightarrow\mr{gr}(\mc{A}^\wedge_x)
  \rightarrow\mr{gr}(\mc{A}^\wedge_x)\otimes_{\mr{gr}(\mc{A}_U)}
   \mr{gr}(M_U)\rightarrow 0
 \end{equation}
 is exact. The sequence
 \begin{equation}
  \label{flatisom}
  0\rightarrow\mc{A}_x^\wedge\otimes_{\mc{A}_U}K\rightarrow
  \mc{A}_x^\wedge\otimes_{\mc{A}_U}\mc{A}_U\rightarrow 0
 \end{equation}
 is also exact by the flatness of $r$ in (\ref{isthisflat}) and
 the hypothesis on $M$. We endow these two modules with the tensor
 filtrations (cf.\ \cite[p.57]{HO}). Consider the following diagram:
 \begin{equation*}
  \xymatrix{
   0\ar[r]&\mr{gr}(\mc{A}^\wedge_x)\otimes\mr{gr}(K)\ar[r]^<>(.5){\beta}
   \ar[d]&\mr{gr}(\mc{A}^\wedge_x)\ar[r]\ar[d]^{\sim}&
   \mr{gr}(\mc{A}^\wedge_x)\otimes\mr{gr}(M_U)\ar[r]&0\\
  &\mr{gr}(\mc{A}_x^\wedge\otimes K)\ar[r]^<>(.5){\alpha}&\mr{gr}(\mc{A}
   _x^\wedge\otimes\mc{A}_U)&&}
 \end{equation*}
 where the vertical homomorphisms are canonical ones. The right vertical
 homomorphism is an isomorphism by \cite[Ch.I, 6.15]{HO}. Since
 the upper row is exact and
 \begin{equation*}
  \mr{gr}(\mc{A}_x^\wedge)\otimes\mr{gr}(K)\rightarrow\mr{gr}
   (\mc{A}_x^\wedge\otimes K)
 \end{equation*}
 is surjective by \cite[p.58]{HO}, $\alpha$ is injective. This
 implies that the homomorphism (\ref{flatisom}) is strict by \cite[Ch.I,
 4.2.4 (2)]{HO}, and $\alpha$ is an isomorphism. Thus, by diagram
 chasing, $\beta$ is also an isomorphism, and
 \begin{equation*}
  \mr{gr}(\mc{A}_x)\otimes\mr{gr}(M_U)\cong\mr{gr}(\mc{A}_x^\wedge)
   \otimes\mr{gr}(M_U)=0.
 \end{equation*}
 Since $\mr{gr}(\mc{A}_x)\cong(\mr{gr}(\mc{A}))_x$, this shows that
 there exists $V\in\mf{B}$ such that $x\in V$, $V\subset U$, and
 \begin{equation*}
   \Gamma(V,\mr{gr}(\mc{A}))\otimes_{\Gamma(U,\mr{gr}(\mc{A}))}
    \mr{gr}(M_U)=0.
 \end{equation*}
 Let ${M}_{V}:=\Gamma(V,\mc{A})\otimes_{\Gamma(U,\mc{A})}M_U$, and
 equip it with the tensor filtration. Since the filtration on $M_U$ is good,
 the filtration on $M_V$ is also good by \cite[Ch.I, 6.14]{HO}. Since
 the canonical homomorphism
 \begin{equation*}
  \Gamma(V,\mr{gr}(\mc{A}))\otimes_{\Gamma(U,\mr{gr}(\mc{A}))}\mr{gr}(M_U)
   \cong\mr{gr}\bigl(\Gamma(V,\mc{A})\bigr)\otimes_{\mr{gr}
   (\Gamma(U,\mc{A}))}\mr{gr}(M_U)
   \rightarrow\mr{gr}(M_V)
 \end{equation*}
 is surjective, here the first isomorphism comes form condition 4, we have
 $\mr{gr}(M_V)=0$. Since $\Gamma(V,\mc{A})$ is complete and the
 filtration on $M_V$ is good, we obtain that $M_V=0$ (cf.\
 \ref{noethfiltdef}). Since $M\cong\mc{A}_x\otimes_{\mc{A}_V}M_V=0$, the
 fully faithfulness follows.

 Since $\mr{gr}(\mc{A}_x)$ is noetherian, $\mc{A}_x$ is Zariskian by
 \cite[Ch.II, 2.1.2 (4)]{HO}, and in particular
 $\mc{A}_x$ is noetherian. To show that $\mc{A}$ is noetherian,
 it remains to prove that $\mc{A}$ is coherent. For this, it suffices to
 check the conditions of \cite[3.1.1]{Ber1}.
 We have already checked (a). The flatness of the restriction
 $\Gamma(U,\mc{A})\rightarrow\Gamma(V,\mc{A})$ for open subsets
 $V\subset U$ in $\mf{B}$ follows by \cite[Ch.II, 1.2.1]{HO}, thus (b)
 is satisfied, and the lemma follows.
\end{proof}

\begin{lem}
 \label{tangenttwistshfi}
 We use the notation of {\normalfont \ref{subsecmicloc}} and
 {\normalfont{\ref{micshdfe}}}, and we further assume that $\mr{gr}(A)$ is
 a noetherian ring.

 (i) The rings $\mc{O}_V(0)$ and $\mc{O}_V(*)$ are noetherian. Moreover,
 $\mc{O}_{V}(n)$ is a coherent $\mc{O}_V(0)$-module on $\mathring{V}$
 for any integer $n$.

 (ii) The microlocalization $\mc{B}$ is noetherian and pointwise
 Zariskian on $V$. Moreover, $\varphi$ is flat.
\end{lem}
\begin{proof}
 Let us check (i). By \cite[II, 2.1.5]{EGA}, the ring $A_0$ is also
 noetherian, and $\mr{gr}(A)$ is of finite type over $A_0$. Then by the
 same argument as \cite[A.3.1.8]{Lau}, $\mc{O}_V(0)$ and $\mc{O}_V(*)$
 are noetherian. Now, since $\mc{O}_{P}(n)$ is a coherent
 $\mc{O}_P\cong\mc{O}_P(0)$-module and $\mc{O}_X(*)$ is of finite type
 over $\mc{O}_X(0)$, the second claim follows.

 Let us check (ii). Applying Lemma \ref{noetherianlemma} to the
 microlocalization $\mc{B}$ on $\mathring{V}$, $\mc{B}$ is a noetherian
 ring and pointwise Zariskian on $\mathring{V}$ by
 (\ref{gradedisomicdi}). Let us show $\mc{B}$ is in fact
 noetherian and pointwise Zariskian {\em on $V$}. It is pointwise
 Zariskian on the zero section by
 (\ref{noprimecalcfromprima}) (or more directly by
 \cite[A.3.1.5]{Lau}). To check condition
 (i)-2 of Definition \ref{defnoethrigsh}, we apply
 (\ref{zerosectionpullcallau}), and for (i)-3, we apply
 (\ref{noprimecalcfromprima}). It remains to show that $\mc{B}$ is
 coherent, which follows directly from \cite[3.1.1]{Ber1}.
 The flatness follows by \cite[A.3.1.7]{Lau}.
\end{proof}

\begin{lem}
\label{lemmonnoethfilt}
 Let $({A},\{{A}_i\}_{i\in\mb{Z}})$ be a filtered ring such that ${A}_0$
 is noetherian filtered, $\bigoplus_{i\geq0}\mr{gr}_i({A})$ is
 noetherian. Then ${A}$ is noetherian filtered.
\end{lem}
\begin{proof}
 By \cite[II, 2.1.5, 2.1.6]{EGA}, $\mr{gr}_i({A})$ is finitely generated
 over ${A}_0$ for any $i\in\mb{Z}$. Then the statement is nothing but
 \cite[Proposition 1.1.5]{KK} applying in the case where the topological
 space is just a point.
\end{proof}

\begin{lem}
 \label{Zarisep}
 Let $(\mc{A},\mc{A}_i)$ be a pointwise Zariskian filtered ring on a
 topological space $X$. Let $(\mc{M},\mc{M}_i)$ be a good filtered
 $(\mc{A},\mc{A}_i)$-module. Then the filtration $\{\mc{M}_i\}$ is
 separated ({\it i.e.}\ $\invlim_i\mc{M}_i=0$).
\end{lem}
\begin{proof}
 Since $\invlim_i\mc{M}_i\hookrightarrow\mc{M}$,
 we get the following commutative diagram for any $x\in X$.
 \begin{equation*}
  \xymatrix@C=30pt@R=15pt{
   (\invlim_i\mc{M}_i)_x\ar@{^{(}->}[r]\ar[dr]&\mc{M}_x\\
  &\invlim_i\mc{M}_{i,x}.\ar[u]
   }
 \end{equation*}
 Since $\mc{A}$ is pointwise Zariskian, $\invlim_i\mc{M}_{i,x}=0$, and
 thus, $\invlim_i\mc{M}=0$.
\end{proof}

\section{Microdifferential sheaves}
\label{micdiff}
We apply the results of the previous section to the theory of arithmetic
$\ms{D}$-modules, and define the rings of naive microdifferential
operators of finite level.

\subsection{}
\label{generalsetup}
Let $S$ be a scheme over $\mb{Z}_p$ (which may not be locally of finite
type). Let $X$ be a smooth scheme
over $S$, and let $m$ be a non-negative integer. Then we may consider
the sheaf of $S$-linear differential operators of
level $m$ denoted by $\Dmod{m}{X/S}$ on $X$. We often abbreviate this as
$\Dmod{m}{X}$. For the details on this
sheaf, we can refer to \cite{Ber1}, \cite{Ber2}, \cite{BerInt}. For
$i\in\mb{Z}$, let $\Dmod{m}{X,i}$ be the sub-$\mc{O}_X$-module
consisting of operators whose orders are less than or equal to $i$
in $\Dmod{m}{X}$ (cf.\ \cite[2.2.1]{Ber1}). By definition,
$\Dmod{m}{X,i}=0$ for $i<0$. Then
$\bigl\{\Dmod{m}{X,i}\bigr\}_{i\in\mb{Z}}$ is an increasing filtration of
$\Dmod{m}{X}$, which we call the {\em filtration by order}. By
\cite[2.2.4]{Ber1}, the ring $\mr{gr}(\Dmod{m}{X})$ is
commutative. Let\footnote{We warn the reader that the notation
$T^{(m)*}X$ is used in \cite{Ber1} for the associated {\em reduced
scheme} $\bigl(\mr{Spec}(\mr{gr}(\Dmod{m}{X}))\bigr)_{\mr{red}}$.
}
\begin{equation*}
 T^{(m)*}X:=\mr{Spec}(\mr{gr}(\Dmod{m}{X})),\qquad
 {P}^{(m)*}X:=\mr{Proj}(\mr{gr}(\Dmod{m}{X})).
\end{equation*}
We call these the {\em pseudo cotangent bundles of level $m$}.
When we need to emphasize the base, we denote $T^{(m)*}X$ by
$T^{(m)*}(X/S)$.
When $m=0$, we denote $T^{(m)*}X$ and $P^{(m)*}X$ by $T^*X$ and $P^*X$
respectively, which are nothing but the usual cotangent bundles of
$X$. Let $\mathring{T}^{(m)*}X:=T^{(m)*}X\setminus s(X)$ where $s\colon
X\rightarrow T^{(m)*}X$ denotes the zero section. Then there exist
the canonical morphisms (cf.\ \ref{subsecmicloc}) as follows:
\begin{equation*}
 \xymatrix@R=10pt@C=30pt{
 T^{(m)*}X\ar[rd]_<>(.5){\pi_m}&\mathring{T}^{(m)*}X
 \ar[r]|{q}\ar[l]\ar[d]&P^{(m)*}X\ar[ld]\\&X.&
  }
\end{equation*}
Recall the notation $\mc{O}_{T^{(m)*}X}(n)$ for $n\in\mb{Z}$ of
\ref{subsecmicloc} which is a subsheaf of
$\mc{O}_{T^{(m)*}X}(*)$ consisting of homogeneous elements of degree
$n$. There is a canonical isomorphism
$q^{-1}\mc{O}_{P^{(m)*}X}(n)\cong\mc{O}_{T^{(m)*}X}(n)$ on
$\mathring{T}^*X$ for any integer $n$ (cf.\ \ref{subsecmicloc}).
We remind that $\mc{O}_{T^{(m)*}X}$ does not coincide with
$\mc{O}_{T^{(m)*}X}(*)$. The following lemma is immediate from Lemma
\ref{tangenttwistshfi}.

\begin{lem*}
 The rings $\mc{O}_{T^{(m)*}X}(0)$, $\mc{O}_{T^{(m)*}X}(*)$
 are noetherian, and $\mc{O}_{\mathring{T}^{(m)*}X}(n)$ is a coherent
 $\mc{O}_{\mathring{T}^{(m)*}X}(0)$-module for any integer
 $n$. Moreover, $\mc{O}_{\mathring{T}^{(m)*}X}(*)$ is an
 $\mc{O}_{\mathring{T}^{(m)*}X}(0)$-algebra of finite type.
\end{lem*}

\subsection{}
\label{dfnnaivemicsch}
We can consider the microlocalization of
$(\Dmod{m}{X/S},\Dmod{m}{X/S,i})$ denoted by
$(\Emod{m}{X/S},\Emod{m}{X/S,i})$ using the technique of
\ref{subsecmicloc}. We often abbreviate this as
$(\Emod{m}{X},\Emod{m}{X,i})$. This is a filtered ring on $T^{(m)*}X$.
Then there exists a canonical homomorphism of filtered rings
\begin{equation*}
 \varphi_m\colon\pi_m^{-1}(\Dmod{m}{X},\Dmod{m}{X,i})\rightarrow
  (\Emod{m}{X},\Emod{m}{X,i}).
\end{equation*}
By (\ref{gradedisomicdi}), we have canonical isomorphisms
\begin{equation}
 \label{relagr}
 \mr{gr}_n(\Emod{m}{X})\cong\mc{O}_{T^{(m)*}X}(n),\qquad
  \mr{gr}(\Emod{m}{X})\cong\mc{O}_{T^{(m)*}X}(*).
\end{equation}
Since $\mr{gr}(\Dmod{m}{X})$ is a noetherian ring by
the proof of \cite[2.2.5]{Ber1}, $\Emod{m}{X}$ is pointwise Zariskian
and noetherian, and moreover $\varphi_m$ is flat by Lemma
\ref{tangenttwistshfi}. Since the canonical homomorphism
$\pi_m^{-1}\pi_{m*}\mc{O}_{T^{(m)*}X}(*)\rightarrow\mc{O}_{T^{(m)*}X}(*)$
is injective, $\mr{gr}(\varphi_m)$ is injective as well, and thus
$\varphi_m$ is strictly injective by \cite[Ch.I, 4.2.4 (2)]{HO}.

\begin{rem*}
 The $\pi^{-1}\mc{O}_X$-modules $\Emod{m}{X}$ and $\Emod{m}{X,i}$ {\em
 do not} possess $\mc{O}_{T^{(m)*}X}$-module structure (cf.\ Remark
 \ref{micshdfe}).
\end{rem*}

\begin{lem}
 \label{compatibility}
 We assume that $S$ and $X$ are affine, and $S=\mr{Spec}(A)$. Let
 $S':=\mr{Spec}(B)$ be an affine scheme {\em finite} over $S$. We put
 $X':=X\times_SS'$, and we have the base change isomorphism
 $T^{(m)*}(X'/S')\cong T^{(m)*}(X/S)\times_SS'$ (cf.\ {\em
 \cite[2.2.2]{Ber1}}). Let $f$ be a homogeneous section of
 $\Gamma(T^{(m)*}X,\mc{O}_{T^{(m)*}X})$, and $f'$ be the image in
 $\Gamma(T^{(m)*}X',\mc{O}_{T^{(m)*}X'})$. We put $U:=D(f)$ and
 $U':=D(f')$. Then there exists a canonical isomorphism of filtered
 rings
 \begin{equation*}
  \Gamma\bigl(U,(\Emod{m}{X/S},\Emod{m}{X/S,i})\bigr)\otimes_AB
   \xrightarrow{\sim}
   \Gamma\bigl(U',(\Emod{m}{X'/S'},\Emod{m}{X'/S',i})\bigr).
 \end{equation*}
\end{lem}
\begin{proof}
 We may assume that $\deg(f)>0$.
 By \cite[2.2.2]{Ber1}, there exists an isomorphism
 \begin{equation}
  \label{basechaisoD}
  \Gamma\bigl(X,(\Dmod{m}{X/S},\Dmod{m}{X/S,i})\bigr)\otimes_AB
   \xrightarrow{\sim}
   \Gamma\bigl(X',(\Dmod{m}{X'/S'},\Dmod{m}{X'/S',i})\bigr).
 \end{equation}
 We denote $\Gamma\bigl(U,(\Emod{m}{X},\Emod{m}{X,i})\bigr)$ by
 $(E_X,E_{X,i})$, and
 $\Gamma\bigl(X',(\Dmod{m}{X'},\Dmod{m}{X',i})\bigr)$ by
 $(D_{X'},D_{X',i})$. The isomorphism (\ref{basechaisoD}) induces a
 homomorphism of filtered rings
 $(D_{X'},D_{X',i})\rightarrow(E_X,E_{X,i})\otimes_AB$.
 Since $B$ is finite over $A$, $(E_X,E_{X,i})\otimes_AB$ is a complete
 filtered ring by \cite[Ch.II, 1.2.10 (5)]{HO}. By the
 universality \cite[Proposition A.2.3.3]{Lau}, $(E_X,E_{X,i})\otimes_AB$
 is the microlocalization of $(D_{X'},D_{X',i})$, and the lemma follows.
\end{proof}

\begin{rem*}
 Consider the general situation of \ref{generalsetup}, and
 let $S'\rightarrow S$ be a finite morphism. Put $X':=X\times_SS'$.
 We have the base change isomorphism $f'\colon T^{(m)*}(X'/S')
 \xrightarrow{\sim}T^{(m)*}(X/S)\times_SS'$, and using the lemma, we
 have the following isomorphism:
 \begin{equation*}
  f'^{-1}\bigl((\Emod{m}{X/S},\Emod{m}{X/S,i})
   \otimes_{S}\mc{O}_{S'}\bigr)\xrightarrow{\sim}
   (\Emod{m}{X'/S'},\Emod{m}{X'/S',i}).
 \end{equation*}
 Note, however, that since $\Emod{m}{X}$ and $\Emod{m}{X,i}$ are not
 quasi-coherent, {\it a priori}, the tensor product does not
 commute with global section functor over affine schemes. In
 this sense, the assertion of the lemma is slightly stronger than this
 global version.
\end{rem*}

\subsection{}
\label{limitnaive}
Now, we pass to the limit. Let $R$ be a complete discrete valuation
ring of mixed characteristic $(0,p)$ whose residue field is denoted by
$k$. We denote the field of fractions by $K$, and let $\pi$ be a
uniformizer of $R$. For a non-negative integer $i$, we put
$R_i:=R/(\pi^{i+1})$. From now on, we use these notation freely without
referring to this subsection.

Let $\ms{X}$ be a smooth formal scheme over
$R$. We denote by $X_i$ the reduction of $\ms{X}$ over $R_i$.
We define $T^{(m)*}\ms{X}$ and $P^{(m)*}\ms{X}$ by the limit of
$T^{(m)*}X_i$ and $P^{(m)*}X_i$ over $i$ respectively. We also define
$\mc{O}_{T^{(m)*}\ms{X}}(*)$ (resp.\ $\mc{O}_{T^{(m)*}\ms{X}}(n)$) to be
the limit of $\mc{O}_{T^{(m)*}X_i}(*)$ (resp.\
$\mc{O}_{T^{(m)*}X_i}(n)$) over $i$, and put
$\mc{O}_{T^{(m)*}\ms{X},\mb{Q}}(*)$
(resp.\ $\mc{O}_{T^{(m)*}\ms{X},\mb{Q}}(n)$) to be
$\mc{O}_{T^{(m)*}\ms{X}}(*)\otimes\mb{Q}$ (resp.\
$\mc{O}_{T^{(m)*}\ms{X}}(n)\otimes\mb{Q}$). Let
$\epsilon_\cdot\colon\mf{O}(V)\rightarrow\mf{O}(V')$ be the
functor in \ref{subsecmicloc} where $V=T^{(m)*}\ms{X}$.

\begin{dfn*}
 (i) Let $\mf{B}'$ be the open basis of $V'$ consisting of $D(f)$ in
 $T^{(m)*}\ms{X}$ over an affine open subscheme $\ms{U}$ of $\ms{X}$
 where $f$ is a homogeneous element of
 $\Gamma(T^{(m)*}\ms{U},\mc{O}_{T^{(m)*}\ms{U}})$. An open subset
 $\ms{U}\subset\mathring{T}^{(m)*}\ms{U}$ is said to be {\em strictly
 affine} if $\ms{U}\in\mf{B}'$.

 (ii) We define an open basis $\mf{B}$ of $V$ to be the set consisting
 of $U\in\mf{O}(V)$ such that $\epsilon_\cdot(U)\in\mf{B}'$.
\end{dfn*}

\begin{lem*}
 The rings $\mc{O}_{T^{(m)*}\ms{X}}(0)$ and
 $\mc{O}_{T^{(m)*}\ms{X}}(*)$ are noetherian with respect to $\mf{B}$,
 and $\mc{O}_{\mathring{T}^{(m)*}\ms{X}}(n)$ is a coherent
 $\mc{O}_{\mathring{T}^{(m)*}\ms{X}}(0)$-module for any integer
 $n$. Moreover, $\mc{O}_{\mathring{T}^{(m)*}\ms{X}}(*)$ is an
 $\mc{O}_{\mathring{T}^{(m)*}\ms{X}}(0)$-algebra of finite type.
\end{lem*}
\begin{proof}
 The proof is the same as Lemma \ref{tangenttwistshfi}, so we only
 sketch here. We put $\ms{V}:=T^{(m)*}\ms{X}$ and
 $\mathring{\ms{V}}:=\mathring{T}^{(m)*}\ms{X}$. To check that
 $\mc{O}_{\mathring{\ms{V}}}(n)$ is a coherent
 $\mc{O}_{\mathring{\ms{V}}}(0)$-module, it suffices to point out that
 $\mc{O}_{P^{(m)*}\ms{X}}(n)$ is a coherent
 $\mc{O}_{P^{(m)*}\ms{X}}\cong\mc{O}_{P^{(m)*}\ms{X}}(0)$-module. The
 proof that $\mc{O}_{\mathring{\ms{V}}}(*)$ is of
 finite type is the same. It remains to show that $\mc{O}_{\ms{V}}(0)$
 is noetherian. The only thing we need to check is the coherence of
 $\mc{O}_{\ms{V}}(0)$ and $\mc{O}_{\ms{V}}(*)$ around the zero section,
 and for this, apply \cite[3.1.1]{Ber1} as Lemma
 \ref{tangenttwistshfi}.
\end{proof}

We define a sheaf of rings on the topological space
$T^{(m)*}\ms{X}\approx T^{(m)*}X_0$ ($\approx$ denotes the canonical
homeomorphism of topological spaces) by
\begin{equation*}
 \Ecomp{m}{\ms{X}}:=\invlim_i\Emod{m}{X_i}.
\end{equation*}
For $j\in\mb{Z}$, we also define
\begin{equation*}
 \Emod{m}{\ms{X},j}:=\invlim_i\Emod{m}{X_i,j}.
\end{equation*}
We remark that the ``filtration'' $\Emod{m}{\ms{X},j}$ of
$\Ecomp{m}{\ms{X}}$ is {\it not exhaustive}. We define a submodule
(which is in fact a {\em ring} by Lemma \ref{veryelementaryprop} (iii)
below) by
\begin{equation*}
 \Emod{m}{\ms{X}}:=\indlim_{j\rightarrow\infty}
  \Emod{m}{\ms{X},j}~\subset
  \Ecomp{m}{\ms{X}}.
\end{equation*}
There is a canonical homomorphism of rings on $T^{(m)*}\ms{X}$
\begin{equation}
 \label{canhomna}
 \widehat{\varphi}_m\colon\pi_m^{-1}\Dcomp{m}{\ms{X}}
 \rightarrow\Ecomp{m}{\ms{X}}.
\end{equation}
This homomorphism is injective by the injectivity of $\varphi_m$ in
\ref{dfnnaivemicsch}. Since $\widehat{\varphi}_m(\pi^{-1}_m
\Dmod{m}{\ms{X},n})\subset\Emod{m}{\ms{X},n}$,
$\widehat{\varphi}_m|_{\pi^{-1}_m\Dmod{m}{\ms{X}}}$ induces a
homomorphism of modules
$\pi_m^{-1}\Dmod{m}{\ms{X}}\rightarrow\Emod{m}{\ms{X}}$. We abusively
denote this homomorphism by $\varphi_m$. We see from the following Lemma
\ref{veryelementaryprop} (iii) that this homomorphism is in fact a
homomorphism of {\em rings}.

\begin{lem}
 \label{veryelementaryprop}
 Let $\ms{X}$ be a smooth formal scheme over $R$. Let $\ms{U}$ be an
 open formal subscheme of $T^{(m)*}\ms{X}$ belonging to $\mf{B}$. Let
 $i$ be a non-negative integer, and we denote $\ms{U}\otimes R_i$ by
 $U_i$.

 (i) The ring $\Gamma(\ms{U},\Ecomp{m}{\ms{X}})$ is $\pi$-adically
 complete and flat over $R$.
 Moreover, the canonical homomorphisms $\Ecomp{m}{\ms{X}}\otimes
 R_i\rightarrow\Emod{m}{X_i}$ and $\Emod{m}{\ms{X}}\otimes
 R_i\rightarrow\Emod{m}{X_i}$ are isomorphisms.

 (ii) Let $j$ be an integer, and $k$ be a positive integer. Let $\ms{E}$
 be one of $\Emod{m}{\ms{X},j+k}$,
 $\Emod{m}{\ms{X}}$, $\Ecomp{m}{\ms{X}}$. We have
 \begin{equation*}
  \Gamma(\ms{U},\ms{E}/\Emod{m}{\ms{X},j})\cong
   \Gamma(\ms{U},\ms{E})/
  \Gamma(\ms{U},\Emod{m}{\ms{X},j}),\quad
  \Emod{m}{\ms{X},j+k}/\Emod{m}{\ms{X},j}\cong\invlim_i
  \Emod{m}{X_i,j+k}/\Emod{m}{X_i,j}.
 \end{equation*}

 (iii) Let $j$ and $k$ be integers. Then
 $\Emod{m}{\ms{X},j}\cdot\Emod{m}{\ms{X},k}\subset\Emod{m}{\ms{X},j+k}$
 in $\Ecomp{m}{\ms{X}}$, and in particular,
 $(\Emod{m}{\ms{X}},\{\Emod{m}{\ms{X},j}\}_{j\in\mb{Z}})$ is a {\em
 filtered ring}. Moreover, the $\pi$-adic completion of
 $\Emod{m}{\ms{X}}$ is isomorphic to $\Ecomp{m}{\ms{X}}$.

 (iv) The filtered rings $\Emod{m}{X_i}$ and $\Emod{m}{\ms{X}}$ are
 complete with respect to the filtration by order.
\end{lem}
\begin{proof}
 For a projective system $\{\mc{F}_i\}_{i\geq0}$ on a topological space
 $T$ and for an open subset $U$ of $T$,
 \begin{equation}
  \label{commprojga}
  \Gamma(U,\invlim_i\mc{F}_i)\xrightarrow{\sim}
  \invlim_i\Gamma(U,\mc{F}_i)
 \end{equation}
 by \cite[$0_\mr{I}$, 3.2.6]{EGA}. For an inductive system
 $\{\mc{F}_i\}_{i\geq0}$ on a {\em noetherian} topological space $T$ and
 for an open subset $U$ of $T$,
 \begin{equation}
  \label{commindga}
  \indlim_i\Gamma(U,\mc{F}_i)\xrightarrow{\sim}
  \Gamma(U,\indlim_i\mc{F}_i)
 \end{equation}
 by \cite[Ch.II, 3.10]{Go}.
 Since $\ms{U}$ is an open subset of an affine
 formal scheme $\epsilon_\cdot(\ms{U})$, $\ms{U}$ is a noetherian
 space.

 By (\ref{commprojga}) and the definition of
 $\Ecomp{m}{\ms{X}}$ in \ref{limitnaive},
 $\Gamma(\ms{U},\Ecomp{m}{\ms{X}})\cong\invlim_{i}\Gamma
 (U_i,\Emod{m}{X_i})$.  Since
 $\Gamma(U_i,\Emod{m}{X_i})$ is flat over $\Gamma(U_i,\Dmod{m}{X_i})$
 (cf.\ \ref{dfnnaivemicsch}),
 the ring $\Gamma(U_i,\Emod{m}{X_i})$ is flat over $R_i$. Thus we get
 first two claims of (i) by Lemma \ref{compatibility} and the following
 Lemma \ref{completion}. For $\Emod{m}{\ms{X}}\otimes
 R_i\xrightarrow{\sim}\Emod{m}{X_i}$, we show
 $\Emod{m}{\ms{X},j}\otimes R_i\xrightarrow{\sim}\Emod{m}{X_i,j}$ for
 any $j\in\mb{Z}$ by the same argument, and take inductive limit over
 $j$.

 Let us prove (ii) for the $\ms{E}=\Emod{m}{\ms{X},j+k}$ case.
 By (\ref{quoticalclau}),
 \begin{equation*}
  \Gamma(U_i,\Emod{m}{X_i,j+k}/\Emod{m}{X_i,j})\cong
   \Gamma(U_i,\Emod{m}{X_i,j+k})/\Gamma(U_i,\Emod{m}{X_i,j}).
 \end{equation*}
 Since the projective system
 $\bigl\{\Gamma(U_i,\Emod{m}{X_i,j})\bigr\}_{i\geq0}$ satisfies the
 Mittag-Leffler condition by (i), the sequence
 \begin{equation*}
  0\rightarrow\invlim_i\Gamma(U_i,\Emod{m}{X_i,j})\rightarrow
   \invlim_i\Gamma(U_i,\Emod{m}{X_i,j+k})\rightarrow
   \invlim_i\Gamma(U_i,\Emod{m}{X_i,j+k}/\Emod{m}{X_i,j})
   \rightarrow 0
 \end{equation*}
 is exact. Considering (\ref{commprojga}), this shows that
 \begin{equation}
  \label{glbsect(ii)}
  \Gamma(\ms{U},\invlim_i\Emod{m}{X_i,j+k}/\Emod{m}{X_i,j})\cong
   \Gamma(\ms{U},\Emod{m}{\ms{X},j+k})/
   \Gamma(\ms{U},\Emod{m}{\ms{X},j}).
 \end{equation}
 Thus, since $\mf{B}$ is a basis of the topology, the canonical
 homomorphism
 $\Emod{m}{\ms{X},j+k}/\Emod{m}{\ms{X},j}\rightarrow\invlim_i
 \Emod{m}{X_i,j+k}/\Emod{m}{X_i,j}$ is an isomorphism, and the second
 equality of (ii) follows.
 The first equality of (ii) follows by using (\ref{glbsect(ii)}) once
 again. To deal with the $\ms{E}=\Ecomp{m}{\ms{X}}$ case, just replace
 $\Emod{m}{X_i,j+k}$ (resp.\ $\Emod{m}{\ms{X},j+k}$) by $\Emod{m}{X_i}$
 (resp.\ $\Ecomp{m}{\ms{X}}$) in the argument above.
 For the $\ms{E}=\Emod{m}{\ms{X}}$ case, use (\ref{commindga})
 and the $\ms{E}=\Emod{m}{\ms{X},j+k}$ case.

 The first claim of (iii) follows since $\Emod{m}{X_i}$ is a filtered
 ring. By (i), $\Ecomp{m}{\ms{X}}$ is the $\pi$-adic completion of
 $\Emod{m}{\ms{X}}$.

 Let us prove (iv). The completeness of $\Emod{m}{X_i}$ follows by
 definition.
 Let us see the completeness of $\Emod{m}{\ms{X}}$. For an open affine
 subscheme $\ms{U}$ in $\mf{B}$, consider the following exact sequence
 \begin{equation*}
  0\rightarrow\Gamma(\ms{U},\Emod{m}{\ms{X},k}/\Emod{m}{\ms{X},j})
   \rightarrow\Gamma(\ms{U},\Ecomp{m}{\ms{X}}/\Emod{m}{\ms{X},j})
   \rightarrow\Gamma(\ms{U},\Ecomp{m}{\ms{X}}/\Emod{m}{\ms{X},k})
   \rightarrow 0
 \end{equation*}
 for integers $k\geq j$. The last surjection is deduced by using (ii).
 Since the projective system
 $\bigl\{\Gamma(\ms{U},\Emod{m}{\ms{X}}/\Emod{m}{\ms{X},j})\bigr\}_{j}$
 satisfies the Mittag-Leffler condition by (ii), the following sequence
 is exact:
 \begin{equation*}
  \xymatrix{
  0\ar[r]&\invlim_{j}\indlim_{k}
   \Emod{m}{\ms{X},k}/\Emod{m}{\ms{X},j}\ar[r]\ar[d]^{\sim}&
   \invlim_{j}\indlim_{k}\Ecomp{m}{\ms{X}}/\Emod{m}{\ms{X},j}\ar[r]
   \ar[d]^{}&\invlim_{j}\indlim_{k}
   \Ecomp{m}{\ms{X}}/\Emod{m}{\ms{X},k}\ar[r]\ar[d]^{\sim}&0\\
  &\invlim_{j}\Emod{m}{\ms{X}}/\Emod{m}{\ms{X},j}&
   \Ecomp{m}{\ms{X}}&\Ecomp{m}{\ms{X}}/\Emod{m}{\ms{X}}&
   }
 \end{equation*}
 where $j\rightarrow -\infty$ and $k\rightarrow\infty$. The middle
 vertical homomorphism is an isomorphism as well by the commutativity
 (\ref{commprojga}) and the fact that two projective limits commute.
 Thus the lemma is proven.
\end{proof}

\begin{lem}
 \label{completion}
 Let $\{E_i\}_{i\geq 0}$ be a projective system of $R$-modules such that
 for each $i$, $E_i$ is a flat $R_i$-module. Assume that the homomorphism
 $E_{i+1}\otimes R_i\rightarrow E_i$ induced by the transition
 homomorphism is an isomorphism for any non-negative integer $i$. Let
 $E:=\invlim_i E_i$. Then the canonical homomorphism $E\otimes
 R_i\rightarrow E_i$ is an isomorphism for any non-negative integer
 $j$. Moreover, $E$ is $\pi$-adically complete and flat over $R$.
\end{lem}
\begin{proof}
 We leave the proof to the reader.
\end{proof}

\begin{lem}
 \label{noethgrlem}
 Let $\mathring{\ms{V}}:=\mathring{T}^{(m)*}\ms{X}$.
 Let $\mc{A}=\bigoplus_{i\in\mb{Z}}\mc{A}_i$ be a graded
 $\mc{O}_{\mathring{\ms{V}}}(0)$-algebra of finite type on
 $\mathring{\ms{V}}$ such that $\mc{A}_i$ is a coherent
 $\mc{O}_{\mathring{\ms{V}}}(0)$-module for any $i\in\mb{Z}$. Then for
 any $V\subset U$ in $\mf{B}$, the restriction homomorphism
 $\Gamma(U,\mc{A})\rightarrow\Gamma(V,\mc{A})$ is flat, and
 $\mc{A}$ is noetherian with respect to $\mf{B}$.
\end{lem}
\begin{proof}
 We put $\mc{O}:=\mc{O}_{\mathring{\ms{V}}}(0)$. Let us check the
 conditions of Definition \ref{defnoethrigsh} (i).
 Condition 2 follows since $\mc{A}$ is of finite type over $\mc{O}$. Let
 $U$ be an open subset of $\mathring{\ms{V}}$ in $\mf{B}$ such that
 there exists a surjection
 $\phi\colon\mc{O}[T_1,\dots,T_n]|_U\rightarrow\mc{A}|_U$. We claim that
 the homomorphism
 \begin{equation*}
  \Gamma(U,\mc{O}[T_1,\dots,T_n])\rightarrow\Gamma(U,\mc{A})
 \end{equation*}
 is surjective. Indeed, since $\mc{A}_i$ is a coherent $\mc{O}$-module
 for any $i$, $\mr{Ker}(\phi)$ is an inductive limit of coherent
 $\mc{O}|_U$-modules. Since $U$ is noetherian and separated,
 $H^1(U,-)$ commutes with inductive limit by \cite[Ch.II,
 4.12.1]{Go}, and we have $H^1(U,\mr{Ker}(\phi))=0$,
 which implies the claim. Thus condition 3 is fulfilled. It remains
 to show that $\mc{A}$ is a coherent ring. For this, it suffices to
 check the conditions of \cite[3.1.1]{Ber1}. For $V\subset U$ in
 $\mf{B}$, we have the restriction isomorphism
 $\Gamma(V,\mc{O})\otimes_{\Gamma(U,\mc{O})}\Gamma(U,\mc{A}_i)
 \xrightarrow{\sim}\Gamma(V,\mc{A}_i)$ for any $i$ since $\mc{A}_i$ is a
 coherent $\mc{O}$-module using (\ref{noprimecalcfromprima}). This
 induces an isomorphism
 \begin{equation*}
  \Gamma(V,\mc{O})\otimes_{\Gamma(U,\mc{O})}\Gamma(U,\mc{A})
   \xrightarrow{\sim}\Gamma(V,\mc{A}).
 \end{equation*}
 Since the restriction homomorphism
 $\Gamma(U,\mc{O})\rightarrow\Gamma(V,\mc{O})$ is flat, this isomorphism
 shows that $\Gamma(U,\mc{A})\rightarrow\Gamma(V,\mc{A})$ is flat
 as well. Thus the claim follows.
\end{proof}

\begin{prop}
\label{noetheriansheafofmic}
 Let $\ms{X}$ be a smooth formal scheme over $R$.

 (i) The rings $\Emod{m}{\ms{X}}$, $\Ecomp{m}{\ms{X}}$,
 $\Emod{m}{\ms{X},0}$ are noetherian with respect to $\mf{B}$.

 (ii) The homomorphism $\widehat{\varphi}_m$ of
 {\normalfont(\ref{canhomna})} is flat.

 (iii) Let $\ms{E}$ be either $\Emod{m}{\ms{X},0}$ or $\Emod{m}{\ms{X}}$
 or $\Ecomp{m}{\ms{X}}$. For any open subsets $\ms{U}\supset\ms{V}$
 in $\mf{B}$, the restriction homomorphism
 $\Gamma(\ms{U},\ms{E})\rightarrow\Gamma(\ms{V},\ms{E})$
 is flat.
\end{prop}
\begin{proof}
 Let us prove (i). First, we show the claim for
 $\Emod{m}{\ms{X}}$ and $\Emod{m}{\ms{X},0}$. Let us check the
 conditions of Lemma \ref{noetherianlemma} for $\Emod{m}{\ms{X}}$
 (resp.\ $\Emod{m}{\ms{X},0}$) on
 $\mathring{\ms{V}}:=\mathring{T}^{(m)*}\ms{X}$. Conditions 1 and 4
 hold by Lemma \ref{veryelementaryprop}. By Lemma
 \ref{veryelementaryprop} (ii),
 $\mr{gr}(\Emod{m}{\ms{X}})\cong\mc{O}_{\mathring{\ms{V}}}(*)$ as graded
 rings. By Lemma \ref{limitnaive}, this implies that
 $\mr{gr}_i(\Emod{m}{\ms{X}})$ is a coherent
 $\mc{O}_{\mathring{\ms{V}}}(0)$-module on $\mathring{\ms{V}}$ for
 any $i\in\mb{Z}$, and $\mr{gr}(\Emod{m}{\ms{X}})$ (resp.\
 $\mr{gr}(\Emod{m}{\ms{X},0})$) is an
 $\mc{O}_{\mathring{\ms{V}}}(0)$-module of finite type on
 $\mathring{\ms{V}}$. Thus by Lemma \ref{noethgrlem}, conditions 2
 and 3 are fulfilled.
 This implies that $\Emod{m}{\ms{X}}$ and $\Emod{m}{\ms{X},0}$ are
 noetherian with respect to $\mf{B}$ on $\mathring{\ms{V}}$.
 Using \cite[Ch.II, 1.2.1]{HO}, $\widehat{\varphi}_m$ is flat, and (ii)
 follows. It remains
 to check that the rings are noetherian
 around the zero section. By using (\ref{noprimecalcfromprima}), we
 only need to prove the coherence. This follows from
 \cite[3.1.1]{Ber1}.

 For $\Ecomp{m}{\ms{X}}$, let us endow with the $\pi$-adic filtration
 $\bigl\{\pi^{-i}\Ecomp{m}{\ms{X}}\bigr\}_{i\leq0}$ (cf.\
 \ref{adicfiltex}). Since $\Ecomp{m}{\ms{X}}$ is $\pi$-torsion free
 by Lemma \ref{veryelementaryprop} (i), the homomorphism
 $\Emod{m}{X_0}[T]\rightarrow\mr{gr}(\Ecomp{m}{\ms{X}})$ sending $T$ to
 $\pi\in\mr{gr}_1(\Ecomp{m}{\ms{X}})$ is an isomorphism. It is
 straightforward to check the conditions of Lemma
 \ref{noetherianlemma}. We remind that the $\pi$-adic filtration can
 also be used to show that $\Emod{m}{\ms{X},0}$ is noetherian.

 To prove (iii), it suffices to apply (i) and \cite[Ch.II, 1.2.1]{HO}.
\end{proof}

\begin{rem*}
 By the proof, we can moreover say that $\Emod{m}{\ms{X}}$ and
 $\Emod{m}{\ms{X},0}$ are pointwise Zariskian with respect to the
 filtration by order on $\mathring{T}^{(m)*}\ms{X}$, and
 $\Ecomp{m}{\ms{X}}$ and $\Emod{m}{\ms{X},0}$ are pointwise
 Zariskian with respect to the $\pi$-adic filtration on
 $\mathring{T}^{(m)*}\ms{X}$.
\end{rem*}

\subsection{}
\label{defofnaivemic}
Now, we define
\begin{equation*}
 \EcompQ{m}{\ms{X}}:=\Ecomp{m}{\ms{X}}\otimes\mb{Q},\qquad
  \Emod{m}{\ms{X},\mb{Q}}:=\Emod{m}{\ms{X}}\otimes\mb{Q}.
\end{equation*}
Note that $\otimes\mb{Q}$ commutes with global section functor over
noetherian space by \cite[3.4]{Ber1}.
The homomorphism $\widehat{\varphi}_m$ of (\ref{canhomna}) induces a
canonical injective homomorphism
\begin{equation*}
 \widehat{\varphi}_m\otimes\mb{Q}\colon\pi_m^{-1}
  \DcompQ{m}{\ms{X}}\rightarrow\EcompQ{m}{\ms{X}}.
\end{equation*}
If there is no risk of confusion, we sometimes denote
$\widehat{\varphi}_m\otimes\mb{Q}$ abusively by $\widehat{\varphi}_m$.
We call the sheaves $\Emod{m}{X_i}$, $\Ecomp{m}{\ms{X}}$,
$\EcompQ{m}{\ms{X}}$ the rings of {\em naive microdifferential
operators of level $m$}. Proposition \ref{noetheriansheafofmic} implies
the following.

\begin{cor*}
 The rings $\Emod{m}{\ms{X},\mb{Q}}$ and $\EcompQ{m}{\ms{X}}$ are
 noetherian with respect to $\mf{B}$. Moreover,
 $\widehat{\varphi}_m\otimes\mb{Q}$
 and the restriction homomorphism $\Gamma(\ms{U},\ms{E})\rightarrow
 \Gamma(\ms{V},\ms{E})$ are flat for $\ms{U}\supset\ms{V}$ in $\mf{B}$,
 where $\ms{E}$ is either $\Emod{m}{\ms{X},\mb{Q}}$ or
 $\EcompQ{m}{\ms{X}}$.
\end{cor*}

\subsection{}
\label{defofsituation}
Let us describe sections of rings of microdifferential operators
explicitly. We use the notation of \ref{limitnaive}. Suppose in addition
that $\ms{X}$ is {\em affine}, and possesses a system of local
coordinates $\{x_1,\dots,x_d\}$. Let
$\bigl\{\partial_1,\dots,\partial_d\bigr\}$ be the corresponding
differential operators, and $\bigl\{\xi_1,\dots,\xi_d\bigr\}$ be the
corresponding basis of $\Gamma(T^*\ms{X},\mc{O}_{T^*\ms{X}})$. Let $k$
be a positive integer. We have a
differential operator $\partial_i^{\angles{m}{k}}$ for any $1\leq i\leq
d$ in $\Dmod{m}{X_l}$ for any integer $l\geq0$ or in
$\Dcomp{m}{\ms{X}}$ (cf.\ \cite[2.2.3]{Ber1}).
Write $k=p^m\,q+r$ with $0\leq r<p^m$. Recall that there is a relation
(cf.\ \cite[(2.2.3.1)]{Ber1})
\begin{equation*}
 k!\,\partial_i^{\angles{m}{k}}=q!\,\partial^k_i.
\end{equation*}
Now, these operators define elements in $\mr{gr}(\Dmod{m}{X_l})$ by
taking the principal symbol (cf.\ \ref{princisymbdfn}). We denote
$\sigma(\partial_i^{\angles{m}{k}})$ by $\xi_i^{\angles{m}{k}}$ in
$\mr{gr}_k(\Dmod{m}{X_l})\subset\mr{gr}(\Dmod{m}{X_l})$. From now on,
we use the multi-index notation. For example, for
$\underline{k}=(k_1,\dots,k_d)\in\mb{N}^d$, we denote
$\underline{\xi}^{\angles{m}{\underline{k}}}:=
\xi_1^{\angles{m}{k_1}}\dots\xi_d^{\angles{m}{k_d}}$,
$\underline{\partial}^{\angles{m}{\underline{k}}}:=\partial_1
^{\angles{m}{k_1}}\dots\partial_d^{\angles{m}{k_d}}$, and
$|\underline{k}|:=k_1+\dots+k_d$. We denote by
$\underline{k}\geq\underline{k'}$ if $k_i\geq k'_i$ for any $1\leq i\leq
d$. For $n\in\mb{Z}$, we sometimes denote $(n,\dots,n)$ by
$\underline{n}$ if there is no possible confusion.

Let
\begin{equation*}
 \Theta\in\Gamma(T^*\ms{X},\mc{O}_{T^*\ms{X}})
\end{equation*}
be a non-zero homogeneous section of degree $n$. We consider
$\pi^{-1}_m\mc{O}_{\ms{X}}$ as a subring of $\mc{O}_{T^{(m)*}\ms{X}}$.
This section induces a section
$\Theta^{(m)}\in\Gamma(T^{(m)*}\ms{X},\mc{O}_{T^{(m)*}\ms{X}})$ for any
integer $m\geq0$ as follows: we may write
\begin{equation}
 \label{desklocal}
  \Theta=\sum_{|\underline{k}|=n}a_{\underline{k}}\,
  \underline{\xi}^{\underline{k}}
\end{equation}
where $\underline{k}\in\mb{N}^d$ and
$a_{\underline{k}}\in\Gamma(\ms{X},\mc{O}_{\ms{X}})$ in a unique way.
We put
\begin{equation*}
 \Theta^{(m)}:=\sum_{|\underline{k}|=n}a^{p^m}_{\underline{k}}
  \bigl(\underline{\xi}^{\angles{m}{p^m}}\bigr)^{\underline{k}}.
\end{equation*}
The homogeneous element $\Theta$ induces also elements in
$\Dmod{m}{\ms{X}}$ or $\Dmod{m}{X_l}$. We put
\begin{equation*}
 \widetilde{\Theta}^{(m)}_{\mr{le}}:=
  \sum_{|\underline{k}|=n}a^{p^m}_{\underline{k}}
  \bigl(\underline{\partial}^{\angles{m}{p^m}}\bigr)
  ^{\underline{k}}
\end{equation*}
where the subscript ``$\mr{le}$'' stands for ``left''. Since
$\Dmod{m}{\ms{X},\mb{Q}}\cong\Dmod{m'}{\ms{X},\mb{Q}}$ for any
non-negative integer $m'$, we sometimes consider these operators as
sections of $\Dmod{m'}{\ms{X},\mb{Q}}$.

Let $\ms{U}$ be the open affine subset of $T^{(m)*}\ms{X}$ defined by
$\Theta^{(m)}$.
We claim that the operator $\widetilde{\Theta}_{\mr{le}}^{(m)}$
is invertible is $\Gamma(\ms{U},\Emod{m}{\ms{X}})$. Indeed,
the inverse of $\widetilde{\Theta}_{\mr{le}}^{(m)}$ in $\Emod{m}{X_l}$
has degree $-np^m$ for any $l$. Since the inverse of
$\widetilde{\Theta}^{(m)}_{\mr{le}}$ is unique in $\Emod{m}{X_l}$, these
elements induce an element of
$\invlim_l\Emod{m}{X_l,-np^m}=\Emod{m}{\ms{X},-np^m}
\subset\Emod{m}{\ms{X}}$.

Let $\{b_{\underline{k},i}\}$
be a sequence in $\Gamma(\ms{X},\mc{O}_{\ms{X}})$ for
$\underline{k}\in\mb{N}^d$ and $i\in\mb{N}$ such that the following
holds: for each integer $N$, let
$\beta_{N,i}:=\sup_{|\underline{k}|=inp^m+N}|b_{\underline{k},i}|$,
where $|\cdot|$ denotes the spectral norm (cf.\ \cite[2.4.2]{Ber1}) on
$\Gamma(\ms{X},\mc{O}_{\ms{X}})$. Then
 \begin{equation}
  \label{conditiononconvergence}
   \lim_{i\rightarrow\infty}\beta_{N,i}=0,\qquad
   \lim_{N\rightarrow+\infty}\sup_i\bigl\{\beta_{N,i}\bigr\}=0.
 \end{equation}
In the sequel, we consider the ring $\pi_m^{-1}\Dmod{m}{X_l}$ (resp.\
$\pi_m^{-1}\Dcomp{m}{\ms{X}}$) as a subring of $\Emod{m}{X_l}$ (resp.\
$\Ecomp{m}{\ms{X}}$) for any $l$ by $\varphi_m$ of
\ref{dfnnaivemicsch} (resp.\ $\widehat{\varphi}_m$ of
(\ref{canhomna})). The sum
\begin{equation}
 \label{sum}
  \sum_{N\in\mb{Z}}\sum_{|\underline{k}|-inp^m=N}
  b_{\underline{k},i}\,\underline{\partial}
  ^{\angles{m}{\underline{k}}}\,(\widetilde{\Theta}^{(m)}_\mr{le})^{-i}
\end{equation}
converges in $\Gamma(U_l,\Emod{m}{X_l})$ for any $l$. We note that the
order of $\underline{\partial}^{\angles{m}{\underline{k}}}\,
(\widetilde{\Theta}^{(m)}_{\mr{le}})^{-i}$ is $N$,
$\sum_{|\underline{k}|-inp^m=N}...$ is a finite sum by the first
condition of (\ref{conditiononconvergence}), and
$\sum_{|\underline{k}|-inp^m=N}...=0$ for $N\gg0$ by the second
condition. Since these elements form a projective system over $l$, we
have an element in $\Gamma(\ms{U},\Ecomp{m}{\ms{X}})$. Even though
the sum of (\ref{sum}) does not converge in
$\Gamma(\ms{U},\Ecomp{m}{\ms{X}})$ with respect to the $\pi$-adic
topology in general\footnote{
However, we are able to define a reasonable weaker topology on
$\Gamma(\ms{U},\Emod{m}{\ms{X}})$ such that this sum
converges. In the curve case, see \cite[1.2.2]{AM}.
}, we abusively denote by (\ref{sum}) the operator in
$\Gamma(\ms{U},\Ecomp{m}{\ms{X}})$.

\begin{lem*}
 For any element $P\in\Gamma(\ms{U},\Ecomp{m}{\ms{X}})$, there exists
 a sequence $\{b_{\underline{k},i}\}$ for $\underline{k}\in\mb{N}^d$ and
 $i\in\mb{N}$ satisfying
 {\normalfont(\ref{conditiononconvergence})} such that $P$ can be
 written as {\normalfont(\ref{sum})}. Moreover, if
 $P\in\Gamma(\ms{U},\Emod{m}{\ms{X},j})$ for some integer $j$, we can
 take $b_{\underline{k},i}=0$ for
 $|\underline{k}|-inp^m>j$.

 This is called a {\em left presentation of level $m$}. We remind here
 that presentation is not unique.
\end{lem*}
\begin{proof}
 Since $\Gamma(\ms{U},\Emod{m}{\ms{X},j})$ and
 $\Gamma(\ms{U},\Ecomp{m}{\ms{X}})$ are flat over $R$ and
 $\pi$-adically complete by Lemma \ref{veryelementaryprop} (i), it
 suffices to show that any element of $\Gamma(\ms{U},\Emod{m}{X_0,j})$
 can be written as
 \begin{equation*}
  \sum_{N\in\mb{Z}}\sum_{|\underline{k}|-inp^m=N}
   c_{\underline{k},i}\,\underline{\partial}
   ^{\angles{m}{\underline{k}}}\,(\widetilde{\Theta}^{(m)}
   _{\mr{le}})^{-i},
 \end{equation*}
 with $c_{\underline{k},i}\in\Gamma(X_0,\mc{O}_{X_0})$ such that
 the following holds: for each integer $N$, $c_{\underline{k},i}=0$ for
 almost all couples $(\underline{k},i)\in\mb{N}^d\times\mb{N}$ such that
 $|\underline{k}|-inp^m=N$, and $c_{\underline{k},i}=0$ for any
 $|\underline{k}|-inp^m>j$. This follows from
 \ref{microlocaldef}.
\end{proof}

\begin{rem*}
 Instead of using $\widetilde{\Theta}^{(m)}_{\mr{le}}$, we can also use
 \begin{equation*}
  \widetilde{\Theta}^{(m)}_{\mr{ri}}:=
   \sum_{|\underline{k}|=n}\bigl(\underline{\partial}
   ^{\angles{m}{p^m}}\bigr)^{\underline{k}}\,
   a^{p^m}_{\underline{k}},
 \end{equation*}
 to get {\em right presentations}. The construction is essentially the
 same, so we leave the details to the reader.
\end{rem*}

\subsection{}
 \label{moreexplicitdesc}
 We used $\widetilde{\Theta}_{\mr{le}}^{(m)}$ to describe elements of
 $\Ecomp{m}{\ms{X}}$. We may also use a variant of
 $\widetilde{\Theta}_{\mr{le}}^{(m+j)}$ for $j\geq 0$ to describe
 them. Suppose $\Theta$ is written as (\ref{desklocal}). Then we put
 \begin{equation*}
  \Theta^{(m,m+j)}:=\sum_{|\underline{k}|=n}a^{p^{m+j}}
  _{\underline{k}}\,\bigl(\underline{\xi}^{\angles{m}{p^{m}}}\bigr)^
  {\underline{k}\,p^j},
  \qquad
  \widetilde{\Theta}_{\mr{le}}^{(m,m+j)}:=\sum_{|\underline{k}|=n}
  a^{p^{m+j}}_{\underline{k}}\,
  \bigl(\underline{\partial}^{\angles{m}{p^{m}}}\bigr)
  ^{\underline{k}\,p^j}.
 \end{equation*}
 {\em If there is no risk of confusion, we sometimes abbreviate
 $\widetilde{\Theta}^{(m)}_{\mr{le}}$
 (resp.\ $\widetilde{\Theta}^{(m,m')}_{\mr{le}}$) as $\Thetatil{(m)}$
 (resp.\ $\Thetatil{(m,m')}$)}.

 \begin{lem*}
  Let $m'\geq m$ be an integer, and we put $j:=m'-m$.

  (i) Let $r_{m,m'}:=(p^{m'}!)\cdot(p^m!)^{-p^j}\in\mb{Z}_p$. Then
  $\Theta^{(m,m')}=r_{m,m'}^n\cdot\Theta^{(m')}$.

  (ii) The operator $\widetilde{\Theta}_{\mr{le}}^{(m,m')}$ is invertible
  in $\Gamma(\ms{U},\Emod{m}{\ms{X}})$.
 \end{lem*}
 \begin{proof}
  We know that for any $1\leq i\leq d$,
  \begin{equation*}
   \bigl(\xi_i^{{\angles{m}{p^m}}}\bigr)^{p^j}=
    r_{m,m'}\cdot\xi_i^{\angles{m'}{p^{m'}}}.
  \end{equation*}
  Since $r_{m,m'}$ does not depend on $i$, we get (i) by definition.
  For the proof of (ii), just copy the proof of the invertibility
  of $\widetilde{\Theta}_{\mr{le}}^{(m)}$ in \ref{defofsituation}.
 \end{proof}

Let $m'\geq m$ be an integer. We claim that any
$S\in\Gamma(\ms{U},\Ecomp{m}{\ms{X}})$ can also be written as
\begin{equation*}
 \sum_{N\in\mb{Z}}\sum_{|\underline{k}|-inp^{m'}=N}
  c_{\underline{k},i}\,\underline{\partial}
  ^{\angles{m}{\underline{k}}}\,(\widetilde{\Theta}^{(m,m')}
  _{\mr{le}})^{-i},
\end{equation*}
with a sequence $\{c_{\underline{k},i}\}$ in
$\Gamma(\ms{X},\mc{O}_{\ms{X}})$ for $\underline{k}\in\mb{N}^d$
and $i\in\mb{N}$, such that the following holds: for each integer $N$,
let $\gamma_{N,i}:=\sup_{|\underline{k}|=inp^{m'}+N}|
c_{\underline{k},i}|$. Then
\begin{equation*}
 \lim_{i\rightarrow\infty}\beta_{N,i}=0,\qquad
  \lim_{N\rightarrow+\infty}\sup_i\{\gamma_{N,i}\}=0.
\end{equation*}
The verification is left to the reader.

\subsection{}
\label{uniqrepisposs}
Assume further that $\Theta$ is of the form
$\xi^{\underline{k}_0}+\sum_{\underline{k}}
a_{\underline{k}}\,\xi^{\underline{k}}$ where
$a_k\in\Gamma(\ms{X},\mc{O}_{\ms{X}})$, $|\underline{k}_0|>0$, and
$\underline{k}$ runs through $\underline{k}\in\mb{N}^d$ such that
$|\underline{k}|=|\underline{k}_0|$ and
$\underline{k}\neq\underline{k}_0$. Then we can check that
$\Gamma(D(\Theta),\mc{O}_{T^{(m)*}X_j}(*))$ is a free
$\Gamma(X_j,\mc{O}_{X_j})$-module with basis
$\bigl\{\underline{\xi}^{\angles{m}{\underline{k}}}
\cdot(\Theta^{(m,m')})^{-i}\bigr\}_{(\underline{k},i)\in I}$ where $I$
is the set of pairs $(\underline{k},i)$ such that $i\geq0$ and
$\underline{k}-p^{m'}\underline{k}_0$ is not $\geq\underline{0}$. Let
$m'\geq m$. Using this basis, we define a continuous homomorphism of
left $\Gamma(\ms{X},\mc{O}_{\ms{X}})$-modules
\begin{equation*}
 \phi\colon\Gamma\bigl(D(\Theta),\mc{O}_{T^{(m)*}\ms{X}}(*)\bigr)
  \rightarrow\Gamma\bigl(D(\Theta),\Emod{m}{\ms{X}}\bigr)
\end{equation*}
by sending $\underline{\xi}^{\angles{m}{\underline{k}}}
\cdot(\Theta^{(m,m')})^{-i}$ to
$\underline{\partial}^{\angles{m}{\underline{k}}}
\cdot(\widetilde{\Theta}^{(m,m')}_{\mr{le}})^{-i}$.
Now, take $P_k\in\Gamma(\ms{U},\mc{O}_{T^{(m)*}\ms{X}}(k))$
for each $k\in\mb{Z}$ such that $\lim_{k\rightarrow\infty}|P_k|=0$. Then
the infinite sum $\sum_k\phi(P_k)$ makes sense in
$\Gamma(\ms{U},\Ecomp{m}{\ms{X}})$ as (\ref{sum}).
Conversely, we have the following lemma whose verification is similar.

\begin{lem*}
 For any element $P\in\Gamma(\ms{U},\Ecomp{m}{\ms{X}})$, there
 exists {\em unique}
 $P_k\in\Gamma(\ms{U},\mc{O}_{T^{(m)*}\ms{X}}(k))$
 for each $k\in\mb{Z}$ such that $\lim_{k\rightarrow\infty}|P_k|=0$,
 and $P=\sum_{k}\phi(P_k)$.
\end{lem*}

\begin{ex}
 Consider the case where $\dim(\ms{X})=1$, and $\ms{X}$ possesses a
 local coordinate. Let $\ms{U}:=T^{(m)*}\ms{X}\setminus s(\ms{X})$ where
 $s$ is the zero section.
 Take $k>0$, and write $k=q\cdot p^m-r$ where $0\leq
 r<p^m$. Put $\partial^{\angles{m}{-k}}:=\partial^{\angles{m}{r}}
 \cdot(\partial^{\angles{m}{p^m}})^{-q}$, which is defined in
 $\Gamma(\ms{U},\Emod{m}{\ms{X}})$. Then any element of
 $\Gamma(\ms{U},\Ecomp{m}{\ms{X}})$ can be
 written uniquely as
 $\sum_{_{k\in\mb{Z}}}a_k\cdot\partial^{\angles{m}{k}}$ with
 $a_k\in\Gamma(\ms{X},\mc{O}_\ms{X})$ such
 that $\lim_{k\rightarrow\infty}a_k=0$.
\end{ex}

\subsection{}
\label{charvardef}
Let us clarify the relation between the characteristic varieties and the
supports of microlocalizations. Let $\ms{M}$ be a coherent
$\DcompQ{m}{\ms{X}}$-module. Let us recall the definition of
the characteristic variety of $\ms{M}$ defined in
\cite[5.2.4]{BerInt}. First, we
take a $\pi$-torsion free coherent $\Dcomp{m}{\ms{X}}$-module $\ms{M}'$
such that $\ms{M}'\otimes\mb{Q}\cong\ms{M}$ using
\cite[3.4.5]{Ber1}. Then, $\ms{M}'/\pi$ is a coherent
$\Dmod{m}{X_0}$-module. Now, we can check that
$\mr{Char}(\ms{M}'/\pi)\subset T^{(m)*}X_0$ (cf.\ Definition
\ref{charandsuppformal}) does not depend on the choice of
$\ms{M}'$. This $\mr{Char}(\ms{M}'/\pi)$ is called the {\em
characteristic variety} of $\ms{M}$ denoted\footnote{
We warn that the characteristic variety $\mr{Char}^{(m)}(\ms{M})$ is
denoted by $\mr{Car}^{(m)}(\ms{M})$ in \cite{BerInt}.}
by $\mr{Char}^{(m)}(\ms{M})$. By using the canonical homeomorphism
$T^{(m)*}X_0\approx T^{(m)*}\ms{X}$, we consider that the characteristic
varieties are in $T^{(m)*}\ms{X}$.

Let us define another subvariety of $T^{(m)*}\ms{X}$ defined by
$\ms{M}$. Consider the following coherent $\EcompQ{m}{\ms{X}}$-module
\begin{equation*}
 \EcompQ{m}{\ms{X}}(\ms{M}):=\EcompQ{m}{\ms{X}}
  \otimes_{\pi_m^{-1}\DcompQ{m}{\ms{X}}}\pi_m^{-1}\ms{M},
\end{equation*}
which is called the {\em microlocalization} of $\ms{M}$. Note here that
since $\EcompQ{m}{\ms{X}}(\ms{M})$ is an $\EcompQ{m}{\ms{X}}$-module of
finite type, the support
$\mr{Supp}(\EcompQ{m}{\ms{X}}(\ms{M}))\subset T^{(m)*}\ms{X}$ is closed
by \cite[$0_{\mr{I}}$, 5.2.2]{EGA}\footnote{In \cite{EGA}, only
commutative case is treated, but the same argument can be used also for
non-commutative case.}.

\begin{prop}
\label{characteristicvariety}
 Let $\ms{M}$ be a coherent $\DcompQ{m}{\ms{X}}$-module. Then, we have
 the following equality of closed subsets of $T^{(m)*}\ms{X}$:
 \begin{equation*}
  \mr{Char}^{(m)}(\ms{M})=\mr{Supp}(\EcompQ{m}{\ms{X}}(\ms{M})).
 \end{equation*}
\end{prop}
\begin{proof}
 Take a coherent $\Dcomp{m}{\ms{X}}$-module $\ms{M}'$ flat over $R$ and
 $\ms{M}'\otimes\mb{Q}\cong\ms{M}$.
 Let us calculate the support of the
 microlocalization. Since $\Ecomp{m}{\ms{X}}$ is pointwise
 Zariskian with respect to the $\pi$-adic filtration by Remark
 \ref{noetheriansheafofmic}, the $\pi$-adic filtration on
 $\Ecomp{m}{\ms{X}}\otimes\pi_m^{-1}\ms{M}'$ is separated by Lemma
 \ref{Zarisep}, and thus
 \begin{equation*}
  \mr{Supp}(\Ecomp{m}{\ms{X}}\otimes\pi_m^{-1}\ms{M}')
   =\mr{Supp}(\Emod{m}{X}\otimes\pi_m^{-1}\ms{M}')=
   \mr{Char}(\ms{M}'\otimes k)=:
   \mr{Char}^{(m)}(\ms{M}),
 \end{equation*}
 where the second equality holds by Lemma \ref{charandsuppformal}.
 Moreover, since $\Ecomp{m}{\ms{X}}$ is flat
 over $\pi_m^{-1}\Dcomp{m}{\ms{X}}$,
 $\Ecomp{m}{\ms{X}}\otimes\ms{M}'$ is $\pi$-torsion free. This implies
 that $\mr{Supp}(\EcompQ{m}{\ms{X}}(\ms{M}))=
 \mr{Supp}(\Ecomp{m}{\ms{X}}\otimes\pi_m^{-1}\ms{M}')$, and the
 proposition follows.
\end{proof}

\begin{rem}
\label{remBer}
 P. Berthelot pointed out to the author another method to define
 $\Ecomp{m}{\ms{X}}$. Let $\ms{X}$ be a smooth affine formal scheme over
 $R$. Let $\Theta$ be a homogeneous section of
 $\Gamma(T^*\ms{X},\mc{O}_{T^*\ms{X}})$. For each $i\geq 0$, there
 exists an integer $m'\geq m$ such that
 $\widetilde{\Theta}^{(m,m')}_{\mr{le}}$
 is contained in the center of $\Dmod{m}{X_i}$. Note that in this case,
 $\widetilde{\Theta}^{(m,m')}_{\mr{le}}=
 \widetilde{\Theta}^{(m,m')}_{\mr{ri}}$.
 Let $A$ be a ring, and $S$ be a multiplicative system of $A$ consisting
 of elements in the center of $A$. We can construct the ring of
 fractions $S^{-1}A$  as the commutative case. (The details are left to
 the reader.) Using this, we define
 \begin{equation*}
  \Gamma(D(\Theta^{(m)}),\ms{L}\Dmod{m}{X_i}):=
   S_{\Theta^{(m,m')}}^{-1}\,\Gamma(D(\Theta^{(m)}),
   \pi^{-1}\Dmod{m}{X_i}),
 \end{equation*}
 where $S_{\Theta^{(m,m')}}$ denotes the multiplicative system generated
 by $\widetilde{\Theta}^{(m,m')}_{\mr{le}}$. We can check easily that
 this does not depend on the choice of $m'$ and defines a sheaf. By
 taking the completion with respect to the filtration by order, we get
 $\Emod{m}{X_i}$. By definition, the sheaf $\ms{L}\Dmod{m}{X_i}$ is a
 noetherian ring.
\end{rem}

\section{Pseudo cotangent bundles and pseudo polynomials}
\subsection{}
Recall the notation of \ref{limitnaive}.
Let $A$ be a commutative $R$-algebra, and $m$ be a non-negative integer,
$d$ be a positive integer. We define 
\begin{equation*}
 A[\xi_1,\dots,\xi_d]^{(m)}:=A\bigl[\xi_j^{\angles{m}{p^i}}
  \mid j=1,\dots,d,\ i=0,\dots,m\bigr]/I_m
\end{equation*}
where $\xi_j^{\angles{m}{p^i}}$ is an indeterminate for any $i$ and $j$,
and $I_m$ is the ideal generated by the relations
\begin{equation*}
 (\xi_j^{\angles{m}{p^i}})^p=\frac{(p^{i+1})!}{(p^i!)^p}\,
  \xi_j^{\angles{m}{p^{i+1}}}
\end{equation*}
for $1\leq j\leq d$ and $0\leq i<m$. We note that
$(p^{i+1}!)\cdot(p^i!)^{-p}\in\mb{Z}_p$. We define
$\deg\bigl(\xi_j^{\angles{m}{p^i}}\bigr):=p^i$, which makes
$A[\xi_1,\dots,\xi_d]^{(m)}$ a graded ring. We call this the {\em ring
of pseudo polynomials over $A$}. We denote
by $A\{\xi_1,\dots,\xi_d\}^{(m)}$ the $\pi$-adic completion of
$A[\xi_1,\dots,\xi_d]^{(m)}$. We call this the {\em pseudo Tate algebra
over $A$}. We note that for an $R$-algebra $A$,
\begin{equation*}
 A\otimes_RR[\xi_1,\dots,\xi_d]^{(m)}\cong A[\xi_1,\dots,\xi_d]^{(m)}.
\end{equation*}

\begin{lem*}
 \label{pseudopolyrateq}
 Let $A$ be a commutative $R$-algebra. For any non-negative integers
 $m'\geq m$, there exists a unique isomorphism of graded rings
 \begin{equation*}
  A[\xi_1,\dots,\xi_d]^{(m)}\otimes\mb{Q}\xrightarrow{\sim}
   A[\xi_1,\dots,\xi_d]^{(m')}\otimes\mb{Q}
 \end{equation*}
 sending $\xi_i^{\angles{m}{1}}$ to $\xi_i^{\angles{m'}{1}}$ for $1\leq
 i\leq d$.
\end{lem*}
\begin{proof}
 Left to the reader.
\end{proof}

\begin{lem}
 \label{isomcoordgr}
 Let $\ms{X}=\mr{Spf}(A)$ be an affine smooth formal scheme over $R$
 possessing a system of local coordinates $\{x_1,\dots,x_d\}$ on
 $\ms{X}$. Let $A_i:=A\otimes_RR_i$. Then there exists a unique
 isomorphism of graded rings
 \begin{equation*}
  A_i[\xi_1,\dots,\xi_d]^{(m)}\xrightarrow{\sim}
   \Gamma(X_i,\mr{gr}(\Dmod{m}{X_i}))
 \end{equation*}
 sending $\xi_j^{\angles{m}{p^i}}$ to
 $\sigma(\partial_j^{\angles{m}{p^i}})$, where $\sigma$ denotes the
 principal symbol {\normalfont(cf.\
 {\normalfont\ref{princisymbdfn}})}, for $1\leq k\leq d$.
\end{lem}
\begin{proof}
 To construct the homomorphism of graded rings, use \cite[2.2.4]{Ber1}.
 This homomorphism is surjective by \cite[2.2.5]{Ber1}. To check the
 injectivity, it suffices to show that, for any $k\geq0$, the parts of
 degree $k$ of both sides are free over $A_i$ with the same ranks. The
 detail is left to the reader.
\end{proof}

\subsection{}
\label{tangspident}
Let $X$ be a smooth scheme over $k$, and $m\geq 0$ be an integer. Let
$X^{(m)}:=X\otimes_{k,F^{m*}_k}k$ where $F^{m*}_k\colon k\rightarrow k$
is the $m$-th absolute Frobenius homomorphism ({\it i.e.\ }the
homomorphism sending $x$ to $x^{p^m}$).
By \cite[5.2.2]{BerInt}, we have a canonical isomorphism
\begin{equation*}
 (T^{(m)*}X)_{\mr{red}}\xrightarrow{\sim}X\times_{X^{(m)}}T^*X^{(m)}
\end{equation*}
where $_{\mr{red}}$ denotes the associated reduced scheme. The scheme
$T^*X^{(m)}$ is deduced from $T^*X$ by the base change
$X^{(m)}\rightarrow X$. This induces the canonical morphism of
schemes (which may not be a morphism over $k$)
\begin{equation*}
 (T^{(m)*}X)_{\mr{red}}\rightarrow T^*X
\end{equation*}
such that the underlying continuous map is a homeomorphism
of topological spaces.

Now, let $\ms{X}$ be a smooth formal scheme over $R$. Since the
topological space of $T^{(m)*}X_i$ is
homeomorphic to that of $T^{(m)*}\ms{X}$, we also get a canonical
homeomorphism $T^{(m)*}\ms{X}\approx T^{*}\ms{X}$. 
Consider the situation as in \ref{defofsituation}. The affine open
subset of $T^*\ms{X}$ defined by
$\Theta$ and that of $T^{(m)*}\ms{X}$ defined by $\Theta^{(m)}$ are
homeomorphic under this canonical homeomorphism. From now on, {\em we
identify the spaces $T^*\ms{X}$, $T^{(m)*}\ms{X}$, $T^*X_i$, and
$T^{(m)*}X_i$ using these homeomorphisms}. In particular, we consider
$\Emod{m}{\ms{X}}$ {\it etc.\ }as sheaves on $T^*\ms{X}$ or $T^*X_i$. We
denote the projection $\pi\colon T^*\ms{X}\rightarrow\ms{X}$. The
notation ``$\pi$'' is the same as the uniformizer of $R$, but we do not
think there is any confusion. This identification also induces
the identification of topological spaces
\begin{equation*}
 P^*\ms{X}\approx P^{*}X_i\approx P^{*(m)}X_i\approx P^{*(m)}\ms{X}.
\end{equation*}

\begin{lem}
\label{normcalc}
 Let $\ms{X}$ be an affine smooth formal scheme over $R$ of dimension
 $d$. We use the notation and the identifications in
 {\normalfont\ref{tangspident}}.
 We take non-negative integers $m'\geq m$.

 (i) For any integer $k$, there exist integers $a_k,b_k\geq0$ such that
 the following holds: let
 $\Theta\in\Gamma(T^*\ms{X},\mc{O}_{T^*\ms{X}})$ be a homogeneous
 section of degree $n$.

 \begin{itemize}
  \item[(a)] The operator
	     \begin{equation*}
	      p^{a_k}\,\underline{\partial}^{\angles{m}{\underline{l}}}\,
	       (\Thetatil{(m,m')})^{-i},
	     \end{equation*}
	     which is {\it a priori} contained in
	     $\Gamma(D(\Theta),\Emod{m'}{\ms{X},\mb{Q}})$ by Lemma
	     {\normalfont\ref{moreexplicitdesc} (i)}, is contained
	     in $\Gamma(D(\Theta),\Emod{m'}{\ms{X}})$ for any
	     $|\underline{l}|-inp^{m'}\geq k$. If $d\,p^{m'+1}<k$, we
	     may take $a_k=0$.
  \item[(b)] The operator
	     \begin{equation*}
	      p^{b_k}\,\partial^{\angles{m'}{\underline{l}}}\,
	       (\Thetatil{(m')})^{-i}
	     \end{equation*}
	     is in $\Gamma(D(\Theta),\Emod{m}{\ms{X}})$ for any
	     $|\underline{l}|-inp^{m'}\leq k$. If $k<p^{m+1}$, we may
	     take $b_k=0$.
 \end{itemize}

 (ii) Let $\Theta\in\Gamma(T^*\ms{X},\mc{O}_{T^*\ms{X}})$ be a
 homogeneous section. Take an integer $m''$ such that $m\leq m''\leq
 m'$. Suppose $P=\alpha\cdot\underline{\partial}^{\angles{m}
 {\underline{k}}}\,(\Thetatil{(m,m')})^{-i}$ with
 $\alpha\in R$ is contained in
 $\Gamma(D(\Theta),\Emod{m'}{\ms{X}})$. Then it is also contained in
 $\Gamma(D(\Theta),\Emod{m''}{\ms{X}})$.
\end{lem}
\begin{proof}
 First, let us show (i). Since the proof for (b) is essentially the
 same, we concentrate on proving (a). We show the following.
 \begin{cl}
  Let $m'\geq0$ be an integer. For integers $m$, $a$, $k$ such that
  $m'\geq m\geq0$, $m'-m\geq a\geq0$, there exists
  an integer $\alpha_{k,m,a}\geq0$
  such that, for any $\Theta$ and $|\underline{l}|-inp^{m'}\geq k$,
  $p^{\alpha_{k,m,a}}\underline{\partial}^{\angles{m}{\underline{l}}}
  (\Thetatil{(m,m')})^{-i}$ is equal to
  $\alpha\cdot\underline{\partial}^{\angles{m+a}{\underline{l}}}
  (\Thetatil{(m+a,m')})^{-i}$ with some $\alpha\in
  \mb{Z}_p$. If $k>d\,p^{m'+1}$, we can take $\alpha_{k,m,a}=0$.
 \end{cl}
 Once this claim is proven, the lemma follows by taking $a=m'-m$.
 \begin{proof}[Proof of the claim]
  \renewcommand{\qedsymbol}{$\square$}
  Let $b:=m'-m-1\geq-1$. We show the claim using the induction on
  $b$. When $b=-1$ or more generally $a=0$, we can take
  $\alpha_{k,m,a}=0$. Since we can take
  $\alpha_{k,m,a}=\alpha_{k,m+1,a-1}+\alpha_{k,m,1}$, it suffices to
  show the existence of $\alpha_{k,m,1}$ by the induction hypothesis.

  There exists a number $c\in\mb{Z}_p^*$ such that
  \begin{equation*}
   \Thetatil{(m,m')}=c\,p^{np^b}\cdot
    \Thetatil{(m+1,m')}.
  \end{equation*}
  For $\underline{l'}\in\mb{N}^d$, we put
  \begin{equation*}
   g(\underline{l'}):=\sum_{j=1}^d\biggl[\frac{l'_j}{p^{m+1}}\biggr]
  \end{equation*}
  where $[\alpha]$ denotes the maximum integer less
  than or equal to $\alpha$. Then
  \begin{equation*}
   \underline{\partial}^{\angles{m}{\underline{l}}}=c'\,
    p^{g(\underline{l})}\,\underline{\partial}
    ^{\angles{m+1}{\underline{l}}}
  \end{equation*}
  with $c'\in \mb{Z}_p^*$. Since
  \begin{equation*}
   \frac{l_j}{p^{m+1}}-1<\biggl[\frac{l_j}{p^{m+1}}\biggr]
    \leq\frac{l_j}{p^{m+1}},
  \end{equation*}
  we get inequalities
  \begin{equation*}
   inp^b+\frac{k}{p^{m+1}}-d\leq\frac{|\underline{l}|}{p^{m+1}}-d
    <\sum_{j=1}^d\biggl[\frac{l_j}{p^{m+1}}\biggr]=g(\underline{l}).
  \end{equation*}
  Thus,
  \begin{equation*}
   \underline{\partial}^{\angles{m}{\underline{l}}}\,
    (\Thetatil{(m,m')})^{-i}=
    c'\,c^{-i}\,p^{g(\underline{l})-inp^b}\,\underline{\partial}
    ^{\angles{m+1}{\underline{l}}}\,
    (\Thetatil{(m+1,m')})^{-i},
  \end{equation*}
  and we may take
  $\alpha_{k,m,1}=\max\bigl\{0,[d-kp^{-(m+1)}+1]\bigr\}$.
  Thus, we conclude the proof of the claim.
 \end{proof}

 Let us prove (ii) on $D(\Theta)$. We get
 \begin{equation*}
  \Thetatil{(m,m')}=u~(p^{m'-m}~!)^n\cdot\Thetatil{(m',m')}
 \end{equation*}
 where $n$ denotes the order of $\Theta$, and $u$ denotes a number in
 $\mb{Z}_p^*$. Thus, for $m\leq l\leq m'$,
 \begin{equation}
  \label{calcreldiflev1}
  \Thetatil{(m,m')}=u'\frac{(p^{m'-m}~!)^n}
   {(p^{m'-l}~!)^n}\cdot\Thetatil{(l,m')}
   =u_l~p^{a_l}\cdot\Thetatil{(l,m')}
 \end{equation}
 where $u'$ and $u_l$ denote numbers in $\mb{Z}_p^*$, and $a_l$ is
 equal to $n\cdot(p^{m'-l}+\dots+p^{m'-m-1})$. We also get
 \begin{equation}
  \label{calcreldiflev2}
  \partial^{\angles{m}{\underline{k}}}=u'\,p^{b_l}\,
   \partial^{\angles{l}{\underline{k}}},\qquad
   b_l=\sum_{j=1}^d\sum_{i=m+1}^{l}[p^{-i}k_j].
 \end{equation}
 Now, we define two functions $f,g\colon[m,m']\rightarrow\mb{R}$.
 The function $f$ is the continuous function such that it is affine on
 the interval $[l,l+1]$ for any integer $l$ in $\left[m,m'\right[$, and
 \begin{equation*}
  f(l):=\mr{ord}_p(\alpha)+b_l=\mr{ord}_p(\alpha)+
   \sum_{j=1}^d\sum_{i=m+1}^{l}[p^{-i}k_j],
 \end{equation*}
 where $\mr{ord}_p$ denotes the $p$-adic order normalized so that
 $\mr{ord}_p(p)=1$. The function $g$ is the continuous function such
 that it is affine on the interval $[l,l+1]$ for any integer $l$ in
 $\left[m,m'\right[$, and
 \begin{equation*}
  g(l):=i\cdot a_l=ni\cdot(p^{m'-l}+p^{m'-l+1}+\dots+p^{m'-m-1}).
 \end{equation*}
 Since the operator
 $\partial^{\angles{m}{\underline{k}}}(\Thetatil{(m,m')})^{-i}$ is a
 section of $\Emod{m}{\ms{X}}$, we have $g(m)\leq f(m)$. By
 (\ref{calcreldiflev1}) and (\ref{calcreldiflev2}), it suffices to show
 that if $g(m')\leq f(m')$, then $g(l)\leq f(l)$ for any integer $l$ in
 $\left[m,m'\right]$. We put
 \begin{equation*}
  Df(l):=f(l)-f(l-1)=\sum_{j=1}^d\,[p^{-l}\,k_j],\qquad
   Dg(l):=g(l)-g(l-1)=nip^{m'-l}.
 \end{equation*}
 For any $a\in\mb{R}$, we have $p^{-1}\cdot[a]\geq[p^{-1}a]$. Indeed,
 $p^{-1}a\geq[p^{-1}a]$, and $a\geq p\cdot[p^{-1}a]$. Since
 $p\cdot[p^{-1}a]$ is an integer, we get what we want by the definition
 of $[\cdot]$. This implies that
 $p^{-1}\cdot[p^{-l}k_j]\geq[p^{-(l+1)}k_j]$, and thus
 \begin{equation*}
  p^{-1}\cdot Df(l)\geq Df(l+1).
 \end{equation*}
 In turn, we have $p^{-1}\cdot Dg(l)=Dg(l+1)$. Suppose there exists an
 integer $l$ in $\left]m,m'\right]$ such that $f(l)<g(l)$ and $f(a)\geq
 g(a)$ for any integer $a$ in $\left[m,l\right[$. This shows
 that $Df(l)<Dg(l)$. Thus, $Df(b)<Dg(b)$ for any $l\leq b$, which
 implies $f(m')<g(m')$. This contradicts with the assumption, and we
 conclude that $g(l)\leq f(l)$ for any $m\leq l\leq m'$.
\end{proof}

\subsection{}
Let $M$ and $M'$ be {\em $\pi$-torsion free} $R$-modules. A {\em
$p$-isogeny} $\phi\colon M\dashrightarrow M'$ is an isomorphism
\begin{equation*}
 \phi_{\mb{Q}}:M\otimes\mb{Q}\xrightarrow{\sim}M'\otimes\mb{Q}
\end{equation*}
such that there exist positive integers $n$ and $n'$ satisfying
\begin{equation*}
 p^n\cdot\phi_{\mb{Q}}(M)\subset M'\subset
  p^{-n'}\cdot\phi_{\mb{Q}}(M).
\end{equation*}
Here $\phi_{\mb{Q}}$ is called the {\em realization} of the
$p$-isogeny. We say that the $p$-isogeny is a {\em homomorphism} if we
can take $n$ to be $0$.

\begin{lem*}
 \label{complofisog}
 Let $M$ and $M'$ be $\pi$-torsion free $R$-modules, and let
 $\phi\colon M\dashrightarrow M'$ be a $p$-isogeny. Then this induces a
 canonical $p$-isogeny
 \begin{equation*}
  \widehat{\phi}\colon M^{\wedge}\dashrightarrow M'^{\wedge}.
 \end{equation*}
 where $^\wedge$ denotes the $\pi$-adic completion. If the given
 $p$-isogeny is a homomorphism, the induced $p$-isogeny is also a
 homomorphism.
\end{lem*}
\begin{proof}
 Let $\phi_{\mb{Q}}\colon M\otimes\mb{Q}\rightarrow M'\otimes\mb{Q}$ be
 the realization of the isogeny. By definition, there exists an integer
 $n$ such that $p^n\cdot\phi_{\mb{Q}}$ induces a homomorphism
 $M\rightarrow M'$. We denote this homomorphism by $\phi_n\colon
 M\rightarrow M'$. Let $C:=\mr{Coker}(\phi_n)$. Since $\phi$ is a
 $p$-isogeny, $C$ is a $\pi^{n'}$-torsion module for some integer
 $n'\geq 0$. We have an exact sequence of projective systems
 \begin{equation*}
  0\rightarrow\{C\}'_{i\geq n'}\rightarrow
   \bigl\{M\otimes R_i\bigr\}_{i\geq n'}\rightarrow
   \bigl\{M'\otimes R_i\bigr\}_{i\geq n'}
   \rightarrow\{C\}_{i\geq n'}
   \rightarrow0,
 \end{equation*}
 where $\{C\}'_{i\geq n'}$ denotes the projective system
 $\bigl\{\dots\rightarrow C\xrightarrow{\pi}C\xrightarrow{\pi}
 C\rightarrow\dots\bigr\}$, and $\{C\}_{i\geq n'}$ is the projective
 system $\bigl\{\dots\rightarrow C\xrightarrow{\mr{id}}C
 \xrightarrow{\mr{id}}C\rightarrow\dots\bigr\}$. Since any projective
 system appearing in the short exact sequence above satisfies the
 Mittag-Leffler condition, the exact sequence induces an exact sequence
 \begin{equation*}
  0\rightarrow\widehat{M}\xrightarrow{\widehat{\phi_n}}\widehat{M'}
   \rightarrow C\rightarrow0
 \end{equation*}
 by taking the projective limit. Thus, we get a $p$-isogeny
 $\widehat{\phi}_{\mb{Q}}:=p^{-n}\cdot\widehat{\phi_n}\colon\widehat{M}
 \otimes\mb{Q}\rightarrow\widehat{M'}\otimes\mb{Q}$ as desired. By
 construction, the homomorphism $\widehat{\phi}_{\mb{Q}}$ does not
 depend on the choice of the integer $n$.
\end{proof}

Let $\ms{M}$, $\ms{M}'$ be $\pi$-torsion free $R$-modules a topological
space $X$. Then exactly in the same way, we can define $p$-isogeny
$\phi\colon\ms{M}\dashrightarrow\ms{M}'$. Namely, it is a homomorphism
of sheaves of modules $\phi_{\mb{Q}}\colon\ms{M}\otimes\mb{Q}
\rightarrow\ms{M}'\otimes\mb{Q}$ such that there exist positive integers
$n$ and $n'$ satisfying
$p^n\cdot\phi_{\mb{Q}}(\ms{M})\subset\ms{M}'\subset
p^{-n'}\phi_{\mb{Q}}(\ms{M})$. We say that the $p$-isogeny is a
homomorphism if we can take $n$ to be $0$.

\subsection{}
\label{pisogconst}
Let $\ms{X}=\mr{Spf}(A)$ be an affine smooth formal scheme over $R$, and
assume that it possesses a system of local coordinates
$\{x_1,\dots,x_d\}$. We identify the ring of global sections of
$\mc{O}_{T^{(m)*}\ms{X}}$ with $A\{\xi_1,\dots,\xi_d\}^{(m)}$ using
Lemma \ref{isomcoordgr}. Let $\Theta$ be a homogeneous element of
$A[\xi_1,\dots,\xi_d]$ whose degree is {\em strictly greater than
$0$}. For a commutative graded ring $\Lambda$ and a homogeneous element
$f\in\Lambda$, we denote the submodule of degree $n$ of the graded ring
$\Lambda_f$ by $\Lambda_{(f)}(n)$. Then by construction of
$\mc{O}_{P^{*(m)}\ms{X}}$,
\begin{equation*}
 \Gamma(D_+(\Theta),\mc{O}_{P^{*(m)}\ms{X}}(n))\cong
  (A[\xi_1,\dots,\xi_d]^{(m)}_{(\Theta^{(m)})}(n))^{\wedge}
\end{equation*}
where $^{\wedge}$ denotes the $\pi$-adic completion, and we used the
notation of \cite[II, (2.3.3)]{EGA}. For $m'\geq m$, we note that
\begin{equation}
\label{isomorphdifflevel}
 (A[\xi_1,\dots,\xi_d]^{(m)}_{(\Theta^{(m)})}(n))^{\wedge}\cong
  (A[\xi_1,\dots,\xi_d]^{(m)}_{(\Theta^{(m,m')})}(n))^{\wedge},
\end{equation}
since there exists $Q\in A[\xi_1,\dots,\xi_d]^{(m)}$ such that
\begin{equation*}
 (\Theta^{(m)})^{p^{m'-m}}=\Theta^{(m,m')}+pQ.
\end{equation*}
Lemma \ref{moreexplicitdesc} (i) and the isomorphism
$A[\xi_1,\dots,\xi_d]^{(m)}\otimes\mb{Q}\cong
A[\xi_1,\dots,\xi_d]^{(m')}\otimes\mb{Q}$ of Lemma \ref{pseudopolyrateq}
induces the following homomorphism.
\begin{equation*}
 A[\xi_1,\dots,\xi_d]^{(m')}_{\Theta^{(m')}}\rightarrow
  A[\xi_1,\dots,\xi_d]^{(m)}_{\Theta^{(m,m')}}\otimes\mb{Q}
\end{equation*}
Using Lemma \ref{normcalc} (i)-(b), this homomorphism defines a
$p$-isogeny
\begin{equation*}
 A[\xi_1,\dots,\xi_d]^{(m')}_{(\Theta^{(m')})}(n)\dashrightarrow
  A[\xi_1,\dots,\xi_d]^{(m)}_{(\Theta^{(m,m')})}(n)
\end{equation*}
for any $n\in\mb{Z}$. For $n<p^{m+1}$, this $p$-isogeny is moreover a
homomorphism by the same lemma. This defines a $p$-isogeny
\begin{equation*}
 (A[\xi_1,\dots,\xi_d]^{(m')}_{(\Theta^{(m')})}(n))^{\wedge}
  \dashrightarrow
  (A[\xi_1,\dots,\xi_d]^{(m)}_{(\Theta^{(m,m')})}(n))^{\wedge}
\end{equation*}
by Lemma \ref{complofisog}. Composing this with
(\ref{isomorphdifflevel}), we get a canonical $p$-isogeny
\begin{equation*}
 (A[\xi_1,\dots,\xi_d]^{(m')}_{(\Theta^{(m')})}(n))^{\wedge}
  \dashrightarrow
  (A[\xi_1,\dots,\xi_d]^{(m)}_{(\Theta^{(m)})}(n))^{\wedge},
\end{equation*}
which is a homomorphism for $n<p^{m+1}$. By construction, this
$p$-isogeny is compatible with restrictions. Moreover, since $b_n$ of
Lemma \ref{normcalc} does not depend on $\Theta$, this induces a
$p$-isogeny of {\em sheaves}. Summing up, we obtain the following
lemma.

\begin{lem*}
\label{pisog}
 Let $m'\geq m$ be non-negative integers.  For any $n\in\mb{Z}$, there
 exist canonical $p$-isogenies of sheaves of modules
  \begin{equation*}
   \mc{O}_{P^{*(m')}\ms{X}}(n)\dashrightarrow
    \mc{O}_{P^{*(m)}\ms{X}}(n),\qquad
    \mc{O}_{T^{*(m')}\ms{X}}(n)\dashrightarrow
    \mc{O}_{T^{*(m)}\ms{X}}(n)
  \end{equation*}
 on the {\em topological spaces} $P^*\ms{X}$ and $\mathring{T}^*\ms{X}$
 respectively. These are homomorphisms for
 $n<p^{m+1}$.
\end{lem*}

\begin{lem}
 By using the homomorphism of Lemma {\normalfont\ref{pisog}},
 $\mc{O}_{P^{*(m)}\ms{X}}$ can be seen as an
 $\mc{O}_{P^{*(m')}\ms{X}}$-algebra. Then $\mc{O}_{P^{*(m)}\ms{X}}$ is
 a coherent $\mc{O}_{{P}^{*(m')}\ms{X}}$-algebra.
\end{lem}
\begin{proof}
 Let $\Theta$ be a homogeneous element of $A[\xi_1,\dots,\xi_d]$ whose
 degree is strictly greater than $0$. First of all, let us show that the
 homomorphism of rings
 \begin{equation}
  \label{pseudohomfin}
   (A[\xi_1,\dots,\xi_d]^{(m')}_{(\Theta^{(m')})}(0))^{\wedge}\rightarrow
   (A[\xi_1,\dots,\xi_d]^{(m)}_{(\Theta^{(m)})}(0))^{\wedge}
 \end{equation}
 is finite.
 By construction of (\ref{pseudohomfin}), it suffices to show the
 finiteness of the homomorphism
 \begin{equation*}
  A[\xi_1,\dots,\xi_d]^{(m')}_{(\Theta^{(m')})}(0)\rightarrow
   A[\xi_1,\dots,\xi_d]^{(m)}_{(\Theta^{(m,m')})}(0).
 \end{equation*}
 Let
 \begin{equation*}
  S:=\Biggl\{(\underline{k},\underline{k}',i)\in\mb{N}^d
   \times\mb{N}^d\times\mb{N}\,\Bigg|\,\parbox{40ex}{$k_j<p^{m'}$ for
   any $j$, $|\underline{k}'|<\mr{ord}(\Theta)$,
  and $|\underline{k}|+|\underline{k}'|\,p^{m'}=ip^{m'}\,
   \mr{ord}(\Theta)$}\Biggr\}.
 \end{equation*}
 The condition $|\underline{k}|+|\underline{k}'|\,p^{m'}=ip^{m'}\,
 \mr{ord}(\Theta)$ means that the order of
 \begin{equation*}
 \underline{\xi}^{\angles{m}{\underline{k}+\underline{k}'p^{m'}}}
  (\Theta^{(m,m')})^{-i}
 \end{equation*}
 is equal to $0$. Obviously, $\#S<\infty$. Let 
 \begin{equation*}
  T:=\bigl\{\underline{k}\in\mb{N}^d\,\big|\,\underline{k}\in
   p^{m'}\,\mb{N}^d,
   |\underline{k}|=ip^{m'}\,\mr{ord}(\Theta)\mbox{ for some
   integer $i$} \bigr\}.
 \end{equation*}
 The set $T$ is a submonoid of the commutative monoid $\mb{N}^d$. For
 any $\underline{k}\in T$, there exists $u\in\mb{Z}_p^*$ such that
 \begin{equation}
  \label{relationinT}
  \underline{\xi}^{\angles{m}{\underline{k}}}(\Theta^{(m,m')})^{-i}=
   u\cdot\underline{\xi}^{\angles{m'}{\underline{k}}}
   (\Theta^{(m')})^{-i}.
 \end{equation}
 Let
 \begin{equation*}
  U:=\bigl\{\underline{k}\in\mb{N}^d\,\big|\,|\underline{k}|=
   ip^{m'}\,\mr{ord}(\Theta)\mbox{ for some $i$}\bigr\}.
 \end{equation*}
 This is also a submonoid of $\mb{N}^d$. Let
 \begin{equation*}
  S':=\bigl\{\underline{l}\in\mb{N}^d\,\big|\,\mbox{there exists
   $(\underline{k},\underline{k'},i)\in S$ such that }
   \underline{l}=\underline{k}+p^{m'}\underline{k'}\bigr\}.
 \end{equation*}
 The monoid $T$ is a submonoid of $U$, and $S'$ is a finite subset of
 $U$. We claim that $U=T+S'$. Indeed, take $\underline{l}\in
 U$. We can write
 $\underline{l}=\underline{i}+p^{m'}\underline{i'}$ such that
 $\underline{i},\underline{i'}\in\mb{N}^d$ and
 $i_j<p^{m'}$ for any $j$. Now, there exists $\underline{k'}$ such that
 $|\underline{k'}|<\mr{ord}(\Theta)$, $i'_j\geq k'_j$ for any $j$, and
 \begin{equation*}
 |\underline{i'}|-|\underline{k'}|=\left[|\underline{i'}|\cdot(\mr{ord}
 (\Theta))^{-1}\right]\cdot\mr{ord}(\Theta)
 \end{equation*}
 where $[\alpha]$ denotes the maximum integer less than or equal to
 $\alpha$.
 We put $\underline{k}:=\underline{i}$. Then there exists an integer $i$
 such that $|\underline{i}|+p^{m'}|\underline{k'}|=ip^{m'}\mr{ord}
 (\Theta)$. By construction $(\underline{k},\underline{k'},i)\in S$, and
 $p^{m'}\cdot(\underline{i'}-\underline{k'})\in T$.
 Since $\underline{l}=p^{m'}\cdot(\underline{i'}-\underline{k'})+
 (\underline{k}+p^{m'}\underline{k'})$, the claim follows.
 Considering (\ref{relationinT}), this implies that the homomorphism
 \begin{equation*}
 \bigoplus_{\underline{l}\in S}~A[\xi_1,\dots,\xi_d]^{(m')}
  _{(\Theta^{(m')})}(0)\rightarrow A[\xi_1,\dots,\xi_d]^{(m)}
  _{(\Theta^{(m,m')})}(0)
 \end{equation*}
 sending $1$ sitting at the $(\underline{k},\underline{k'},i)\in
 S$ component to $\underline{\xi}^{\angles{m}{\underline{k}+
 \underline{k}'p^{m'}}}(\Theta^{(m,m')})^{-i}$ is surjective. Thus the
 homomorphism (\ref{pseudohomfin}) is finite.

 Let us prove the coherence.
 Let $\Xi$ be another homogeneous element of $A[\xi_1,\dots,\xi_d]$
 whose degree is strictly greater than $0$. Let
 $\ms{U}$ be the affine open subset of $P^*\ms{X}$ defined by $\Theta$,
 and $\ms{U}'$ by that of $\Xi\cdot\Theta$. It suffices to show that the
 canonical homomorphism
 \begin{equation}
  \label{isthiscomlete}
   \Gamma(\ms{U},\mc{O}_{{P}^{*(m)}\ms{X}})\otimes_
   {\Gamma(\ms{U},\mc{O}_{{P}^{*(m')}\ms{X}})}
   \Gamma(\ms{U}',\mc{O}_{{P}^{*(m')}\ms{X}})\rightarrow
   \Gamma(\ms{U}',\mc{O}_{P^{*(m)}\ms{X}})
 \end{equation}
 is an isomorphism.
 By changing $\Theta$ and $\Xi$ to some powers of $\Theta$ and $\Xi$
 respectively, we may assume that
 $\mr{ord}(\Xi)=\mr{ord}(\Theta)$. We put
 \begin{alignat*}{3}
  \mc{A}&:=A[\xi_1,\dots,\xi_d]_{(\Theta^{(m')})}^{(m')}(0),&\qquad
  \Psi&:=\frac{\Xi^{(m')}}{\Theta^{(m')}},\\
  \mc{B}&:=A[\xi_1,\dots,\xi_d]_{(\Theta^{(m,m')})}^{(m)}(0),&\qquad
  \Phi&:=\frac{\Xi^{(m)}}{\Theta^{(m)}},&\qquad
  \Phi'&:=\frac{\Xi^{(m,m')}}{\Theta^{(m,m')}}.
 \end{alignat*}
 Let $\phi\colon\mc{A}\rightarrow\mc{B}$ be the canonical homomorphism.
 Firstly, $\Phi'=\phi(\Psi)$ in $\mc{B}$ by Lemma \ref{moreexplicitdesc}
 (i). Secondly,
 \begin{equation*}
  (\mc{B}_{\Phi})^\wedge\cong(\mc{B}_{\Phi'})^\wedge
 \end{equation*}
 by the same reason as (\ref{isomorphdifflevel}). Thirdly,
 \begin{align*}
  \mc{B}_{\phi(\Psi)}\cong\mc{B}\otimes_{\mc{A}}\mc{A}_{\Psi}.
 \end{align*}
 Combining these, $(\mc{B}\otimes_{\mc{A}}\mc{A}_{\Psi})^\wedge
 \cong(\mc{B}_{\Phi})^\wedge$. This is implies that the $\pi$-adic
 completion of the left hand side of (\ref{isthiscomlete}) is isomorphic
 to the right hand side. However, by the finiteness of
 (\ref{pseudohomfin}), the left hand side of (\ref{isthiscomlete}) is
 already $\pi$-adically complete by \cite[$0_{\mr{I}}$, 7.3.6]{EGA}, and
 as a result, (\ref{isthiscomlete}) is an isomorphism. Thus we obtain
 the lemma.
\end{proof}

\subsection{}
\label{difflevelcoherent}
Let $\ms{X}$ be an affine smooth formal scheme over $R$ possessing a
system of local coordinates. For any $n\in\mb{Z}$, the module
$\mc{O}_{T^{*(m)}\ms{X}}(n)$ is an $\mc{O}_{T^{*(m)}\ms{X}}(0)$-module
on $\mathring{T}^*\ms{X}$, and by using Lemma \ref{pisog},
$\mc{O}_{\mathring{T}^{*(m)}\ms{X}}(n)$ can be seen as an
$\mc{O}_{\mathring{T}^{*(m')}\ms{X}}(0)$-module.

\begin{cor*}
 The $\mc{O}_{\mathring{T}^{*(m')}\ms{X}}(0)$-module
 $\mc{O}_{\mathring{T}^{*(m)}\ms{X}}(n)$ is coherent.
\end{cor*}

\begin{rem*}
 We show in Lemma \ref{stricnessofhom} that the corollary holds for
 any smooth formal scheme $\ms{X}$ not necessary affine.
\end{rem*}

\section{Intermediate microdifferential sheaves}
In section \ref{micdiff}, we defined the ring of naive microdifferential
operators. However, we do not have any natural homomorphism
$\EcompQ{m}{\ms{X}}\rightarrow\EcompQ{m+1}{\ms{X}}$ as shown in
\ref{countcharnotstable}.
To remedy this situation, we consider the intermediate ring of
microdifferential operators denoted by $\EcompQ{m,m'}{\ms{X}}$ for 
$m'\geq m$, which is an ``intersection'' of $\EcompQ{m}{\ms{X}}$ and
$\EcompQ{m'}{\ms{X}}$. In this section, we define these rings and
prove some basic properties.

\subsection{}
 \label{countcharnotstable}
 To start with, let us show the non-existence of a homomorphism
 $\EcompQ{m}{\ms{X}}\rightarrow\EcompQ{m+1}{\ms{X}}$ compatible with the
 canonical homomorphism
 $\DcompQ{m}{\ms{X}}\rightarrow\DcompQ{m+1}{\ms{X}}$. Suppose the
 homomorphism existed. Then, for any coherent
 $\DcompQ{m}{\ms{X}}$-module $\ms{M}$, we would get
 \begin{equation*}
  \mr{Char}^{(m)}(\ms{M})\supset
   \mr{Char}^{(m+1)}(\DcompQ{m+1}{\ms{X}}
   \otimes_{\DcompQ{m}{\ms{X}}}\ms{M})
 \end{equation*}
 by Proposition \ref{characteristicvariety}.
 However, this does not hold as the following example shows:

 \begin{ex*}
  Let $\ms{X}:=\widehat{\mb{A}}^1_R$, $X$ be the special fiber, $x$ be
  the canonical coordinate, and $\partial$ be the corresponding
  differential operator. We put
  $\ms{M}=\DcompQ{0}{\ms{X}}/\DcompQ{0}{\ms{X}}(\partial-x)$. Then,
  \begin{enumerate}
   \item $\mr{Char}^{(0)}(\ms{M})=s(X)$, where $s\colon X\rightarrow
	 T^*X$ is the zero-section.
   \item $\mr{Char}^{(1)}(\DcompQ{1}{\ms{X}}\otimes_{\DcompQ{0}{\ms{X}}}
	 \ms{M})\cap\mathring{T}^*X\neq\emptyset$.
  \end{enumerate}
 \end{ex*}
 \begin{proof}
  Since $\ms{M}$ is a coherent $\mc{O}_{\ms{X},\mb{Q}}$-module, the first
  claim follows. Let us check the second claim. First,
  let us prove that $\DdagQ{\ms{X}}\otimes_{\DcompQ{0}{\ms{X}}}\ms{M}\neq
  0$. Let $f_n\in K\{x\}$ ({\it i.e.}\ the Tate algebra), and
  $\sum_{n\geq0}f_n\partial^{[n]}\in\Gamma(\ms{X},\DdagQ{\ms{X}})$. We
  get
  \begin{align*}
   \sum_{n\geq 0}f_n\,\partial^{[n]}\cdot(\partial-x)&=\sum_{n\geq 0}
   \Bigl(f_n\,\partial^{[n]}\,\partial-x\,f_n\,\partial^{[n]}-f_n\,
   \partial^{[n-1]}\Bigr)\\
   &=\sum_{n\geq 0}\bigl(n\,f_{n-1}-x\,f_n-f_{n+1}\bigr)\,
   \partial^{[n]}.
  \end{align*}
  Assume $\sum_{n\geq 0}f_n\,\partial^{[n]}\cdot(\partial-x)=1$.
  Then there exist $g_n,h_n\in K[x]$, $\deg(g_n)<n-1$ and $\deg(h_n)<n$,
  such that the equality
  \begin{equation*}
   (-1)^nf_n=(x^{n-1}+g_n)+(x^n+h_n)\cdot f_0
  \end{equation*}
  should hold for any $n>0$. However, there is no $f_0\in K\{x\}$ such
  that $\sum_{n\geq0}f_n\partial^{[n]}\in\Gamma(\ms{X},\DdagQ{\ms{X}})$
  (since $|f_n|=\max\bigl\{1,|f_0|\bigr\}$ by the equality), and
  $\DdagQ{\ms{X}}\otimes\ms{M}\neq 0$.

  Now, let $e$ be the element of $\Gamma(\ms{X},\ms{M})$
  defined by $1\in\Gamma(\ms{X},\DcompQ{0}{\ms{X}})$. As an
  $\mc{O}_{\ms{X},\mb{Q}}$-module, $\ms{M}$ is free of rank $1$. Since
  \begin{equation*}
   \partial^{n}\cdot e=\bigl(x^n+(\mbox{polynomial in $K[x]$ whose degree
    is less than $n$})\bigr)\cdot e,
  \end{equation*}
  the $\DcompQ{0}{\ms{X}}$-module structure on $\ms{M}$ does
  not extend continuously to a $\DcompQ{1}{\ms{X}}$-module
  structure. This shows that
  the canonical homomorphism $\ms{M}\rightarrow\ms{M}^{(1)}:=
  \DcompQ{1}{\ms{X}}\otimes\ms{M}$ is not an isomorphism.

  Garnier shows in \cite[5.2.4]{Ga} that for any coherent
  $\DcompQ{0}{\ms{X}}$-module $\ms{M}$, the characteristic variety
  $\mr{Char}^{(0)}(\ms{M})$ satisfies the Bernstein inequality ({\it
  i.e.\ }the dimension of the characteristic variety is greater than or
  equal to $1$ unless $\ms{M}=0$). Using the relation of characteristic
  varieties of Frobenius descents (cf.\ \cite[5.2.4 (iii)]{BerInt}), the
  Bernstein inequality also holds for any coherent
  $\DcompQ{m}{\ms{X}}$-module. Thus there are three possibilities
  for the characteristic variety $V$ of $\ms{M}^{(1)}$: either
  $\emptyset$ or $[X]$ or $V\cap\mathring{T}^*\ms{X}\neq\emptyset$. Since
  $\ms{M}^{(1)}$ is not $0$, $V$ is not empty. If $V=[X]$, $\ms{M}^{(1)}$
  would be a coherent $\mc{O}_{\ms{X},\mb{Q}}$-module, and
  since $\ms{M}$ is a coherent $\mc{O}_{\ms{X},\mb{Q}}$-module of rank
  $1$, we would get that $\ms{M}\cong\ms{M}^{(1)}$, which is a
  contradiction. Thus the second claim follows.
 \end{proof}

\subsection{}
\label{notationfixring}
Before going to the main theme of this section, let us fix some
frequently used notation.
Consider the situation where an open subscheme $\ms{U}$ of
$T^*\ms{X}$ is given. For non-negative integers $i'\geq i$, we put
\begin{alignat*}{3}
 D^{(i)}&:=\Gamma(\pi(\ms{U}),\Dmod{i}{\ms{X}}),&\qquad
 D_{\mb{Q}}^{(i)}&:=\Gamma(\pi(\ms{U}),\Dmod{i}{\ms{X},\mb{Q}}),
 &\qquad
 E^{(i)}&:=\Gamma(\ms{U},\Emod{i}{\ms{X}}),\\
 E_{\mb{Q}}^{(i)}&:=\Gamma(\ms{U},\Emod{i}{\ms{X},\mb{Q}}),
 &\qquad
 E^{(i,i')}&:=\Gamma(\ms{U},\Emod{i,i'}{\ms{X}}),&\qquad
 E_{\mb{Q}}^{(i,i')}&:=\Gamma(\ms{U},\Emod{i,i'}{\ms{X},\mb{Q}}),
 \\
 \widehat{E}^{(i,i')}&:=\Gamma(\ms{U},\Ecomp{i,i'}{\ms{X}}),
 &\qquad
  \widehat{E}_{\mb{Q}}^{(i,i')}&:=\Gamma(\ms{U},
 \EcompQ{i,i'}{\ms{X}}),&\qquad
 E^{(i,\dag)}_{\mb{Q}}&:=\Gamma(\ms{U},\Emod{i,\dag}
 {\ms{X},\mb{Q}}).
\end{alignat*}
The last 5 rings are defined in \ref{defofintermrings}.
The first 6 rings are considered to be filtered rings.

\subsection{}
Let $\ms{X}$ be a smooth formal scheme over $R$. For a non-negative
integer $m$, we defined the filtered ring
$(\Emod{m}{\ms{X}},\Emod{m}{\ms{X},n})$. By (\ref{relagr}) and Lemma
\ref{veryelementaryprop} (ii),
\begin{equation}
 \label{compgron}
  \mc{O}_{T^{(m)*}\ms{X}}(n)\cong\Emod{m}{\ms{X},n}/
  \Emod{m}{\ms{X},n-1}.
\end{equation}

Let $m'\geq m$ be an integer. Consider the canonical homomorphism
\begin{equation*}
 \phi_{m',m}\colon\Dmod{m}{\ms{X}}\rightarrow\Dmod{m'}{\ms{X}}.
\end{equation*}
This homomorphism becomes an isomorphism if we tensor
with $\mb{Q}$.

\begin{lem*}
\label{stricnessofhom}
 There exists a unique strictly injective homomorphism of  filtered
 rings
 \begin{equation*}
  \psi_{m,m'}\colon\Emod{m'}{\ms{X},\mb{Q}}\rightarrow
   \Emod{m}{\ms{X},\mb{Q}}
 \end{equation*}
 such that the following diagram is commutative:
 \begin{equation*}
  \xymatrix@C=60pt{
   \pi^{-1}\Dmod{m'}{\ms{X},\mb{Q}}\ar[d]_{\varphi_{m'}}&
   \pi^{-1}\Dmod{m}{\ms{X},\mb{Q}}\ar[d]^{\varphi_m}\ar[l]
   _{\phi_{m',m}\otimes\mb{Q}}^{\sim}\\
  \Emod{m'}{\ms{X},\mb{Q}}\ar[r]_{\psi_{m,m'}}&
   \Emod{m}{\ms{X},\mb{Q}}
   }
 \end{equation*}
 where we refer to \ref{limitnaive} for $\varphi_m$.
 For $m''\geq m'\geq m$,
 $\psi_{m,m'}\circ\psi_{m',m''}=\psi_{m,m''}$. By using
 {\normalfont(\ref{compgron})}, $\mr{gr}_n(\psi_{m,m'})$ can be
 identified with the $p$-isogeny in Lemma {\normalfont\ref{pisog}}
 locally.
\end{lem*}
\begin{proof}
 Once the existence and the uniqueness is proven, the compatibility
 $\psi_{m,m'}\circ\psi_{m',m''}=\psi_{m,m''}$ automatically holds by the
 compatibility of $\phi_{m',m}$.

 Let us prove the uniqueness first. Since the problem is local, we may
 assume that $\ms{X}$ possesses a system of local coordinates, and it
 suffices to show the uniqueness for the ring of sections over
 $D(\Theta)$ where $\Theta\in\Gamma(\ms{X},\mc{O}_{T^*\ms{X}})$ is a
 homogeneous element. Suppose there are two homomorphisms $\psi,\psi'$
 satisfying the condition. By the commutativity of the diagram,
 $\psi(\Thetatil{(m')})=\psi'(\Thetatil{(m')})=:\theta$. By the
 uniqueness of the inverse of $\theta$,
 $\psi((\Thetatil{(m')})^{-1})=\psi'((\Thetatil{(m')})^{-1})$.
 Since $\psi$ and $\psi'$ are filtered homomorphisms, these
 homomorphisms are continuous with respect to the topology defined by
 the filtrations (cf.\ \ref{termfiltsoon}). Let $E$ be the subring of
 $\Gamma(D(\Theta),\Emod{m'}{\ms{X}})$ generated by
 $\Gamma(D(\Theta),\Dmod{m'}{\ms{X}})$ and
 $(\Thetatil{(m')})^{-1}$. Then, $\psi|_E=\psi'|_E$. Since
 $\Gamma(D(\Theta),\Emod{m}{\ms{X},\mb{Q}})$
 is separated and $E$ is dense in $\Gamma(D(\Theta),\Emod{m'}{\ms{X}})$,
 we get $\psi=\psi'$, and the uniqueness follows.

 Now, let us check the existence. Since the problem is
 local by the uniqueness, we may suppose that $\ms{X}$ is affine.
 Let $\Theta$ be a homogeneous element of
 $\Gamma(\ms{X},\mc{O}_{T^*\ms{X}})$, and let $\ms{U}:=D(\Theta)$. It
 suffices to construct $\psi_{m,m'}$ on $\ms{U}:=D(\Theta)$. We use the
 notation of \ref{notationfixring}.
 Let $D^{(m')}_S$ be the microlocalization of
 $D^{(m')}$ by using the multiplicative set $S$ of $\mr{gr}(D^{(m')})$
 generated by $\Theta^{(m')}\in\mr{gr}(D^{(m')})$ (cf.\
 \ref{microlocaldef}), and let $(E^{(m)}_\mb{Q})'$ be the completion of
 $E^{(m)}_\mb{Q}$ with respect to the filtration by order.
 Since $\Thetatil{(m,m')}$ is invertible in $(E^{(m)}_\mb{Q})'$,
 $\Thetatil{(m')}$ is also invertible by Lemma
 \ref{moreexplicitdesc} (i). Thus, by the universal property of
 the microlocalization \cite[A.2.3.3]{Lau}, there exists a unique
 homomorphism of filtered rings $\alpha\colon D^{(m')}_S\rightarrow
 (E^{(m)}_\mb{Q})'$ factoring through the canonical homomorphism
 $D^{(m')}\rightarrow(E^{(m)}_{\mb{Q}})'$.
 For any $n$, there exists an integer $N$ such that the homomorphism
 $p^N\cdot\alpha_n$ induces a homomorphism
 $D^{(m')}_{S,n}\rightarrow E^{(m)}_n$ by the concrete description Lemma
 \ref{defofsituation} and Lemma \ref{normcalc} (i)-(b). Let
 $(D^{(m')}_{S,n})^\wedge$ be the $\pi$-adic completion of
 $D^{(m')}_{S,n}$. Since $E^{(m)}_n$ is $\pi$-adically complete, this
 induces the homomorphism $(p^N\cdot\alpha_n)^\wedge\colon
 (D^{(m')}_{S,n})^\wedge\rightarrow E^{(m)}_n$. We define
 \begin{equation*}
  \widehat{\alpha}_n:=p^{-N}\cdot(p^N\cdot\alpha_n)^\wedge\colon
   (D^{(m')}_{S,n})^\wedge\rightarrow E^{(m)}_n\otimes\mb{Q}.
 \end{equation*}
 By construction, we have
 $\widehat{\alpha}_{n+1}|_{D^\wedge_n}=\widehat{\alpha}_n$ where
 $D^\wedge_n:=(D^{(m')}_{S,n})^\wedge$.
 On the other hand, we have $\beta_n\colon
 (D^{(m')}_{S,n})^\wedge\xrightarrow{\sim}E^{(m')}_n$. Indeed, by Lemma
 \ref{compatibility},
 $D^{(m')}_{S,n}\otimes R_i\cong\Gamma(\ms{U},\Emod{m'}{X_i,n})$,
 and since $E^{(m')}_n$ is $\pi$-adically complete, the isomorphism
 follows. Thus, we obtain
 \begin{equation*}
  \indlim_n(\alpha_n\circ\beta_n^{-1})\colon
   E^{(m')}\rightarrow E^{(m)}_\mb{Q},
 \end{equation*}
 which is what we are looking for.

 Finally, let us check that $\psi_{m,m'}$ is strictly injective. By
 construction, locally, $\mr{gr}_n(\psi_{m,m'})$ coincides with the
 $p$-isogeny of Lemma \ref{pisog} for any $n$. This implies that the
 canonical homomorphism
 $\mr{gr}(\Emod{m'}{\ms{X}})\rightarrow\mr{gr}(\Emod{m}{\ms{X},\mb{Q}})$
 is injective. Since $\Emod{m'}{\ms{X}}$ is separated with respect to
 the filtration by order, we get the strict injectivity by \cite[Ch.I,
 4.2.4 (5)]{HO}.
\end{proof}

\subsection{}
We preserve the notation. For non-negative integers $m'\geq m$, we
define a sheaf of rings
\begin{equation*}
 \Emod{m,m'}{\ms{X}}:=\psi_{m,m'}^{-1}
  (\Emod{m}{\ms{X}})\cap\Emod{m'}{\ms{X}},
\end{equation*}
where the intersection is taken in
$\Emod{m'}{\ms{X},\mb{Q}}$. By definition,
$\Emod{m,m}{\ms{X}}=\Emod{m}{\ms{X}}$.
We denote
$\Emod{m,m'}{\ms{X}}\otimes R_i$ by $\Emod{m,m'}{X_i}$.
Let $\ms{U}$ be an open subset of $T^*\ms{X}$. Then the left exactness
of the functor $\Gamma$ implies that
\begin{equation*}
 \Gamma(\ms{U},\Emod{m,m'}{\ms{X}})\cong\psi_{m,m'}^{-1}
  (E^{(m)})\cap E^{(m')}\subset
  E^{(m')}_{\mb{Q}}
\end{equation*}
using the notation of \ref{notationfixring}.

Since $\psi_{m,m'}(\Emod{m'}{\ms{X}})$ and
$\Emod{m}{\ms{X}}$ are sub-$\pi^{-1}\Dmod{m}{\ms{X}}$-algebras of
$\Emod{m}{\ms{X},\mb{Q}}$, the ring $\Emod{m,m'}{\ms{X}}$ is also a
$\pi^{-1}\Dmod{m}{\ms{X}}$-algebra on $T^*\ms{X}$. Moreover, by putting
\begin{equation*}
 \Emod{m,m'}{\ms{X},n}:=\psi_{m,m'}^{-1}
  (\Emod{m}{\ms{X},n})\cap\Emod{m'}{\ms{X},n},
\end{equation*}
we may equip $\Emod{m,m'}{\ms{X}}$ with a filtration, and we consider
$\Emod{m,m'}{\ms{X}}$ as a filtered ring. By Lemma \ref{stricnessofhom},
$\psi_{m,m'}$ is a strict homomorphism, and the canonical homomorphisms
of filtered rings
\begin{equation*}
 (\Emod{m,m'}{\ms{X}},\Emod{m,m'}{\ms{X},n})\rightarrow
 (\Emod{m}{\ms{X}},\Emod{m}{\ms{X},n}),\qquad
 (\Emod{m,m'}{\ms{X}},\Emod{m,m'}{\ms{X},n})\rightarrow
 (\Emod{m'}{\ms{X}},\Emod{m'}{\ms{X},n})
\end{equation*}
are also strict injective homomorphisms. By the explicit presentation
Lemma \ref{defofsituation} and Lemma \ref{normcalc}
(i)-(b), $\psi_{m,m'}(\Emod{m'}{\ms{X},n})\subset\Emod{m}{\ms{X},n}$
for $n<p^{m+1}$, and in particular
\begin{equation}
 \label{negativecoince}
 \Emod{m,m'}{\ms{X},0}=\Emod{m'}{\ms{X},0}.
\end{equation}

\begin{lem}
 \label{inversecontE}
 Assume $\ms{X}$ to be affine, and let $\Theta$ be a homogeneous section
 of $\mc{O}_{T^*\ms{X}}$. We put $\ms{U}:=D(\Theta)$.
\begin{enumerate}
 \item For non-negative integers $M'\geq M\geq m'\geq m$,
       $(\Thetatil{(M,M')})^{-1}\in
       \Gamma(\ms{U},\Emod{m,m'}{\ms{X}})$.

 \item For non-negative integers $M'\geq M$ and $M'\geq m'\geq m$,
       $(\Thetatil{(M,M')})^{-1}\in
       \Gamma(\ms{U},\Emod{m,m'}{\ms{X},\mb{Q}})$.
\end{enumerate}
\end{lem}
\begin{proof}
 For any integer $m''$ such that $m''\geq m'\geq m$,
 $(\Thetatil{(m',m'')})^{-1}\in
 \Gamma(\ms{U},\Emod{m}{\ms{X}})$. Indeed, there exist a non-negative
 integer $n$ and a unit $u$ of $R$ such that
 $\Thetatil{(m',m'')}=u\,p^{-n}\cdot\Thetatil{(m,m'')}$
 by Lemma \ref{moreexplicitdesc},
 and thus
 \begin{equation*}
  (\Thetatil{(m',m'')})^{-1}=u^{-1}\,p^n\cdot
   (\Thetatil{(m,m'')})^{-1}\,\in
   \Gamma(\ms{U},\Emod{m}{\ms{X}}).
 \end{equation*}
 This yields the first claim.

 In turn, for any integer $m''$ such that
 $M'\geq m''$, we can check that
 $(\Thetatil{(M,M')})^{-1}\in
 \Gamma(\ms{U},\Emod{m''}{\ms{X},\mb{Q}})$, which implies
 the second claim.
\end{proof}

\subsection{}
\label{calcintse}
Let $\iota\colon\Emod{m}{\ms{X}}\rightarrow\Emod{m}{\ms{X},\mb{Q}}$ be
the canonical inclusion. Take an integer $n$. Consider the following
commutative diagram
\begin{equation}
 \label{exactgrcart}
 \xymatrix@C=40pt{
  0\ar[r]&\Emod{m,m'}{\ms{X},n-1}\ar[r]\ar@{^{(}->}[d]&
  \Emod{m'}{\ms{X},n-1}\oplus\Emod{m}{\ms{X},n-1}
  \ar@{^{(}->}[d]
  \ar[rr]^<>(.5){\alpha_{n-1}:=}
  _<>(.5){((\psi_{m,m'})_{n-1},\iota_{n-1})}&&
  (\Emod{m}{\ms{X},\mb{Q}})_{n-1}\ar@{^{(}->}[d]\\
 0\ar[r]&\Emod{m,m'}{\ms{X},n}\ar[r]&\Emod{m'}{\ms{X},n}
  \oplus\Emod{m}{\ms{X},n}
  \ar[rr]^<>(.5){\alpha_n:=}
  _<>(.5){((\psi_{m,m'})_{n},\iota_n)}
  &&(\Emod{m}{\ms{X},\mb{Q}})_n,
  }
\end{equation}
whose rows are exact sequences.

\begin{lem*}
 Let $\mr{inc}\colon(\Emod{m}{\ms{X},\mb{Q}})_{n-1}\rightarrow
 (\Emod{m}{\ms{X},\mb{Q}})_n$ be the canonical inclusion.
 The following sequence is exact:
 \begin{equation*}
  0\rightarrow\mr{Im}(\alpha_{n-1})
   \rightarrow
   (\Emod{m}{\ms{X},\mb{Q}})_{n-1}\oplus
   \mr{Im}(\alpha_n)
   \rightarrow(\Emod{m}{\ms{X},\mb{Q}})_{n},
 \end{equation*}
 where the second (resp.\ last) homomorphism is induced by
 $(\alpha_{n-1},\mr{inc})$ (resp.\ $\mr{inc}-\alpha_n$).
\end{lem*} 
\begin{proof}
 Since the statement is local, we may suppose that $\ms{X}$ is affine
 and possesses a system of local coordinates. Moreover, it suffices to
 show the exactness for the modules of sections
 over $\ms{U}:=D(\Theta)$ where $\Theta$ is a homogeneous section of
 $\Gamma(T^*\ms{X},\mc{O}_{T^*\ms{X}})$. We use the notation of
 \ref{notationfixring}. We also denote $\psi_{m,m'}$ by
 $\psi$ and $(\psi_{m,m'})_n$ by $(\psi)_n$ for short.

 An operator $P$ of $\DcompQ{m}{\ms{X}}$ is said to be {\em homogeneous
 of degree $l$} if we can write $P=\sum_{|\underline{k}|=l}
 a_{\underline{k}}\,\underline{\partial}^{\angles{m}{\underline{k}}}$
 with $a_{\underline{k}}\in\mc{O}_{\ms{X},\mb{Q}}$.
 Take $S\in\mr{Im}((\psi)_n,\iota_n)$.
 Using a left presentation, we can write
 \begin{equation*}
  S:=\psi\Bigl(\sum_{\substack{k\leq n\\i\in\mb{Z}}}P_{k,i}
   (\Thetatil{(m')})^{i}\Bigr)+\sum_{\substack{k\leq n\\
  i\in\mb{Z}}}Q_{k,i}(\Thetatil{(m,m')})^{i}
 \end{equation*}
 where the first sum is an element of $E_n^{(m')}$
 and the second sum is one of $E_n^{(m)}$, and $P_{k,i},Q_{k,i}$
 are homogeneous operators of degree $k-ip^{m'}\mr{ord}(\Theta)$
 in $\DcompQ{m'}{\ms{X}}$ and
 $\DcompQ{m}{\ms{X}}$ respectively with some convergence
 conditions. Suppose
 $S\in(E_{\mb{Q}}^{(m)})_{n-1}$. We need to show that this element is
 contained in $\mr{Im}(\alpha_{n-1})$.
 Since $\lim_{i\rightarrow\pm\infty}P_{n,i}=0$ in $\DcompQ{m'}{\ms{X}}$
 and by Corollary \ref{difflevelcoherent},
 there exists a {\em finite} subset $I\subset\mb{Z}$ such that
 \begin{equation*}
  \psi\Bigl(\sum_{i\not\in I}P_{n,i}(\Thetatil{(m')})^i\Bigr)
   \in E^{(m)}_n.
 \end{equation*}
 This is in fact contained in $E^{(m)}_n\cap\psi(E^{(m')}_n)$.
 Then for $N\gg 0$, there exists
 $a_{\underline{k}}\in\Gamma(\ms{X},\mc{O}_{\ms{X}})$
 for $|\underline{k}|=n+Np^{m'}\mr{ord}(\Theta)=:M$ and
 $\underline{k}\geq\underline{0}$ such that
 \begin{equation*}
  \sum_{i\in I}P_{n,i}(\Thetatil{(m')})^i
   \in\sum_{|\underline{k}|=M}a_{\underline{k}}\underline{\partial}
   ^{\angles{m'}{\underline{k}}}(\Thetatil{(m')})^{-N}
   +E^{(m')}_{n-1}+
   \bigl(\psi^{-1}(E^{(m)}_n)\cap E^{(m')}_n\bigr)
 \end{equation*}
 and $a_{\underline{k}}\underline{\partial}^{\angles{m'}
 {\underline{k}}}(\Thetatil{(m')})^{-N}
 \not\in\psi^{-1}(E^{(m)})$ for any
 $|\underline{k}|=M$ and $\underline{k}\geq\underline{0}$ such that
 $a_{\underline{k}}\neq0$. By
 the same argument for $E^{(m)}$ and increasing $N$ if necessary, we
 may also suppose that there exists
 $b_{\underline{k}}\in\Gamma(\ms{X},\mc{O}_{\ms{X}})$ for
 $|\underline{k}|=M$ such that
 $b_{\underline{k}}\underline{\partial}^{\angles{m}{\underline{k}}}
 (\Thetatil{(m,m')})^{-N}\not\in E^{(m')}$ for any
 $b_{\underline{k}}\neq0$, and
 \begin{align*}
  \sum_{i}Q_{n,i}(\Thetatil{(m,m')})^i
  \in\sum_{|\underline{k}|=M}b_{\underline{k}}\underline{\partial}
  ^{\angles{m}{\underline{k}}}(\Thetatil{(m,m')})^{-N}
  +E^{(m)}_{n-1}+(\psi(E^{(m')}_n)\cap E^{(m)}_n)
 \end{align*}
 in $E^{(m)}_{\mb{Q}}$. However, since $S\in(E_{\mb{Q}}^{(m)})_{n-1}$,
 \begin{equation}
  \label{inclresco}
  \sum_{|\underline{k}|=M}
  a_{\underline{k}}\underline{\partial}
  ^{\angles{m'}{\underline{k}}}(\Thetatil{(m')})^{-N}+
  b_{\underline{k}}\underline{\partial}
  ^{\angles{m}{\underline{k}}}(\Thetatil{(m,m')})^{-N}
  \in (E^{(m)}_\mb{Q})_{n-1}+(\psi(E^{(m')}_n)\cap E^{(m)}_n).
 \end{equation}
 The finite set
 $\bigl\{\underline{\xi}^{\underline{k}}\cdot(\Theta^{(m')})^{-N}
 \bigr\}_{|\underline{k}|=M}$ in $\mc{O}_{T^*\ms{X},\mb{Q}}(n)$ is
 linearly independent over $\mc{O}_{\ms{X},\mb{Q}}$.
 Thus, by the choice of $a_{\underline{k}}$ and $b_{\underline{k}}$,
 (\ref{inclresco}) is possible only when
 $a_{\underline{k}}=b_{\underline{k}}=0$, and the lemma is proven.
\end{proof}

\begin{cor*}
 \label{intersectioncompatible}
 We have
 \begin{equation*}
   \mr{gr}(\Emod{m,m'}{\ms{X}})=\mr{gr}(\psi_{m,m'})^{-1}
   \bigl(\mr{gr}(\Emod{m}{\ms{X}})\bigr)
   \cap\mr{gr}(\Emod{m'}{\ms{X}})
 \end{equation*}
 where the intersection is taken in the ring 
 \begin{equation*}
  \mr{gr}(\Emod{m}{\ms{X},\mb{Q}})\cong\bigoplus_{i\in\mb{Z}}
   \mc{O}_{T^*\ms{X},\mb{Q}}(i)=\mc{O}_{T^*\ms{X},\mb{Q}}(*).
 \end{equation*}
\end{cor*}
\begin{proof}
 The canonical homomorphism
 $\mr{Im}(\alpha_n)/\mr{Im}(\alpha_{n-1})\rightarrow
 (\Emod{m}{\ms{X},\mb{Q}})_n/(\Emod{m}{\ms{X},\mb{Q}})_{n-1}$ is
 injective by the lemma above. Thus, the corollary follows by
 (\ref{exactgrcart}).
\end{proof}

\begin{lem}
 \label{notensqincl}
 Let $m'\geq m$ be non-negative integers.
 We have an inclusion $\Emod{m-1,m'}{\ms{X}}\subset\Emod{m,m'}{\ms{X}}$
 in $\Emod{m'}{\ms{X},\mb{Q}}$.
 Moreover, there exists a unique strict injective homomorphisms of
 filtered $\pi^{-1}\Dmod{m-1}{\ms{X}}$-algebras
 $\alpha\colon\Emod{m,m'+1}{\ms{X}}\rightarrow\Emod{m,m'}{\ms{X}}$ such
 that $\psi_{m,m'}\circ\alpha=\psi_{m,m'+1}$.
\end{lem}
\begin{proof}
 Let us show the first claim.
 It suffices to show that $\psi_{m-1,m'}\colon\Emod{m-1,m'}{\ms{X}}
 \rightarrow\Emod{m-1}{\ms{X},\mb{Q}}$
 factors through the composition
 $\Emod{m,m'}{\ms{X}}\rightarrow\Emod{m}{\ms{X}}\rightarrow
 \Emod{m-1}{\ms{X},\mb{Q}}$. We consider $\Emod{m-1,m'}{\ms{X}}$ and
 $\Emod{m,m'}{\ms{X}}$ as subrings of $\Emod{m-1}{\ms{X},\mb{Q}}$ using
 these injections. We may assume $\ms{X}$ is affine, and let
 $\ms{U}:=D(\Theta)$ where $\Theta$ is a homogeneous section of
 $\Gamma(T^*\ms{X},\mc{O}_{T^*\ms{X}})$. Let us use the notation of
 \ref{notationfixring}.
 It suffices to show that $E^{(m-1,m')}$ is contained in $E^{(m,m')}$.
 Take $P\in (E^{(m-1,m')})_N$ for some integer $N$. We
 inductively define $P_i\in(E^{(m-1,m')})_{N-i}$ and
 $Q_i\in(E^{(m,m')})_{N-i}$ such that $P_{i+1}=P_i-Q_i$ for
 $i\geq0$. Put $P_0:=P$. Assume $P_i$ is constructed. We can write
 \begin{equation*}
  \sigma(P_i)=\sum_{|\underline{k}|=N-i}a_{\underline{k}}\,
   \underline{\xi}^{\angles{m-1}{\underline{k}}}\,
   (\Thetatil{(m-1,m')})^{-n}
 \end{equation*}
 with $a_{\underline{k}}\in\Gamma(\ms{X},\mc{O}_{\ms{X}})$ and
 $n\in\mb{Z}$. By Corollary \ref{intersectioncompatible}, this is
 contained in both $\mr{gr}(E^{(m-1)})$ and
 $\mr{gr}(E^{(m')})$. Thus, Lemma \ref{normcalc} (ii) is
 shows that
 \begin{equation*}
  Q_i:=\sum_{|\underline{k}|=N-i}a_{\underline{k}}\,
   \underline{\partial}^{\angles{m-1}{\underline{k}}}\,
   (\Thetatil{(m-1,m')})^{-n}
 \end{equation*}
 is in $E^{(m,m')}$. By construction, $P_{i+1}:=P_i-Q_i$ is
 contained in $(E^{(m-1,m')})_{N-i-1}$. The filtered ring
 $E^{(m,m')}$ is complete by (\ref{negativecoince}) and Lemma
 \ref{veryelementaryprop} (iv).
 Thus, $P=\sum_{i\geq0}Q_i\in E^{(m,m')}$. The second claim can
 be checked similarly, and left to the reader.
\end{proof}

\begin{lem}
 \label{intersectioncoherent}
 For any $n\in\mb{Z}$, $\mr{gr}_n(\Emod{m,m'}{\ms{X}})$
 is a coherent $\mc{O}_{T^{(m')*}\ms{X}}(0)$-module on
 $\mathring{T}^*\ms{X}$. Moreover, on $\mathring{T}^*\ms{X}$,
 $\bigoplus_{i\geq0}\mr{gr}_i(\Emod{m,m'}{\ms{X}})$ and
 $\mr{gr}(\Emod{m,m'}{\ms{X}})$ are
 $\mc{O}_{T^{(m')*}\ms{X}}(0)$-algebras of finite type, and they are
 noetherian.
\end{lem}
\begin{proof}
 The modules $\mr{gr}_n(\Emod{m}{\ms{X}})$ and
 $\mr{gr}_n(\Emod{m'}{\ms{X}})$ are coherent
 $\mc{O}_{T^{(m')*}\ms{X}}(0)$-modules on $\mathring{T}^*\ms{X}$ by
 Corollary \ref{difflevelcoherent} and (\ref{compgron}). Since
 \begin{equation*}
  \mr{gr}_n(\Emod{m}{\ms{X},\mb{Q}})\cong\mc{O}_{T^{(m)*}\ms{X},\mb{Q}}
   (n)=\sum_{i\geq 0} p^{-i}\mc{O}_{T^{(m)*}\ms{X}}(n),
 \end{equation*}
 there exists an integer $i_0$ such that
 $\mr{gr}_n(\Emod{m}{\ms{X}})+\mr{gr}_n(\Emod{m'}{\ms{X}})\subset
 p^{-i_0}\mc{O}_{T^{(m)*}\ms{X}}(n)$. Since the intersection of two
 coherent modules in a coherent module is coherent,
 $\mr{gr}_n(\Emod{m}{\ms{X}})\cap\mr{gr}_n(\Emod{m'}{\ms{X}})$ is also a
 coherent $\mc{O}_{T^{(m')*}\ms{X}}(0)$-module, and
 $\mr{gr}_n(\Emod{m,m'}{\ms{X}})$ is coherent as well by
 Corollary \ref{intersectioncompatible}.

 \begin{cl}
  Let $\mc{A}=\bigoplus_{i\in\mb{Z}}\mc{A}_i$ be a finitely generated
  $\mc{O}:=\mc{O}_{\mathring{T}^{(m)*}\ms{X}}(0)$-algebra such that
  $\mc{A}_i$ are coherent $\mc{O}$-modules for all $i$. Let
  $\mc{B}:=\bigoplus_{i\in\mb{Z}}\mc{B}_i$
  be a sub-$\mc{O}$-algebra of $\mc{A}$ such that
  $\mc{B}_i\subset\mc{A}_i$ and $\mc{B}_i$ are coherent $\mc{O}$-modules
  for all $i$. Assume $\mc{A}$ is finite over $\mc{B}$. Then $\mc{B}$ is
  a noetherian ring and finitely generated over $\mc{O}$.
 \end{cl}
 The proof is similar to that of Lemma \ref{noethgrlem} using
 \cite[7.8]{AtMac}.
 For any integer $m''\geq0$,
 $\bigoplus_{i\geq0}\mr{gr}_i(\Emod{m''}{\ms{X}})$ is finitely
 generated $\mc{O}_{T^{(m'')*}\ms{X}}(0)$-algebras by Lemma
 \ref{limitnaive}, $\mc{O}_{T^{(m)*}\ms{X}}(0)$ is a coherent
 $\mc{O}_{T^{(m')}\ms{X}}(0)$-module on $\mathring{T}^*\ms{X}$ by
 Corollary \ref{difflevelcoherent}, and
 $\bigoplus_{i\geq0}\mr{gr}_i(\Emod{m}{\ms{X}})$ is finite over
 $\bigoplus_{i\geq0}\mr{gr}_i(\Emod{m,m'}{\ms{X}})$ by Lemma
 \ref{normcalc} (i)-(a). Using the claim, this shows that
 $\bigoplus_{i\geq0}\mr{gr}_i(\Emod{m,m'}{\ms{X}})$ is an
 $\mc{O}_{T^{(m')*}\ms{X}}(0)$-algebra of finite type and
 noetherian.
 Now, since $\bigoplus_{i\leq0}\mr{gr}_i(\Emod{m,m'}{\ms{X}})
 \cong\bigoplus_{i\leq0}\mr{gr}_i(\Emod{m'}{\ms{X}})$ is an
 $\mc{O}_{T^{(m')*}\ms{X}}(0)$-algebra of finite type,
 $\mr{gr}(\Emod{m,m'}{\ms{X}})$ is finitely generated as well. Applying
 the claim again to both $\mc{A}$ and $\mc{B}$ being equal to
 $\mr{gr}(\Emod{m,m'}{\ms{X}})$, we conclude
 that the ring is noetherian.
\end{proof}

\begin{prop}
 \label{noetherianintermediate}
 (i) The filtered ring $(\Emod{m,m'}{\ms{X}},\Emod{m,m'}{\ms{X},n})$ is
 complete.

 (ii) The rings $\Emod{m,m'}{\ms{X},0}$ and $\Emod{m,m'}{\ms{X}}$ are
 noetherian on $\mathring{T}^*\ms{X}$.

 (iii) For open subsets $\mathring{T}^*\ms{X}\supset\ms{U}\supset\ms{V}$
 in $\mf{B}$, the restriction homomorphism
 $\Gamma(\ms{U},\ms{E})\rightarrow\Gamma(\ms{V},\ms{E})$
 is flat. Here $\ms{E}$ denotes $\Emod{m,m'}{\ms{X},0}$ or
 $\Emod{m,m'}{\ms{X}}$.
\end{prop}
\begin{proof}
 We have already used (i) but we rewrite the statement because of the
 importance. This follows from the fact that
 $\Emod{m,m'}{\ms{X},n}=\Emod{m'}{\ms{X},n}$ for any $n<p^{m+1}$ by
 (\ref{negativecoince}) and Lemma \ref{veryelementaryprop} (iv). We
 also get (ii) and (iii) for $\ms{E}=\Emod{m,m'}{\ms{X},0}$ by
 Proposition \ref{noetheriansheafofmic}.

 Let us prove (ii) for $\Emod{m,m'}{\ms{X}}$. Let us check
 the conditions of Lemma \ref{noetherianlemma} for the filtered ring
 $(\Emod{m,m'}{\ms{X}},\Emod{m,m'}{\ms{X},n})$. The first condition is
 nothing but (i). The second and the third conditions follow from Lemma
 \ref{noethgrlem} together with Lemma
 \ref{intersectioncoherent}. The last condition follows from Lemma
 \ref{veryelementaryprop} (ii) and
 Corollary \ref{intersectioncompatible}. Hence $\Emod{m,m'}{\ms{X}}$ is
 a noetherian ring. Thus, for any open subscheme $\ms{U}$ in $\mf{B}$,
 $\Gamma(\ms{U},\Emod{m,m'}{\ms{X}})$ is noetherian, and (iii) holds by
 using \cite[Ch.II, 1.2.1]{HO}.
\end{proof}

\begin{rem*}
 By the proof, we can moreover say that $\Emod{m,m'}{\ms{X}}$ is
 pointwise Zariskian with respect to the filtration by order on
 $\mathring{T}^*\ms{X}$. This
 implies that $\Emod{m,m'}{X_i}$ is also pointwise Zariskian with
 respect to the filtration by order for any integer $i\geq0$ on
 $\mathring{T}^*\ms{X}$.
\end{rem*}

\begin{lem}
\label{calcofcoh}
 Let $\ms{X}$ be an affine smooth formal scheme. Let
 $\ms{U}\subset\mathring{T}^*\ms{X}$ be an open set in $\mf{B}$,
 and $\mf{U}$ be a finite $\mf{B}$-covering ({\it i.e.}\ a covering
 consisting of subsets in $\mf{B}$) of $\ms{U}$. Let $\ms{E}$ be
 either $\Emod{m,m'}{\ms{X}}$ or $\Emod{m,m'}{\ms{X},\mb{Q}}$. Then
 $\check{H}_{\mr{aug}}^i(\mf{U},\ms{E})=0$ for $i\in\mb{Z}$. Here
 $\check{H}_{\mr{aug}}^i$ denotes the augmented \v{C}ech cohomology
 {\normalfont(cf.\ {\normalfont\cite[8.1.3]{BGR}})}. In particular,
 $H^1(\ms{U},\ms{E})=0$.
\end{lem}
\begin{proof}
 Let $V,W\in\mf{B}$. Since $\epsilon_\cdot(V\cap
 W)=\epsilon_\cdot(V)\cap\epsilon_\cdot(W)$, we may assume that $\ms{U}$
 is strictly affine and $\mf{U}$ is a finite strictly affine covering
 (cf.\ Definition \ref{limitnaive}).
 Let us show that
 $\check{H}^i_{\mr{aug}}(\mf{U},\Emod{m,m'}{\ms{X},k})=0$ for any $k$.
 By Lemma \ref{intersectioncoherent},
 $\Emod{m,m'}{\ms{X},k}/\Emod{m,m'}{\ms{X},-n}$ is a coherent
 $\mc{O}_{T^{(m')*}\ms{X}}(0)$-module. Thus,
 \begin{equation*}
  \check{H}^i_{\mr{aug}}(\mf{U},\Emod{m,m'}{\ms{X},k}/
   \Emod{m,m'}{\ms{X},-n})=0
 \end{equation*}
 for $i\in\mb{Z}$.
 By the coherence, the projective system
 $\Bigl\{\Gamma(\ms{V},\Emod{m,m'}{\ms{X},k}/\Emod{m,m'}{\ms{X},-n})
 \Bigr\}_{n\geq0}$ satisfies the Mittag-Leffler condition for any
 strictly affine open subset $\ms{V}$.
 This shows that
 $C^q_{\mr{aug}}(\mf{U},\Emod{m,m'}{\ms{X},k}/\Emod{m,m'}
 {\ms{X},-n})$ satisfies the Mittag-Leffler condition for any
 $q\in\mb{Z}$. Thus
 \begin{equation*}
  \check{H}^i_{\mr{aug}}(\mf{U},\Emod{m,m'}{\ms{X},k})\cong
   \check{H}^i_{\mr{aug}}(\mf{U},\invlim\Emod{m,m'}{\ms{X},k}/
   \Emod{m,m'}{\ms{X},-n})\cong\invlim\check{H}^i_{\mr{aug}}(\mf{U},
   \Emod{m,m'}{\ms{X},k}/\Emod{m,m'}{\ms{X},-n})=0
 \end{equation*}
 for $i\in\mb{Z}$, where the first equality holds since
 $\Emod{m,m'}{\ms{X},k}$ is complete by Proposition
 \ref{noetherianintermediate} (i). Thus, we get what we wanted.

 Now, since $\Emod{m,m'}{\ms{X}}\cong\indlim_{k}\Emod{m,m'}{\ms{X},k}$,
 we have $\check{H}^i_{\mr{aug}}(\mf{U},\Emod{m,m'}{\ms{X}})=0$
 by using (\ref{commindga}). The vanishing for
 $\Emod{m,m'}{\ms{X},\mb{Q}}$ follows immediately from the
 $\ms{E}=\Emod{m,m'}{\ms{X}}$
 case. Finally, let us prove $H^1(\ms{U},\ms{E})=0$. We know that
 $H^1(\ms{U},\ms{E})\cong\indlim_{\mf{V}}\check{H}^1(\mf{V},\ms{E})$
 where $\mf{V}$ runs over open coverings of $\ms{U}$. Given an open
 covering $\mf{V}$ of $\ms{U}$, there exists a refinement $\mf{V}'$
 which is a $\mf{B}$-covering. Thus the statement follows from the first
 claim.
\end{proof}

\begin{rem*}
 We may also prove that $H^i(\ms{U},\ms{E})=0$ for $i>0$. This can be
 proven in the same way by using \cite[$0_{\mr{III}}$ 13.3.1]{EGA}.
\end{rem*}

\subsection{}
\label{defofintermrings}
We define
\begin{equation*}
 \Ecomp{m,m'}{\ms{X}}:=\invlim_{i}\Emod{m,m'}{X_i},\qquad
 \EcompQ{m,m'}{\ms{X}}:=\Ecomp{m,m'}{\ms{X}}\otimes\mb{Q}.
\end{equation*}
These are $\pi^{-1}\Dcomp{m}{\ms{X}}$-algebras and the latter is
moreover a $\pi^{-1}\DcompQ{m}{\ms{X}}$-algebra.
We call $\Ecomp{m,m'}{\ms{X}}$ and $\EcompQ{m,m'}{\ms{X}}$
the {\em intermediate rings of microdifferential
operators of level $(m,m')$}. Let $\ms{U}$  be an open subset of
$\mathring{T}^*\ms{X}$ in $\mf{B}$. Applying Lemma
\ref{calcofcoh} to the exact sequence
\begin{equation*}
 0\rightarrow\Emod{m,m'}{\ms{X}}\xrightarrow{\pi^{i+1}}\Emod{m,m'}
  {\ms{X}}\rightarrow\Emod{m,m'}{X_i}\rightarrow0,
\end{equation*}
we get an isomorphism
\begin{equation}
 \label{redintmicsh}
 \Gamma(\ms{U},\Emod{m,m'}{X_i})\cong
  \Gamma(\ms{U},\Emod{m,m'}{\ms{X}})\otimes R_i,
\end{equation}
and by taking the projective limit over $i$, we have
$\Gamma(\ms{U},\Ecomp{m,m'}{\ms{X}})\cong
\Gamma(\ms{U},\Emod{m,m'}{\ms{X}})^{\wedge}$,
where $^{\wedge}$ denotes the $\pi$-adic completion. By Lemma
\ref{completion}, $\Gamma(\ms{U},\Ecomp{m,m'}{\ms{X}})\otimes R_i
\cong\Gamma(\ms{U},\Emod{m,m'}{X_i})$, and thus
$\Ecomp{m,m'}{\ms{X}}\otimes R_i\cong\Emod{m,m'}{X_i}$, in particular,
$\Ecomp{m,m'}{\ms{X}}$ is $\pi$-adically complete. We also define
\begin{equation*}
 \Emod{m,\dag}{\ms{X},\mb{Q}}:=\invlim_{m'\rightarrow\infty}
  \EcompQ{m,m'}{\ms{X}}.
\end{equation*}
This is a ring on $T^*\ms{X}$. Note that there exists a canonical
homomorphism $\Emod{m,\dag}{\ms{X},\mb{Q}}
\rightarrow\Emod{m+1,\dag}{\ms{X},\mb{Q}}$
of rings by Lemma \ref{notensqincl}. We define
\begin{equation*}
 \EdagQ{\ms{X}}:=\indlim_{m\rightarrow\infty}
  \Emod{m,\dag}{\ms{X},\mb{Q}}.
\end{equation*}
For a quasi-compact open subscheme $\ms{U}$ of $\mathring{T}^*\ms{X}$,
(\ref{commindga}) shows
\begin{equation*}
 \Gamma(\ms{U},\EdagQ{\ms{X}})\cong\indlim_{m'\rightarrow\infty}
  \Gamma(\ms{U},\Emod{m',\dag}{\ms{X},\mb{Q}}).
\end{equation*}

\begin{prop}
\label{statementcomplete}
 The rings $\Ecomp{m,m'}{\ms{X}}$, $\EcompQ{m,m'}{\ms{X}}$
 are noetherian on $\mathring{T}^*\ms{X}$, and Proposition
 {\normalfont\ref{noetherianintermediate} (iii)} and Lemma
 {\normalfont\ref{calcofcoh}} are also valid if we take $\ms{E}$ to be
 either $\Ecomp{m,m'}{\ms{X}}$ or $\EcompQ{m,m'}{\ms{X}}$.
\end{prop}
\begin{proof}
 It suffices to show the proposition for
 $\ms{E}=\Ecomp{m,m'}{\ms{X}}$. First, let us check Lemma \ref{calcofcoh}
 for this $\ms{E}$. Since $\Emod{m,m'}{\ms{X}}$ is $\pi$-torsion free, we have
 $\check{H}^i_{\mr{aug}}(\mf{U},\Emod{m,m'}{X_i})=0$
 by (\ref{redintmicsh}) and Lemma \ref{calcofcoh}. The projective system
 $\bigl\{\Gamma(\ms{V},\Emod{m,m'}{X_i})\bigr\}_{i\geq0}$
 satisfies the Mittag-Leffler condition for any $\ms{V}\in\mf{B}$ again
 by (\ref{redintmicsh}), and
 we conclude that
 \begin{equation*}
  \check{H}^i_{\mr{aug}}(\mf{U},\Ecomp{m,m'}{\ms{X}})
   \cong\check{H}^i_{\mr{aug}}(\mf{U},\invlim_i\,
   \Emod{m,m'}{X_i})\cong\invlim_i\,\check{H}^i_{\mr{aug}}
   (\mf{U},\Emod{m,m'}{X_i})=0.
 \end{equation*}

 To prove that $\Ecomp{m,m'}{\ms{X}}$ is noetherian, we check the
 conditions of Lemma \ref{noetherianlemma} for the $\pi$-adic
 filtration. Conditions 2 and 3 follow from the fact that
 $\Emod{m,m'}{\ms{X}}$ is noetherian by Proposition
 \ref{noetherianintermediate}. Condition 4 follows from
 (\ref{redintmicsh}).
 Proposition \ref{noetherianintermediate} (iii) for
 $\ms{E}=\Ecomp{m,m'}{\ms{X}}$ follows directly from the
 $\ms{E}=\Emod{m,m'}{\ms{X}}$ case by using \cite[3.2.3 (vii)]{Ber1},
 and we finish the proof.
\end{proof}

\begin{rem*}
 By the proof we can moreover say that $\Ecomp{m,m'}{\ms{X}}$ is
 pointwise Zariskian with respect to the $\pi$-adic filtration on
 $\mathring{T}^*\ms{X}$.
\end{rem*}

\begin{lem}
 \label{explicitdescription}
 For non-negative integers $m'\geq m$, the canonical homomorphism
 $\EcompQ{m,m'}{\ms{X}}\rightarrow\EcompQ{m}{\ms{X}}$ is injective, and
 the canonical homomorphism
 $\Ecomp{m,m'}{\ms{X}}/\Emod{m,m'}{\ms{X},0}\rightarrow
 \Ecomp{m}{\ms{X}}/\Emod{m}{\ms{X},0}$ is a $p$-isogeny.
\end{lem}
\begin{proof}
 We omit subscripts $\ms{X}$ ({\it e.g.}\ $\Emod{m}{0}$ instead of
 $\Emod{m}{\ms{X},0}$). Consider the following diagram whose rows are
 exact:
 \begin{equation*}
  \xymatrix{0\ar[r]&\Emod{m,m'}{0}\ar[r]\ar[d]&
   \Emod{m,m'}{}\ar[r]\ar[d]&\Emod{m,m'}{}/\Emod{m,m'}{0}
   \ar[r]\ar[d]^{\gamma}&0\\
  0\ar[r]&\Emod{m}{0}\ar[r]&
   \Emod{m}{}\ar[r]&\Emod{m}{}/\Emod{m}{0}\ar[r]&0.}
 \end{equation*}
 Since the injective homomorphism $\Emod{m,m'}{}\rightarrow\Emod{m}{}$
 is strict, $\gamma$ is injective as well. Moreover, by Lemma
 \ref{normcalc} (i)-(a), there exists an integer $a$ such that
 $\mr{Coker}(\gamma)$ is killed by $p^a$.  Since
 $\Emod{m,m'}{}/\Emod{m,m'}{0}$ is $\pi$-torsion free, this implies that
 $\gamma$ is a $p$-isogeny.

 Thus we get the
 following commutative diagram, whose rows are exact
 \begin{equation*}
  \xymatrix{0\ar[r]&(\Emod{m,m'}{0})^\wedge\ar[r]\ar[d]&
   \Ecomp{m,m'}{}\ar[r]\ar[d]&
   \Ecomp{m,m'}{}/\Ecomp{m,m'}{0}\ar[r]\ar[d]^
   {\widehat{\gamma}}&0\\
  0\ar[r]&(\Ecomp{m}{0})^\wedge\ar[r]&
   \Ecomp{m}{}\ar[r]&\Ecomp{m}{}/\Ecomp{m}{0}\ar[r]&0,}
 \end{equation*}
 where $^{\wedge}$ denotes the $\pi$-adic completion. By
 \ref{complofisog}, $\widehat{\gamma}$ is also a $p$-isogeny, and in
 particular, it is an isomorphism after tensoring with $\mb{Q}$.
 Since $\Emod{m}{0}$ and $\Emod{m,m'}{0}$ are already complete with
 respect to the $\pi$-adic topology by Lemma \ref{veryelementaryprop} (i)
 and $\Emod{m,m'}{0}\rightarrow\Emod{m}{0}$ is injective, the left
 vertical homomorphism is injective as well.
\end{proof}

\subsection{}
\label{concrdescrem}
Let the situation be as in \ref{defofsituation}. Let $m''\geq
m'\geq m$ be non-negative integers. We have the following concrete
description:
\begin{align*}
 \mr{Im}\bigl(\Gamma(\ms{U}&,\EcompQ{m,m'}{\ms{X}})\rightarrow
 \Gamma(\ms{U},\EcompQ{m}{\ms{X}})\bigr)\notag\\
 &=\Biggl\{\sum_{|\underline{k}|-inp^{m''}<0}a_{\underline{k},i}
 \underline{\partial}
 ^{\angles{m'}{\underline{k}}}(\widetilde{\Theta}_{\mr{le}}
 ^{(m',m'')})^{-i}+
 \sum_{|\underline{k}|-inp^{m''}\geq 0}a_{\underline{k},i}
 \underline{\partial}^
 {\angles{m}{\underline{k}}}
 (\widetilde{\Theta}_{\mr{le}}^{(m,m'')})^{-i}\,
 \Big|\,\mbox{$(*)$}\Biggr\}.
\end{align*}
\begin{quote}
 $(*)$ Let $\underline{k}\in\mb{N}^d$, $i\geq0$,
 $a_{\underline{k},i}\in\Gamma(\ms{X},\mc{O}_{\ms{X},\mb{Q}})$. For
 an integer $N$, put
 $\alpha_{N,i}:=\sup_{|\underline{k}|=inp^{m''}+N}|a_{\underline{k},i}|$.
 Then $\lim_{i\rightarrow\infty}\alpha_{N,i}=0$ for any $N$,
 $\lim_{N\rightarrow\infty}\sup_i\{\alpha_{N,i}\}=0$, and there exists
 a real number $C>0$ such that $C>\sup_i\{\alpha_{N,i}\}$ for any
 $N<0$.
\end{quote}
Moreover, consider the situation of \ref{uniqrepisposs}. The
homomorphism $\phi$ induces
$\Gamma(D(\Theta),\mc{O}_{T^{(m)*}\ms{X}}(*)\cap
\mc{O}_{T^{(m')*}\ms{X}}(*))\rightarrow
\Gamma(D(\Theta),\Ecomp{m,m'}{\ms{X}})$, where the intersection is taken
in $\mc{O}_{T^*\ms{X},\mb{Q}}(*)$. This homomorphism is abusively
denoted by $\phi$. Let us denote by $|\cdot|_{(i)}$ the norm of
$\mc{O}_{T^*\ms{X},\mb{Q}}(*)$ defined by the $p$-adic norm on
$\mc{O}_{T^{(i)*}\ms{X}}(*)$. Then any element of the above image can
be written {\em uniquely} as
\begin{equation}
 \label{intermuniqrep}
  \sum_{k\in\mb{Z}}\phi(P_k)
\end{equation}
where $P_{k}\in\Gamma(\ms{U},\mc{O}_{T^*\ms{X},\mb{Q}}(k))$
such that $\lim_{k\rightarrow\infty}|P_k|_{(m)}=0$ and there exists a
real number $C$ such that $|P_k|_{(m')}<C$ for any $k<0$.

\begin{lem}
 \label{keyisom2}
 Let $m'>m$ be non-negative integers. The homomorphism
 $\EcompQ{m,m'}{\ms{X}}\rightarrow\EcompQ{m+1,m'}{\ms{X}}$ is
 injective, and induces an isomorphism
 \begin{equation*}
  \EcompQ{m,m'}{\ms{X}}/\EcompQ{m,m'+1}{\ms{X}}\xrightarrow{\sim}
   \EcompQ{m+1,m'}{\ms{X}}/\EcompQ{m+1,m'+1}{\ms{X}}.
 \end{equation*}
\end{lem}
\begin{proof}
 Let us check the injectivity. Since the verification is local, we may
 assume that we are in the situation of
 \ref{defofsituation}. Since the restriction homomorphisms of
 $\mc{O}_{T^*\ms{X}}(*)$ are injective, those of
 $\EcompQ{m,m'}{\ms{X}}$ are injective as well,
 and the verification is generic.
 Thus, we may assume that we are in the situation of
 \ref{uniqrepisposs}, and we can use the concrete description
 (\ref{intermuniqrep}). Since the description is unique, the injectivity
 follows.

 For the second claim, it suffices to show that the
 canonical homomorphism
 \begin{equation*}
  (\Emod{m,m'}{\ms{X},\mb{Q}})_0/(\Emod{m,m'+1}{\ms{X},\mb{Q}})_0
   \rightarrow\EcompQ{m,m'}{\ms{X}}/\EcompQ{m,m'+1}{\ms{X}}
 \end{equation*}
is an isomorphism. This follows from Lemma \ref{explicitdescription}.
\end{proof}

\begin{lem}
 \label{keyisom}
 Let $m'>m\geq0$ be integers. Then the canonical injection
 $\EcompQ{m+1,m'+1}{\ms{X}}\rightarrow\EcompQ{m+1,m'}{\ms{X}}$ induces the
 isomorphism:
 \begin{equation*}
  \EcompQ{m+1,m'+1}{\ms{X}}/\EcompQ{m,m'+1}{\ms{X}}\xrightarrow{\sim}
   \EcompQ{m+1,m'}{\ms{X}}/\EcompQ{m,m'}{\ms{X}}.
 \end{equation*}
\end{lem}
\begin{proof}
 We have the following diagram whose rows are exact.
 \begin{equation*}
  \def\objectstyle{\scriptstyle}
  \xymatrix{
   0\ar[r]&\EcompQ{m,m'+1}{\ms{X}}/(\Emod{m,m'+1}{\ms{X,\mb{Q}}})_0
   \ar[r]\ar[d]&\EcompQ{m+1,m'+1}{\ms{X}}/(\Emod{m+1,m'+1}
   {\ms{X,\mb{Q}}})_0\ar[r]\ar[d]&\EcompQ{m+1,m'+1}{\ms{X}}/
   \EcompQ{m,m'+1}{\ms{X}}\ar[r]\ar[d]&0\\
  0\ar[r]&\EcompQ{m,m'}{\ms{X}}/(\Emod{m,m'}{\ms{X,\mb{Q}}})_0\ar[r]
   &\EcompQ{m+1,m'}{\ms{X}}/(\Emod{m+1,m'}{\ms{X,\mb{Q}}})_0\ar[r]
   &\EcompQ{m+1,m'}{\ms{X}}/\EcompQ{m,m'}{\ms{X}}\ar[r]&0
   }
 \end{equation*}
 The first two vertical homomorphisms are isomorphisms by Lemma
 \ref{explicitdescription}. Thus the right vertical homomorphism is an
 isomorphism as well, and the corollary follows.
\end{proof}

By Lemma \ref{explicitdescription} and Lemma \ref{keyisom2}, we have the
following big diagram of rings of microdifferential operators.

\begin{equation*}
 \xymatrix@R=10pt{
  \pi^{-1}\DcompQ{m}{\ms{X}}\ar@{}[r]|{\subset}\ar@{}[d]|{\cap}&
  \pi^{-1}\DcompQ{m+1}{\ms{X}}\ar@{}[r]|{\subset}\ar@{}[d]|{\cap}&
  \pi^{-1}\DcompQ{m+2}{\ms{X}}\ar@{}[r]|{\subset}\ar@{}[d]|{\cap}&
  \dots\ar@{}[r]|{\subset}&\pi^{-1}\DdagQ{\ms{X}}\ar@{}[d]|{\cap}\\
 \EmodQ{m,\dag}{\ms{X}}\ar@{}[r]|{\subset}\ar@{}[d]|{\cap}&
  \EmodQ{m+1,\dag}{\ms{X}}\ar@{}[r]|{\subset}\ar@{}[d]|{\cap}
  &\EmodQ{m+2,\dag}{\ms{X}}\ar@{}[r]|{\subset}\ar@{}[d]|{\cap}
  &\dots\ar@{}[r]|{\subset}&\EdagQ{\ms{X}}\\
 \vdots\ar@{}[d]|{\cap}&\vdots\ar@{}[d]|{\cap}&
  \vdots\ar@{}[d]|{\cap}&\iddots&\\
 \EcompQ{m,m+2}{\ms{X}}\ar@{}[r]|{\subset}\ar@{}[d]|{\cap}
  &\EcompQ{m+1,m+2}{\ms{X}}\ar@{}[r]|{\subset}\ar@{}[d]|{\cap}
  &\EcompQ{m+2}{\ms{X}}&&\\
 \EcompQ{m,m+1}{\ms{X}}\ar@{}[r]|{\subset}\ar@{}[d]|{\cap}
  &\EcompQ{m+1}{\ms{X}}&&&\\
 \EcompQ{m}{\ms{X}}&&&&
  }
\end{equation*}

\section{Flatness results}
\subsection{}
\label{defassocish}
Let $\ms{X}$ be an affine smooth formal scheme, and $\ms{U}$ be a
strictly affine open subscheme of $\mathring{T}^*\ms{X}$.
Let $\ms{E}$ be one of the rings $\Emod{m,m'}{\ms{X}}$,
$\Emod{m,m'}{\ms{X},\mb{Q}}$, $\Ecomp{m,m'}{\ms{X}}$,
$\EcompQ{m,m'}{\ms{X}}$. Let $M$ be a finite
$\Gamma(\ms{U},\ms{E})$-module. We use the terminologies in \cite[9.1,
9.2]{BGR} freely. Let $\mf{T}$ be the Grothendieck topology (in the
sense of \cite[9.1.1/1]{BGR}) on $\ms{U}$ defined in the following way.
\begin{itemize}
 \item A subset is said to be {\it admissible open} if it is strictly
       affine open subset of $\ms{U}$.
 \item A covering is called an {\it admissible covering} if it is open
       covering in the usual sense.
\end{itemize} 
We define a presheaf $M^{\triangle}$ on $(\ms{U},\mf{T})$ by associating
$\Gamma(\ms{V},\ms{E})\otimes_{\Gamma(\ms{U},\ms{E})}M$
with an strictly affine open subscheme $\ms{V}$.

\begin{lem}
\label{calcofcohmod}
 For any finite $\mf{B}$-covering $\mf{U}$ of $\ms{U}$,
 $\check{H}^i_{\mr{aug}}(\mf{U},M^{\triangle})=0$ for $i\in\mb{Z}$.
\end{lem}
\begin{proof}
 We just copy the proof of \cite[8.2.1/5]{BGR} using Lemma
 \ref{calcofcoh} and Lemma \ref{statementcomplete}.
\end{proof}

\begin{cor}
 \label{exactnessofassociation}
 (i) For any finite $\Gamma(\ms{U},\ms{E})$-module $M$, the presheaf
 $M^{\triangle}$ defines a sheaf on $(\ms{U},\mf{T})$, and the functor
 ${}^\triangle$ is exact.

 (ii) Let $\ms{U}$ be a strictly affine open subscheme of
 $\mathring{T}^*\ms{X}$, and suppose there exists a finite presentation
 on $\ms{U}$:
 \begin{equation*}
  (\ms{E}|_{\ms{U}})^{\oplus a}\xrightarrow{\phi}
   (\ms{E}|_{\ms{U}})^{\oplus b}
   \rightarrow\ms{M}\rightarrow 0.
 \end{equation*}
 Then we have a canonical isomorphism
 $\Gamma(\ms{U},\ms{M})^{\triangle}\xrightarrow{\sim}\ms{M}$.
\end{cor}
\begin{proof}
 Let us prove (i). Lemma \ref{calcofcohmod} shows that the
 presheaf $M^{\triangle}$ is a sheaf. The functor ${}^\triangle$ is
 exact since the restriction homomorphism
 $\Gamma(\ms{V},\ms{E})\rightarrow\Gamma(\ms{W},\ms{E})$ is flat where
 $\ms{W}\subset\ms{V}\subset\ms{U}$ are strictly affine by Proposition
 \ref{noetherianintermediate} and \ref{statementcomplete}.
 Let us show (ii). We put $M:=\mr{Coker}(\Gamma(\ms{U},\phi))$. Let
 $E:=\Gamma(\ms{U},\ms{E})$. By the definition of $M$, we have the
 following exact sequence
 \begin{equation*}
  E^{\oplus a}\xrightarrow{\Gamma(\ms{U},\phi)}E^{\oplus b}
  \rightarrow M\rightarrow0.
 \end{equation*}
 Taking the exact functor ${}^\triangle$, we have an isomorphism
 $\ms{M}\cong M^{\triangle}$. Taking the global sections,
 $\Gamma(\ms{U},\ms{M})\cong M$, and the claim follows.
\end{proof}

\begin{rem*}
 We did not prove that any coherent $\ms{E}$-module on $\ms{U}$ can be
 written as $M^{\triangle}$ with a finite $\Gamma(\ms{U},\ms{E})$-module
 $M$. We do not go into the problem further in this paper. We believe,
 however, that for any coherent $\ms{E}$-module $\ms{M}$ on a strict
 affine open subscheme $\ms{U}$, the canonical homomorphism
 $\Gamma(\ms{U},\ms{M})^\triangle\rightarrow\ms{M}$ is an isomorphism.
\end{rem*}

Let us use the notation of \ref{subsecmicloc}.
We consider the induced topology from $(T^*\ms{X})'$ on the underlying
set of $\ms{U}$, and denote the topological space by $\ms{U}'$. We
denote by $\epsilon\colon\ms{U}\rightarrow\ms{U}'$ the continuous map
induced by the identity. The topology of $\ms{U}'$ is slightly finer
(cf. \cite[9.1.2/1]{BGR}) than $\mf{T}$. Thus by \cite[9.2.3/1]{BGR},
the sheaf $M^{\triangle}$ extends uniquely to a sheaf on $\ms{U}'$,
denoted by $(M^\triangle)'$. Now, we get the sheaf
$\epsilon^{-1}((M^\triangle)')$. We also denote this sheaf on the
topological space $\ms{U}$ by $M^\triangle$, and from now on,
$M^\triangle$ indicates the associated sheaf on $\ms{U}$ unless
otherwise stated.

\subsection{}
Let $\ms{X}$ be an affine smooth formal scheme over $R$ possessing
a system of local coordinates. Take a homogeneous element $\Theta$ in
$\Gamma(T^*\ms{X},\mc{O}_{T^*\ms{X}})$ such that $\deg(\Theta)>0$. Let
$\ms{U}:=D(\Theta)$. In the rest of this section, we use the notation of
\ref{notationfixring} freely.

Recall that we have the canonical injection
\begin{equation*}
 \rho_{m,m'}\colon\widehat{E}^{(m,m'+1)}_{\mb{Q}}\rightarrow
  \widehat{E}^{(m,m')}_{\mb{Q}}.
\end{equation*}
We define ${E}^{[m,m']}:=\rho_{m,m'}^{-1}(\widehat{E}^{(m,m')})$, and
equip it with {\em non-exhaustive} filtration by order. Since
$\widehat{E}^{(m,m')}\otimes\mb{Q}=\widehat{E}^{(m,m')}_\mb{Q}$, we
get $E^{[m,m']}\otimes\mb{Q}\cong \widehat{E}^{(m,m'+1)}_\mb{Q}$.
See \ref{defofFSalg} for an account why we need to introduce
this ring.

\begin{lem*}
 \label{mainlemnoefl}
 There exists a subring $\Epr\subset E^{[m,m']}$ such that the following
 holds. We equip $\Epr$ with the induced non-exhaustive filtration
 from $E^{[m,m']}$.
 \begin{enumerate}
  \item The ring $\Epr$ contains $\widehat{E}^{(m,m'+1)}$, and the
	inclusion $\Epr\subset E^{[m,m']}$ is a $p$-isogeny.

  \item The ring $\mr{gr}(\Epr)$ is finitely generated over
	$\mr{gr}(E^{(m,m'+1)})$.

  \item The ring $\Epr$ and the Rees ring $(\Epr_0)_{\bullet}$ of
	$\Epr_0$ are two-sided noetherian.
 \end{enumerate}
\end{lem*}
\begin{proof}
 In this proof, we denote $\widehat{E}^{(m,m'+1)}$ by $E$ for
 simplicity and consider the non-exhaustive filtration by order. By
 Lemma \ref{inversecontE}, we have $(\Thetatil{(m',m'+1)})^{-1}\in
 \widehat{E}^{(m,m')}$ and $(\Thetatil{(m',m'+1)})^{-1}\in
 \widehat{E}^{(m,m'+1)}_{\mb{Q}}$. This shows that
 $(\Thetatil{(m',m'+1)})^{-1}\in E^{[m,m']}$. We put $\Epr$ to be the
 subring of $E^{[m,m']}$ generated by $E$ and
 $\theta:=(\Thetatil{(m',m'+1)})^{-1}$.

 Let us show that the inclusion $\Epr\hookrightarrow E^{[m,m']}$ is a
 $p$-isogeny. Since $E/E_0\rightarrow\widehat{E}^{(m,m')}/E^{(m,m')}_0$
 is a $p$-isogeny by Lemma \ref{explicitdescription}, it suffices to
 check that $\Epr_0\hookrightarrow E_0^{[m,m']}$ is a
 $p$-isogeny. Let $a:=a_{\mr{ord}(\theta)}$ of Lemma
 \ref{normcalc} (i). Let $\partial^{\angles{m'}{\underline{k}}}
 (\Thetatil{(m',m'+1)})^{-i}$ for
 $\underline{k}\geq\underline{0}$ and $i\geq0$ be an operator in
 $E^{[m,m']}$ whose order is less than or equal to $0$. Then there
 exists an integer $j>0$ such that the order of
 $\partial^{\angles{m'}{\underline{k}}}
 (\Thetatil{(m',m'+1)})
 ^{-i+j}$ is strictly greater than $\mr{ord}(\theta)$ and less than or
 equal to $0$. By the choice of $a$, the operator
 $p^a\cdot\partial^{\angles{m'}{\underline{k}}}
 (\Thetatil{(m',m'+1)})^{-i+j}$
 is in $E$, and thus
 \begin{equation*}
  p^a\cdot\partial^{\angles{m'}{\underline{k}}}
   (\Thetatil{(m',m'+1)})^{-i}\,\in E\cdot\theta^j\subset\Epr.
 \end{equation*}
 Take any $P$ in $E_0^{[m,m']}$. There exists an integer $b$ such that
 $p^b\cdot P\in\widehat{E}^{(m,m'+1)}$. Take a left presentation
 (\ref{sum}) of level $m'+1$ such that $p^b\cdot
 b_{\underline{k},i}\in\Gamma(\ms{X},\mc{O}_{\ms{X}})$ for any
 $\underline{k}$ and $i$. For an integer $M'$ and $?\in\{\leq,>\}$, we
 put
 \begin{align*}
  P_{?M'}&:=\sum_{N? M'}\sum_{|\underline{k}|-inp^{m'+1}=N}
  b_{\underline{k},i}\,\underline{\partial}^{\angles{m'+1}
  {\underline{k}}}\,(\Thetatil{(m'+1)})^{-i}.
 \end{align*}
 Since $(\Thetatil{(m'+1)})^{-1}\cdot\Thetatil{(m',m'+1)}\in pE$, the
 operator $(P_{\leq M}\cdot(\Thetatil{(m',m'+1)})^b)$ where
 $M=b\cdot\mr{ord}(\theta)\,(<0)$ is contained in $E$. Thus,
 \begin{equation*}
  p^a\cdot P=p^a\cdot P_{>M}+ p^a\cdot P_{\leq M}\,\in
   \Epr+E\cdot\theta^b\subset\Epr,
 \end{equation*}
 which implies that $p^a\cdot E^{[m,m']}_0\subset\Epr_0$, and the claim
 follows.

 Let us check condition 2 and show that this $\Epr$ is left
 noetherian. We can show $\Epr$ to be right noetherian similarly.
 We define a filtration $G_i$ for $i\geq 0$ on $\Epr$ in the following
 way: we put $G_0(\Epr):=E$. For $i>0$, we inductively define
 $G_{i+1}(\Epr):=E+G_i(\Epr)\cdot\theta$. Let $P\in E_l$ for some
 integer $l$. Since
 $p^n\theta=u\cdot(\Thetatil{(m'+1)})^{-1}$ by
 Lemma \ref{moreexplicitdesc} where $u\in\mb{Z}^*$ and $n$ denotes the
 order of $\Theta$, we can write
 $(p^n\theta)P=P(p^n\theta)+\sum_{k>1}P_k(p^n\theta)^k$ with $P_k\in
 E_{(k-1)np^{m'+1}+l-1}$, thus
 \begin{equation}
  \label{relationT}
   \theta\cdot P\in P\cdot\theta+E_{l-1}\cdot\theta.
 \end{equation} 
 This implies that condition 2 holds. Moreover, the filtration $G$ is
 compatible with ring structure and  exhaustive.
 Let us show that $\mr{gr}^G(\Epr)$ is noetherian.
 Once this is shown, since $G$ is positive, $\Epr$ is noetherian as
 well.

 We put $\overline{E}:=E/p^nE$.
 Let $A:=E\oplus\bigoplus_{k>0}\overline{E}\cdot T^k$ be a graded ring,
 whose graduation is defined by the degree of the indeterminate $T$, and
 the multiplication is defined by
 \begin{equation*}
  T\cdot P=(\theta\,P\,\theta^{-1})\cdot T\quad
   \in (P+E_{l-1})\cdot T
 \end{equation*}
 for $P\in E_l$ where the membership relation holds by
 (\ref{relationT}). It is straightforward to check that this gives us a
 ring structure.
 We denote by $A_i$ the homogeneous part of degree $i$.
 Since $p^n\theta\in E$, there exist the surjection
 $\overline{E}\twoheadrightarrow\mr{gr}^G_i(\Epr)$
 for $i\geq 1$ sending
 $1$ to $\theta^i$, which defines the surjection of rings
 \begin{equation*}
  A\twoheadrightarrow\mr{gr}^G(\Epr).
 \end{equation*}
 It suffices to show that $A$ is noetherian.
 For $Q\in \overline{E}$, we denote by $\sigma(Q)$ the principal symbol
 in $\mr{gr}(\overline{E})$ where the filtration is taken with respect
 to the filtration by order, and for $Q'\in E$, we denote by
 $\sigma(Q')$ the principal symbol of the image of $Q'$ in
 $\overline{E}$.
 Let us define a ``symbol map'' $\Sigma\colon
 A\rightarrow\mr{gr}(\overline{E})$. Let $P\in A$. Then we may write in
 a unique way $P=\sum P_i$ where $P_i\in A_i$. Let $k$ be the
 largest integer such that $P_k\neq 0$.
 We define $\Sigma(P)$ to be $\sigma(P_k)\in\mr{gr}(\overline{E})$. Let
 $I$ be a left ideal of $A$. We define
 \begin{equation*}
  S:=\bigl\{\Sigma(P)\mid P\in I\bigr\}
   \subset\mr{gr}(\overline{E}).
 \end{equation*}
 Since, for $P\in E_l$, we have $T\cdot P\in P\cdot
 T+\overline{E}_{l-1}\cdot T$ in $\mr{gr}^G_1(A)$, the set $S$
 is closed under multiplication by homogeneous elements of
 $\mr{gr}(\overline{E})$. Moreover, $S$ is also closed under addition of
 two homogeneous elements with the same degree.
 Let $I_S$ in $\mr{gr}(\overline{E})$ be the
 ideal generated by $S$. By the above properties,
 we get that
 $S=I_S\cap\bigcup_{i\in\mb{Z}}\mr{gr}_i(\overline{E})$. Since
 $\mr{gr}(\overline{E})$ is noetherian, we can take homogeneous
 generators $P'_1,\dots,P'_k\in S$ of $I_S$. There
 exist elements $P_1,\dots,P_k$ of $A$ such that
 $\Sigma(P_i)=P'_i$ and the degrees of $P_i$ are the same $d>0$ for all
 $i$. For any $P\in I$, the
 completeness of $\overline{E}$ with respect to the filtration
 by order implies that there exists $R_i\in A$ such that $P-\sum R_iP_i$
 is degree less than $d$, and thus contained in
 $\bigoplus_{i<d}A_i$. Since $\bigoplus_{i<d}A_i$ is finite over the
 noetherian ring $E$, there exists $Q_1,\dots,Q_{k'}$
 generating $I\cap\bigoplus_{i<d}A_i$ over $E$. By the
 construction, $P_1\dots,P_k,Q_1,\dots,Q_{k'}$ generate $I$, and in
 particular, $I$ is finitely generated.
 \bigskip


 It remains to show that $(\Epr_0)_{\bullet}$ is noetherian. Although
 the proof is slightly more complicated, the idea is essentially the
 same.
 We define a filtration $F$ on $\bigoplus_{j\leq 0}\Epr\nu^j$: $F_0$ is
 equal to $\bigoplus_{j\leq0}E\nu^j$, and we
 inductively define $F_{i+1}:=F_0+F_i\cdot\theta$, namely
 $F_i=\bigoplus_{j\leq0}G_i(\Epr)\nu^j$.
 We define the induced filtration on $(\Epr_0)_{\bullet}$ from
 $\bigoplus_{j\leq0}\Epr\nu^j$ also denoted by $F$.
 It suffices to show that $\mr{gr}^F((\Epr_0)_{\bullet})$ is
 noetherian. Let $(\mr{gr}^G_i(\Epr))_j$ denotes the image of
 $G_i(\Epr)\cap\Epr_j$ in $\mr{gr}^G_i(\Epr)$, and we put
 $N:=np^{m'+1}=-\mr{ord}(\theta)>0$. Then, for $i>0$, we get a
 surjection
 \begin{equation*}
  \overline{E}_{Ni}\oplus\overline{E}_{Ni-1}\cdot\nu^{-1}\oplus
   \dots\twoheadrightarrow\mr{gr}^F_i((\Epr_0)_{\bullet})\cong
   \bigoplus_{j\leq0}(\mr{gr}^G_i(\Epr))_j\cdot\nu^{j},
 \end{equation*}
 sending $P\cdot\nu^j$ with $P\in\overline{E}_{Ni+j}$ to
 $(P\cdot\theta^i)\cdot\nu^j$. We define a ring {\em graded both by the
 degree of $T$ and $\nu$} by
 \begin{equation*}
  A':=\underbrace{(E_0\oplus E_{-1}\cdot\nu^{-1}\oplus\dots)}_{=:A'_0}
   \oplus\bigoplus_{i>0}\underbrace{
   \Bigl(\overline{E}_{Ni}\oplus\overline{E}_{Ni-1}
   \cdot\nu^{-1}\oplus\dots\Bigr)T^i}_{=:A'_i}
 \end{equation*}
 and define the ring structure in the same way as before using
 (\ref{relationT}).
 If we simply say degree, it means the degree of $T$. We denote by
 $A'_i$ the part of degree$i$. Since there exists a surjection
 $A'\rightarrow\mr{gr}^F((\Epr_0)_{\bullet})$, it suffices to show that
 $A'$ is noetherian. Let $\mr{gr}_{[j]}(\overline{E}):=\bigoplus_{i\leq
 j}\mr{gr}_i(\overline{E})$. We define a {\it double graded commutative}
 ring
 \begin{equation*}
  B:=\bigoplus_{a,b\geq 0}\mr{gr}_{[Na-b]}(\overline{E})\mu^a\nu^{-b},
 \end{equation*}
 whose ring structure is defined by the canonical homomorphism
 \begin{equation*}
  \mr{gr}_{[Na-b]}(\overline{E})\times\mr{gr}_{[Na'-b']}(\overline{E})
   \rightarrow\mr{gr}_{[N(a+a')-(b+b')]}(\overline{E}).
 \end{equation*}
 We claim that this ring is noetherian. For this, it suffices to show
 that $B$ is finitely generated over $\mr{gr}_0(\overline{E})$. We know
 that $\bigoplus_{i\geq0}\mr{gr}_i(\overline{E})$ and
 $\bigoplus_{i\leq0}\mr{gr}_i(\overline{E})$ are finitely generated over
 $\mr{gr}_0(\overline{E})$. Then the following claim leads us to the
 desired conclusion.

 \begin{cl}
  Let $C=\bigoplus_{i\in\mb{Z}}C_i$ be a graded commutative ring. Assume
  that $C$ is noetherian, and that $C_{\leq0}:=\bigoplus_{i\leq0}C_i$
  and $C_{\geq0}:=\bigoplus_{i\geq0}C_i$ are finitely generated over
  $C_0$.
  Let $C_{[j]}:=\bigoplus_{j\geq i}C_i$. Then for any positive integer
  $N$, the ring $D_N:=\bigoplus_{j\geq0}C_{[Nj]}\nu^{Nj}$, where $\nu$
  is an indeterminate, is also finitely generated over $C_0$. Moreover,
  the ring $\bigoplus_{k,j\geq0}C_{[Nj-k]}\nu^j\mu^k$ is finitely
  generated over $C_0$ where $\nu$ and $\mu$ are indeterminates.
 \end{cl}
 \begin{proof}
  \renewcommand{\qedsymbol}{$\square$}
  It suffices to check the $N=1$ case. Indeed, $D_1$ can be seen as a
  $D_N$-algebra, and $D_1$ is integral over $D_n$. Thus if $D_1$ is
  finitely generated over $C_0$, $D_N$ is also finitely generated over
  $C_0$ by \cite[7.8]{AtMac}.

  It suffices to show that $\bigoplus_{j\geq0}\Bigl(\bigoplus_{0\geq
  i}C_i\Bigr)\nu^j\subset D_1$ and
  $\bigoplus_{j\geq0}\Bigl(\bigoplus_{j\geq i\geq0}C_i\Bigr)\nu^j\subset
  D_1$ are finitely generated over $C_0$. Since  the former one is
  isomorphic to $C_{\leq0}[\nu]$, it is finitely generated.
  Let $\{x_i\}_{i\in I}$ be a finite set of generators of
  $C_{\geq0}\cong\bigoplus_{i\geq0}C_i\nu^i\subset D_1$ over $C_0$. Then
  the latter one is generated by $\{x_i\}_{i\in I}$ and $\nu$, and the
  claim follows.
 \end{proof}

 For $P\in A'$, we can write $P=\sum_iP_i$ with $P_i\in A'_i$ in a
 unique way. Let $s$ be the maximal integer such that $P_s\neq 0$. We
 denote $P_s$ in $A'_s$ by $\tau(P)$.
 Let $\tau(P)=\sum_{0\leq i\leq K}P_i\nu^{-i}$ with $P_K\neq0$.
 We define $\Sigma'(P)\in
 B$ to be $\sigma(P_K)\mu^{s}\nu^{-K}$ where $\sigma$ denotes the
 principal symbol with respect to the filtration by order of
 $\overline{E}$. Let $I$ be an ideal of $A'$, and we put
 $S':=\{\Sigma'(P)\mid P\in I\}\subset B$. This set is closed under
 addition of two elements with the same degree, and multiplication by
 homogeneous element. Take a finite set $\{Q_i\}$ in $I$ such that
 $\bigl\{\Sigma'(Q_i)\bigr\}$ is a set of generators of the ideal
 $BS'\subset B$. It is straightforward to check that the set $\{Q_i\}$
 generates $I$.
\end{proof}

\begin{prop}
 \label{FSflatbrac}
 (i) The ring
 $\widehat{E}_{\mb{Q}}^{[m,m']}:=\widehat{E}^{[m,m']}\otimes\mb{Q}$ is
 noetherian where $\widehat{\cdot}$ indicates the $\pi$-adic completion.

 (ii) The canonical homomorphism $\alpha_{m,m'}\colon
 \widehat{E}^{(m,m'+1)}_\mb{Q}\rightarrow\widehat{E}^{[m,m']}_\mb{Q}$
 is flat.

 (iii) The canonical homomorphism $\beta_{m,m'}\colon
 \widehat{E}^{[m,m']}_\mb{Q}\rightarrow\widehat{E}^{(m,m')}_\mb{Q}$ is
 flat.
\end{prop}
\begin{proof}
 We use the notation of Lemma \ref{mainlemnoefl}.
 Since $\Epr$ is noetherian, the canonical homomorphism
 $\Epr\rightarrow\widehat{\Epr}$ is flat and
 $\widehat{\Epr}$ is noetherian by \cite[3.2.3]{Ber1}. Since $\Epr$ is
 $p$-isogeneous to ${E}^{[m,m']}$, they are also $p$-isogeneous even
 after taking $\pi$-adic completion by Lemma \ref{complofisog}. Thus we
 get (i). Since
 $E^{[m,m']}\otimes\mb{Q}\cong\widehat{E}^{(m,m'+1)}_\mb{Q}$, the
 flatness of $\alpha_{m,m'}$ follows, which is (ii).

 Let us prove (iii). We put $\Epr_{\mr{fin}}:=\bigcup_n\Epr_n$. By
 condition 2 of Lemma \ref{mainlemnoefl} and Lemma
 \ref{intersectioncoherent},
 $\bigoplus_{i\geq0}\mr{gr}_i(\Epr_{\mr{fin}})$ is noetherian.
 Since $\Epr_0$ is a noetherian filtered ring with respect to the
 filtration by order by condition 3, $\Epr_{\mr{fin}}$ is also a
 noetherian filtered ring by Lemma \ref{lemmonnoethfilt}. Let
 $\Epr_{\mr{fin}}'$ be the completion of $\Epr_{\mr{fin}}$ with respect
 to the filtration. Then the canonical homomorphism
 $\Epr_{\mr{fin}}\rightarrow \Epr_{\mr{fin}}'$ is flat and
 $\Epr_{\mr{fin}}'$ is noetherian (cf.\ \ref{noethfiltdef}). Thus,
 by taking the $\pi$-adic completion, the canonical homomorphism
 $\widehat{\Epr_{\mr{fin}}}\rightarrow\widehat{\Epr_{\mr{fin}}'}$ is
 flat by \cite[3.2.3 (vii)]{Ber1} where $^\wedge$ denote the $\pi$-adic
 completion. It suffices to show that
 \begin{equation}
  \label{twoisomtoshow}
  \widehat{\Epr_{\mr{fin}}}\otimes{\mb{Q}}\cong
  \widehat{E}^{[m,m']}_{\mb{Q}},\qquad
 \widehat{\Epr_{\mr{fin}}'}\otimes{\mb{Q}}\cong
 \widehat{E}_{\mb{Q}}^{(m,m')}.
 \end{equation}
 Let $E^{[m,m']}_{\mr{fin}}:=\bigcup_nE_n^{[m,m']}$. Since
 $\Epr_{\mr{fin}}\subset E^{[m,m']}_{\mr{fin}}$ is a $p$-isogeny and the
 $\pi$-adic completion of $E^{[m,m']}_{\mr{fin}}$ is
 $\widehat{E}^{[m,m']}_{\mb{Q}}$, we get the first isomorphism. Let us
 show the second one. Note that the completion of
 $E_{\mr{fin}}^{[m,m']}$ with respect to the filtration by order is
 $E^{(m,m')}$. There exists an integer $n$ such that
 $p^nE^{[m,m']}_{\mr{fin}}\subset\Epr_{\mr{fin}}\subset
 E^{[m,m']}_{\mr{fin}}$. Since these inclusions are strict
 homomorphisms, the inclusions are preserved even
 after taking the completion with respect to the filtration by order,
 and we get $p^nE^{(m,m')}\subset\Epr_{\mr{fin}}'\subset E^{(m,m')}$. In
 particular, the inclusion $\Epr_{\mr{fin}}'\subset E^{(m,m')}$ is a
 $p$-isogeny, which implies the second isomorphism of
 (\ref{twoisomtoshow}), and the proposition follows.
\end{proof}

\begin{lem}
\label{calofthered}
 Let $m'>m$. Put $F^{(m')}:=\iota(\widehat{E}^{(m,m')}_{\mb{Q}})
 \cap\widehat{E}^{(m)}$ where
 $\iota\colon\widehat{E}^{(m,m')}_{\mb{Q}}\rightarrow
 \widehat{E}^{(m)}_{\mb{Q}}$ is the canonical injection. Then, for any
 $j\geq0$, $F^{(m')}_{X_j}:=F^{(m')}\otimes R_j$
 does not depend on $m'$.
\end{lem}
\begin{proof}
 By definition, the canonical homomorphism
 $F^{(m')}_{X_j}\hookrightarrow E^{(m)}_{X_j}:=
 \Gamma(\ms{U},\Emod{m}{X_j})$ is injective.
 There exists $m''\geq m'$ such that $\Thetatil{(m,m'')}$ is in the
 center of $D^{(m)}_{X_j}$.
 Let $LD^{(m)}_{X_j}$ be the subring of $E^{(m)}_{X_j}$ generated by
 $D^{(m)}_{X_j}$ and $(\Thetatil{(m,m'')})^{-1}$, which does not depend
 on the choice of $m''$. (In fact,
 $LD^{(m)}_{X_j}:=\Gamma(\ms{U},\ms{L}\Dmod{m}{X_j})$ using the notation
 of Remark \ref{remBer}.) It suffices to show that
 ${LD}^{(m)}_{X_j}=F^{(m')}_{X_j}$ in $E^{(m)}_{X_j}$.
 Since $(\Thetatil{(m,m'')})^{-1}\in
 F^{(m')}$, we have ${LD}^{(m)}_{X_j}\subset F^{(m')}_{X_j}$.
 Let us show the opposite inclusion.
 By Remark \ref{concrdescrem}, any element of $F^{(m')}$ can be written
 as
 \begin{equation*}
  \underbrace{\sum_{k<0}P_{k,i}\cdot(\Thetatil{(m',m'')})
   ^{-i}}_{\ccirc{1}}+
   \underbrace{\sum_{k\geq0}P_{k,i}\cdot(\Thetatil{(m,m'')})
   ^{-i}}_{\ccirc{2}}
 \end{equation*}
 where $P_{k,i}\in(D^{(m)}_{\mb{Q}})_{k+np^{m''}i}$
 ($n:=\mr{ord}(\Theta)$) with some convergence conditions.
 The sum \ccirc{2} becomes finite in $E^{(m)}_{X_j}$ since
 $\lim_{k\rightarrow\infty}P_{k,i}=0$. Since
 $(\Thetatil{(m',m'')})^{-1}\equiv0\bmod p\widehat{E}^{(m)}$, \ccirc{1}
 is also a finite sum in $E^{(m)}_{X_j}$, and we have
 $F^{(m')}_{X_j}\subset{LD}^{(m)}_{X_j}$.
\end{proof}

\begin{cor*}
 \label{denseimage}
 The image of the homomorphism $\widehat{E}^{(m,m'+2)}_\mb{Q}
 \rightarrow\widehat{E}^{[m,m']}_\mb{Q}$
 is dense with respect to the $\pi$-adic topology on
 $\widehat{E}^{[m,m']}_\mb{Q}$ for any $m'\geq m$.
\end{cor*}
\begin{proof}
 Set $(m,m')$ of Lemma \ref{calofthered} to be $(m',m'+1)$. Then
 $F^{(m'+1)}=E^{[m',m']}$, and the lemma implies that the
 image of the homomorphism $\widehat{E}^{(m',m'+2)}_\mb{Q}
 \rightarrow\widehat{E}_\mb{Q}^{[m',m']}$ is dense. This shows that the
 image of the homomorphism
 $(E^{(m',m'+2)}_\mb{Q})_0\rightarrow(E_\mb{Q}^{[m',m']})_0$
 is also dense, and so is the image of the composition
 \begin{equation*}
  (E^{(m,m'+2)}_\mb{Q})_0\xrightarrow{\sim}
   (E^{(m',m'+2)}_\mb{Q})_0\rightarrow
   (\widehat{E}_\mb{Q}^{[m',m']})_0\xleftarrow{\sim}
   (\widehat{E}_\mb{Q}^{[m,m']})_0,
 \end{equation*}
 which is nothing but the canonical homomorphism
 $(E^{(m,m'+2)}_\mb{Q})_0\rightarrow(\widehat{E}_\mb{Q}^{[m,m']})_0$. Since
 \begin{equation*}
  \widehat{E}^{(m,m'+2)}_\mb{Q}/(E^{(m,m'+2)}_\mb{Q})_0
   \xrightarrow{\sim}
   \widehat{E}_\mb{Q}^{[m,m']}/(\widehat{E}_\mb{Q}^{[m,m']})_0,
 \end{equation*}
 we conclude the proof.
\end{proof}

\subsection{}
\label{defofFSalg}
We recall the definition of Fr\'{e}chet-Stein algebra. For more
details, we refer to \cite{ST}. A $K$-algebra $A$ together with
a projective system of $K$-Banach algebras $\{A_i\}_{i\geq 0}$ and a
homomorphism of projective systems $A\rightarrow\{A_i\}$ where $A$
denotes the constant projective system is called a {\em
Fr\'{e}chet-Stein algebra} (cf.\ \cite[\S3]{ST}) if the following hold.
\begin{enumerate}
 \item For any $i\geq 0$, the ring $A_i$ is noetherian.
 \item The transition homomorphism $A_{i+1}\rightarrow A_i$ is flat and
       the image is dense in $A_i$.
 \item The given homomorphism of projective systems induces an
       isomorphism of $K$-algebras $A\rightarrow\invlim_iA_i$.
\end{enumerate}
In general, the image of the homomorphism
$\widehat{E}^{(m,m'+1)}_{\mb{Q}}\rightarrow
\widehat{E}^{(m,m')}_{\mb{Q}}$ is not dense, and the projective system
$\bigl\{\widehat{E}^{(m,m')}_{\mb{Q}}\bigr\}_{m'\geq m}$ does not give a
Fr\'{e}chet-Stein structure on $E^{(m,\dag)}_{\mb{Q}}$. We need to
replace $\widehat{E}^{(m,m')}_{\mb{Q}}$ by
$\widehat{E}^{[m,m']}_{\mb{Q}}$ to get such a structure as the following
theorem shows.

\begin{thm}
 \label{FSprop}
 (i) The ring $E^{(m,\dag)}_{\mb{Q}}$ is a Fr\'{e}chet-Stein algebra
 with respect to the projective system
 $\bigl\{\widehat{E}^{[m,m']}_{\mb{Q}}\bigr\}_{m'\geq m}$.

 (ii) For a finitely presented $E^{(m,\dag)}_{\mb{Q}}$-module $M$, we
 have
 \begin{equation*}
   R^i\invlim_{m'}(\widehat{E}^{(m,m')}_\mb{Q}
   \otimes_{E^{(m,\dag)}_{\mb{Q}}}M)\xleftarrow{\sim}
   \begin{cases}
    M&\mbox{if $i=0$}\\
    0&\mbox{if $i\neq0$.}
   \end{cases}
 \end{equation*}

 (iii) Let $\ms{V}\subset\ms{U}$ be strictly affine open subschemes of
 $\mathring{T}^*\ms{X}$. Then the homomorphism
 $\Gamma(\ms{U},\Emod{m,\dag}{\ms{X},\mb{Q}})\rightarrow
 \Gamma(\ms{V},\Emod{m,\dag}{\ms{X},\mb{Q}})$ is flat.
\end{thm}
\begin{proof}
 For (i), combine Proposition \ref{FSflatbrac} and Corollary
 \ref{denseimage}.
 To check (ii), the projective systems
 $\bigl\{\widehat{E}_\mb{Q}^{[m,m']}\otimes M\bigr\}_{m'\geq m}$ and
 $\bigl\{\widehat{E}_\mb{Q}^{(m,m')}\otimes M\bigr\}_{m'\geq m}$ are
 cofinal in the projective system
 \begin{equation*}
  \dots\rightarrow \widehat{E}_\mb{Q}^{[m,m'+1]}\otimes M\rightarrow
   \widehat{E}_\mb{Q}^{(m,m'+1)}\otimes
   M\rightarrow\widehat{E}_\mb{Q}^{[m,m']}\otimes
   M\rightarrow\widehat{E}
   _\mb{Q}^{(m,m')}\otimes M\rightarrow\dots,
 \end{equation*}
 and these three projective systems have the same $R^i\invlim_{m'}$.
 Thus, \cite[\S3 Theorem]{ST} leads us to (ii). Let us prove (iii).
 Put $E_*:=\Gamma(*,\Emod{m,\dag}{\ms{X},\mb{Q}})$ for
 $*\in\{\ms{U},\ms{V}\}$.
 Let $0\rightarrow I\rightarrow E_{\ms{U}}\rightarrow M\rightarrow0$ be an
 exact sequence of $E_{\ms{U}}$-modules such that $I$ is a finitely
 generated ideal.
 By \cite[I, \S4 Proposition 1]{Bour}, it suffices to show that
 $\mr{Tor}_1^{E_{\ms{U}}}(E_{\ms{V}},M)=0$. Using \cite[Corollary
 3.4-i]{ST}, since $M$ is finitely presented, $M$ is coadmissible, and
 thus $I$ is coadmissible as well. By \cite[Remark 3.2, \S3
 Theorem]{ST} and Proposition \ref{statementcomplete}, the sequence
 \begin{equation*}
  0\rightarrow E_{\ms{V}}\otimes_{E_{\ms{U}}}I\rightarrow
   E_{\ms{V}}\rightarrow E_{\ms{V}}\otimes_{E_{\ms{U}}}M\rightarrow0
 \end{equation*}
 is exact, and we get the vanishing of $\mr{Tor}$.
\end{proof}

\begin{rem*}
 In (ii) of the theorem, we can more generally take $M$ to be a
 coadmissible $E^{(m,\dag)}_{\mb{Q}}$-module (cf.\ \cite[\S3]{ST}).
\end{rem*}

\begin{cor}
 \label{asssheafdag}
 Let $\ms{U}$ be a strictly affine open subscheme of $\mathring{T}^*\ms{X}$,
 and $M$ be a {\em finitely presented}
 $\Gamma(\ms{U},\Emod{m,\dag}{\ms{X},\mb{Q}})$-module. We define the
 presheaf $M^\triangle$ in the same way as
 {\normalfont\ref{defassocish}}. Then Lemma
 {\normalfont\ref{calcofcoh}}, Lemma {\normalfont\ref{calcofcohmod}},
 Corollary {\normalfont\ref{exactnessofassociation}} are also valid for
 $\ms{E}=\Emod{m,\dag}{\ms{X},\mb{Q}},\ms{E}^\dag_{\ms{X},\mb{Q}}$, and
 $M$.
\end{cor}
\begin{proof}
 Let us check the claim for $\ms{E}=\Emod{m,\dag}{\ms{X},\mb{Q}}$.
 For any strictly open subscheme $\ms{V}\subset\ms{U}$,
 \begin{equation*}
  R^i\invlim_{m'}\Bigl(\Gamma(\ms{V},\EcompQ{m,m'}{\ms{X}})\otimes
   M\Bigr)=0
 \end{equation*}
 for $i>0$ by Theorem \ref{FSprop} (ii). Let us denote by
 $\EcompQ{m,m'}{\ms{X}}\otimes M$ the coherent
 $\EcompQ{m,m'}{\ms{X}}|_{\ms{U}}$-module associated with $M$. By Lemma
 \ref{calcofcohmod}, this shows that the sequence
 \begin{equation*}
  \dots\rightarrow
   \invlim_{m'}C^q_{\mr{aug}}(\mf{U},\EcompQ{m,m'}{\ms{X}}\otimes M)
   \rightarrow
   \invlim_{m'}C^{q+1}_{\mr{aug}}(\mf{U},\EcompQ{m,m'}{\ms{X}}\otimes M)
   \rightarrow\dots
 \end{equation*}
 is exact. Since
 \begin{equation*}
  \invlim_{m'}C^q_{\mr{aug}}(\mf{U},\EcompQ{m,m'}{\ms{X}}\otimes M)\cong
   C^q_{\mr{aug}}(\mf{U},\Emod{m,\dag}{\ms{X},\mb{Q}}\otimes M)
 \end{equation*}
 by Theorem \ref{FSprop} (ii), Lemma \ref{calcofcoh} and Lemma
 \ref{calcofcohmod} for this $\ms{E}$ follows. The verification of
 Corollary \ref{exactnessofassociation} is similar.
 For the claims on $\EdagQ{\ms{X}}$, we only note that the functor
 $\indlim$ is exact.
\end{proof}

\begin{cor}
 \label{keyisomdag}
 Let $m'>m$ be non-negative integers. Then the canonical injection
 $\Emod{m+1,\dag}{\ms{X},\mb{Q}}\rightarrow\EcompQ{m+1,m'}{\ms{X}}$ induces the
 isomorphism:
 \begin{equation*}
  \Emod{m+1,\dag}{\ms{X},\mb{Q}}/\Emod{m,\dag}{\ms{X},\mb{Q}}\xrightarrow{\sim}
   \EcompQ{m+1,m'}{\ms{X}}/\EcompQ{m,m'}{\ms{X}}.
 \end{equation*}
\end{cor}
\begin{proof}
 It suffices to show that
 $E^{(m+1,\dag)}_{\mb{Q}}/E^{(m,\dag)}_{\mb{Q}}\xrightarrow{\sim}
 \widehat{E}^{(m+1,m')}_{\mb{Q}}/\widehat{E}^{(m,m')}_{\mb{Q}}$ for any
 strictly affine open subscheme $\ms{U}$.
 This follows from Lemma \ref{keyisom} and the fact that
 $R^1\invlim_{m'}\widehat{E}^{(m,m')}_{\mb{Q}}=0$ by Theorem
 \ref{FSprop}.
\end{proof}

\subsection{}
Now, we argue the flatness of
$\EcompQ{m,m'}{\ms{X}}\rightarrow\EcompQ{m+1,m'}{\ms{X}}$.

\begin{lem*}
 The canonical homomorphism
 $\EcompQ{m,m'}{\ms{X}}\rightarrow\EcompQ{m+1,m'}{\ms{X}}$
 is flat for non-negative integers $m'>m$.
\end{lem*}
\begin{proof}
 Since the verification is local, we may assume that we are in the
 situation of \ref{defofsituation}. It suffices to check that
 $\widehat{E}^{(m,m')}_{\mb{Q}}\rightarrow
 \widehat{E}^{(m+1,m')}_{\mb{Q}}$ is flat. The proof being similar to
 \cite[3.5.3]{Ber1}, we only sketch. Let $F$
 be the subring of $\widehat{E}^{(m,m')}_{\mb{Q}}$ generated over
 $\widehat{E}^{(m,m')}$ by
 $\bigl\{\partial_i^{\angles{m+1}{p^{m+1}}}\bigr\}_{1\leq i\leq
 d}$. Since $[P,\partial_i^{\angles{m+1}{p^{m+1}}}]\in D^{(m)}$ for
 $P\in D^{(m)}$,
 \begin{equation*}
  [(\Thetatil{(m)})^{-1},\partial_i^{\angles{m+1}{p^{m+1}}}]=
   (\Thetatil{(m)})^{-1}
   \cdot[\partial_i^{\angles{m+1}{p^{m+1}}},
   \Thetatil{(m)}]\cdot(\Thetatil{(m)})^{-1}
   \in\widehat{E}^{(m)}.
 \end{equation*}
 Thus, $[Q,\partial_i^{\angles{m+1}{p^{m+1}}}]\in\widehat{E}^{(m,m')}$
 for $Q\in\widehat{E}^{(m,m')}$. This shows that
 \begin{equation*}
  F=\sum_{\underline{k}}\widehat{E}^{(m,m')}\cdot
   \bigl(\underline{\partial}^{\angles{m+1}{p^{m+1}}}\bigr)
   ^{\underline{k}}.
 \end{equation*}
 Now, define the filtration on $F$ by the order of
 $\underline{\partial}^{\angles{m+1}{p^{m+1}}}$. Then we have a
 surjection $\widehat{E}^{(m,m')}[T_1,\dots,T_d]\rightarrow\mr{gr}(F)$
 sending $T_i$ to $\sigma(\partial_i^{\angles{m+1}{p^{m+1}}})$. Thus,
 $F$ is noetherian since $\widehat{E}^{(m,m')}$ is. This implies that
 the homomorphism $F\rightarrow F^\wedge$ is flat. Using Lemma
 \ref{normcalc} (i), we can check that $F^\wedge$ is $p$-isogeneous to
 $\widehat{E}^{(m+1,m')}_{\mb{Q}}$, and the lemma follows.
\end{proof}

\subsection{}
\label{mainresofsecflat}
We sum up the results we got in this section as the following theorem.

\begin{thm*}
 Let $m'\geq m$ be non-negative integers.
 \begin{enumerate}
 \item\label{1}
      The canonical injective homomorphism
      $\Emod{m,\dag}{\ms{X},\mb{Q}}\rightarrow\EcompQ{m,m'}{\ms{X}}$ is
      flat.
 \item\label{2}
      The canonical injective homomorphism
      $\EcompQ{m,m'+1}{\ms{X}}\rightarrow\EcompQ{m,m'}{\ms{X}}$ is
      flat.
 \item\label{3}
      Let $\ms{M}$ be a finitely presented
      $\Emod{m,\dag}{\ms{X},\mb{Q}}$-module. Then we get
      \begin{equation*}
       \ms{M}\xrightarrow{\sim}\invlim_{m'}\EcompQ{m,m'}{\ms{X}}
	\otimes_{\Emod{m,\dag}{\ms{X},\mb{Q}}}\ms{M}.
      \end{equation*}
 \item\label{4}
      The canonical injective homomorphism
      $\EcompQ{m,m'}{\ms{X}}\rightarrow\EcompQ{m+1,m'}{\ms{X}}$ is
      flat.
 \end{enumerate}
\end{thm*}
\begin{proof}
 We restate what we have proven for \ref{2} and \ref{4}. To
 check \ref{1}, it suffices to apply \cite[Remark 3.2]{ST}.
 Let us prove \ref{3}. Since $\ms{M}$ is finitely
 presented, there exists a strictly affine open subscheme $\ms{U}$
 such that there exists a presentation on which $\ms{M}$ possesses a
 finite presentation.
 Then we apply Corollary \ref{asssheafdag}.
\end{proof}

\begin{rem*}
 We do not know if $\EcompQ{m,m'}{\ms{X}}$ is flat
 over $\pi^{-1}\DcompQ{m}{\ms{X}}$. However, in the curve case, this is
 flat by \cite[1.3.4]{AM}.
\end{rem*}

\section{On finiteness of sheaves of rings}
\label{snoethcond}
In this section, we introduce a finiteness property for modules on
certain topological spaces, and prove some stationary type theorem. This
finiteness is especially useful when we consider modules on formal
schemes.

\subsection{}
First, let us introduce conditions on topological spaces and on
sheaves.
\begin{quote}
 A ringed space $(X,\mc{O}_X)$ is said to satisfy condition (FT)
 if the following two conditions hold.
 \begin{enumerate}
  \item The topological space $X$ is sober\footnote{In this paper, we do
	not use the uniqueness of generic points, and this assumption is
	a little stronger than what is really needed.}
	({\it i.e.}\ any irreducible closed subset has a unique generic
	point, see \cite[Exp.\ IV, 4.2.1]{SGA}) and noetherian
	(cf. \cite[$0_{\mr{I}}$, \S2.2]{EGA}).
  \item The structure sheaf $\mc{O}_{X}$ is a coherent ring, and
	$\mc{O}_{X,x}$ is noetherian for any $x\in X$.
 \end{enumerate}
\end{quote}
Let $(X,\mc{O}_X)$ be a ringed space satisfying (FT), and let $\mc{M}$
be a coherent $\mc{O}_X$-module. Let $\mf{Z}:=\{Z_i\}_{i\in I}$ be a
{\em finite} family of irreducible closed subsets. The module $\mc{M}$
is said to satisfy condition (SH) with respect to $\mf{Z}$ if the
following holds.
\begin{quote}
 For any section
 $s\in\Gamma(U,\mc{M})$ over any open subset $U$, there
 exists a subset $I'\subset I$ such that
 $\mr{Supp}(s)=\bigcup_{i\in I'}Z_i\cap U$.
\end{quote}
We simply say that $\mc{M}$ satisfies condition (SH) if there exists
a finite family $\mf{Z}$ such that $\mc{M}$ satisfies (SH) with respect
to $\mf{Z}$.

\begin{lem}
 \label{elemlemshcon}
 Let $(X,\mc{O}_X)$ be a ringed space satisfying {\normalfont(FT)}. Let
 $\mc{M}$ be a coherent $\mc{O}_X$-module satisfying
 {\normalfont(SH)} with respect to $\mf{Z}=\{Z_i\}_{i\in I}$. Then for
 any sub-$\mc{O}_X$-module $\mc{K}$ of $\mc{M}$, there exists an open
 subset $Z'_i$ of $Z_i$ for each $i\in I$ such that
 $\mr{Supp}(\mc{K})=\bigcup_{i\in I}Z'_i$.
\end{lem}
\begin{proof}
 Let $U$ be an open subset of $X$, and take $0\neq
 s\in\Gamma(U,\mc{K})$. Let $\varphi_U\colon\Gamma(U,\mc{K})
 \hookrightarrow\Gamma(U,\mc{M})$ be the inclusion.
 Since $\varphi$ is injective, $\mr{Supp}(s)=\mr{Supp}(\varphi_U(s))$.
 There exists a subset $I_s\subset I$ such that
 \begin{equation*}
  \mr{Supp}(s)=\mr{Supp}(\varphi_U(s))=
   \bigcup_{i\in I_s}Z_i\cap U
 \end{equation*}
 by (SH) of $\mc{M}$. Note that this is an open subset of $\bigcup_{i\in
 I_s}Z_i$. Let $\mf{S}:=\bigcup_{U\subset X}\Gamma(U,\mc{K})$ where $U$
 runs over open subsets of $X$, and $\mf{S}_i$ be the subset of $\mf{S}$
 consisting of the elements $s$ such that $i\in I_s$. Now, we get
 \begin{equation*}
  \mr{Supp}(\mc{K})=\bigcup_{s\in\mf{S}}\mr{Supp}(s)=\bigcup_{i\in
   I}\Bigl(\bigcup_{s\in\mf{S}_i}\mr{Supp}(s)\cap Z_i\Bigr).
 \end{equation*}
 Since $\mr{Supp}(s)\cap Z_i$ is open in $Z_i$,
 the set $Z'_i:=\bigcup_{s\in\mf{S}_i}\mr{Supp}(s)\cap Z_i$ is also open in
 $Z_i$.
\end{proof}

\begin{prop}
\label{stationary}
 Let $(X,\mc{O}_X)$ be a ringed space satisfying {\normalfont(FT)}. Let
 $\mc{M}$ be a coherent $\mc{O}_X$-module, and assume that for any open
 subset $U\subset X$, {\normalfont(SH)} holds for any coherent
 subquotient of $\mc{M}|_{U}$. Now, let
 \begin{equation*}
  \mc{K}_1\subset\mc{K}_2\subset\mc{K}_3\subset\dots\subset\mc{M}
 \end{equation*}
 be an ascending chain of sub-$\mc{O}_X$-modules (not necessarily
 coherent) of $\mc{M}$. Then the chain is stationary.
\end{prop}
\begin{proof}
 Let $n\in\mb{N}$ and $Z$ be a closed subset of
 $X$. We say that the chain is {\em stationary for $(n,Z)$} if
 $\mc{K}_n|_{X\setminus Z}=\mc{K}_i|_{X\setminus Z}$
 for any $i\geq n$. We claim that if the chain is stationary for
 $(n,Z)$ with $Z\neq\emptyset$, then there exists an integer $n'$ and
 $Z'\subsetneq Z$ such that the
 chain is stationary for $(n',Z')$. Once this is proven, (i) follows
 since $X$ is a noetherian space.

 Let us show the claim. By Lemma \ref{elemlemshcon}, there exists an
 integer $a$ such that
 $\overline{\mr{Supp}(\mc{K}_i)}=\overline{\mr{Supp}(\mc{K}_a)}$ for any
 $i\geq a$. We may suppose that
 $Z\subset\overline{\mr{Supp}(\mc{K}_a)}$.
 Take a generic point $\eta$ of $Z$. Since $\mc{O}_{X,\eta}$ is
 noetherian, there exists $n'\geq\max\{a,n\}$ such that
 $\mc{K}_{i,\eta}=\mc{K}_{n',\eta}$
 for any $i\geq n'$. Fix a set of
 generators $\{f_1,\dots,f_\alpha\}$ of $\mc{K}_{n',\eta}$. There exists
 an open neighborhood $U$ of $\eta$ such that $\{f_1,\dots,f_{\alpha}\}$
 can be lifted on $U$ and $U\cap Z$ is irreducible. We fix a set of
 liftings $\{\tilde{f}_1,\dots,\tilde{f}_\alpha\}$ in
 $\Gamma(U,\mc{K}_{n'})$.
 Let $\mc{S}$ be the sub-$\mc{O}_X|_U$-module of $\mc{M}|_U$ generated
 by $\{\tilde{f}_1,\dots,\tilde{f}_\alpha\}$, which is coherent since
 $\mc{O}_X$ is. Now, let $\mc{M}'_U:=\mc{M}|_U/\mc{S}$ be a coherent
 $\mc{O}_X|_U$-module, and denote by $\mc{K}'_i$ the image of
 $\mc{K}_i|_U$ in $\mc{M}'_U$. We know that $\mr{Supp}(\mc{K}_i)\cap
 U\supset\mr{Supp}(\mc{K}'_i)$.
 By construction, $\eta\not\in\mr{Supp}(\mc{K}'_i)$ for any $i\geq
 n'$. By assumption, $\mc{M}'_U$ also
 satisfies (SH). Let $\mf{W}:=\{W_j\}_{j\in J}$ be a finite
 family of irreducible closed subset of $U$ such that $\mc{M}'_U$
 satisfies (SH) with respect to $\mf{W}$. Let $J'$ be the subset of $J$
 such that $\eta\not\in W_j$,
 and we put $W':=\bigcup_{j\in J'}W_j$. We let
 \begin{equation*}
  Z':=(Z\cap W')\cup(Z\setminus U).
 \end{equation*}
 Since $\eta\not\in Z'$, we get $Z'\subsetneq Z$. For any $i\geq n'$, by
 Lemma \ref{elemlemshcon}, there exists an open subset $W'_j$ of $W_j$
 for each $j\in J$ such that 
 \begin{equation*}
  \mr{Supp}(\mc{K}'_i)=\bigcup_{j\in J}W'_j.
 \end{equation*}
 We claim that $W'_j\cap
 Z=\emptyset$ for any $j\not\in J'$. Indeed, $j\not\in
 J'$ implies $\eta\in W_j$ and $Z\cap U\subset W_j$. If $W'_j\cap
 Z\neq\emptyset$, we would get $\eta\in W'_j$ since $Z\cap U$ is
 irreducible closed and $W'_j$ is open in $W_j$. This contradicts with
 $\eta\not\in\mr{Supp}(\mc{K}'_i)$. Thus,
 \begin{equation}
  \label{calcofsupport}
  \mr{Supp}(\mc{K}'_i)\cap(Z\setminus W')\cap U=\emptyset.
 \end{equation}
 Now, the chain is stationary for $(n',Z')$: it suffices to check
  $\mc{K}_{i,z}=\mc{K}_{n',z}$ for any $z\in Z\setminus Z'=(Z\setminus
 W')\cap U$. However, we get $\mc{K}'_{i,z}=0$
 for any $i\geq n'$ by (\ref{calcofsupport}). Thus,
 $\mc{K}_{i,z}=\mc{S}_z$ by the definition of $\mc{K}'_i$, which
 concludes the proof.
\end{proof}

\subsection{}
We show that coherent modules over some noetherian rings we
have defined in this paper satisfy (SH). For this, we prepare some
lemmas. In the following, let $(X,\mc{O}_X)$ be a ringed space
satisfying (FT).

\begin{lem*}
\label{extension*}
 Condition {\normalfont(SH)} is closed under extensions: suppose
 there exists an exact sequence of coherent $\mc{O}_X$-modules
 $0\rightarrow\mc{F}'\rightarrow\mc{F}\rightarrow\mc{F}''\rightarrow 0$
 such that $\mc{F}'$ and $\mc{F}''$ satisfy condition
 {\normalfont(SH)}. Then $\mc{F}$ also satisfies condition
 {\normalfont(SH)}.
\end{lem*}
\begin{proof}
 Left to the reader.
\end{proof}

\begin{lem}
 \label{SHilocal}
 Let $\mc{M}$ be a coherent $\mc{O}_X$-module. The module $\mc{M}$
 satisfies condition {\normalfont(SH)} if and only if there exists a
 covering $\{U_i\}_{i\in I}$ of $X$ such that $\mc{M}|_{U_i}$ satisfies
 the condition on $U_i$ for any $i$.
\end{lem}
\begin{proof}
 We only need the proof for the ``if'' part.
 Since $X$ is quasi-compact, we may assume that the covering is
 finite. By assumption, for each $i\in I$, there exists a family
 $\{Z_j\}_{j\in J_i}$ of closed subsets of $U_i$ such that
 $\mc{M}|_{U_i}$ satisfies (SH) with respect to this family. The module
 $\mc{M}$ satisfies (SH) with respect to the family $\bigcup_{i\in
 I}\bigl\{\overline{Z_j}\bigr\}_{j\in J_i}$. Since the verification is
 straightforward, we leave the details to the reader.
\end{proof}

\begin{lem}
\label{filtration}
 Let $(\mc{F},\mc{F}_i)$ be a separated filtered sheaf. Suppose that
 $\mr{gr}(\mc{F})$ satisfies {\normalfont(SH)}. Then
 $\mc{F}$ also satisfies {\normalfont(SH)}.
\end{lem}
\begin{proof}
 Assume that $\mr{gr}(\mc{F})$ satisfies (SH) with respect to
 $\mf{Z}=\{Z_i\}_{i\in I}$. Let $U$ be an open subset of $X$, and take a
 non-zero $s\in\Gamma(U,\mc{F})$. There exists an integer $i_0$ such
 that $s\in\Gamma(U,\mc{F}_{i_0})$ and
 $s\not\in\Gamma(U,\mc{F}_{i_0-1})$ since the filtration is exhaustive
 (cf.\ see Convention) and separated. Let
 $\sigma(s)\in\Gamma(U,\mr{gr}_{i_0}(\mc{F}))\subset
 \Gamma(U,\mr{gr}(\mc{F}))$ be the principal symbol of $s$. Then there
 exists $J_0\subset I$ such that $\mr{Supp}(\sigma(s))=\bigcup_{j\in
 J_0}Z_j\cap U$.
 For $k\geq0$, we inductively define an open subset $U_k$, and a subset
 $J_k$ of $J$ in the following way. We put $U_0:=U$. Now, let
 $U_{k+1}:=U_{k}\setminus\mr{Supp}(\sigma(s|_{U_k}))$. Then there exists
 $J'_{k+1}$ such that
 \begin{equation*}
  \mr{Supp}(\sigma(s|_{U_{k+1}}))=\bigcup_{j\in J'_{k+1}}Z_j\cap U_{k+1}.
 \end{equation*}
 We define $J_{k+1}:=J_k\cup J'_{k+1}$.
 Obviously, $J_0\subset J_1\subset\dots\subset I$. Since $I$ is a
 finite set, this sequence is stationary. Let $J:=\bigcup_{i}J_i\subset
 I$. Then $\mr{Supp}(s)=\bigcup_{j\in J}Z_j\cap U$,
 and $\mc{F}$ satisfies (SH) with respect to $\mf{Z}$ as well.
\end{proof}

\begin{dfn}
 Let $(X,\mc{O}_X)$ be a ringed space satisfying (FT). The ringed space
 or $\mc{O}_X$ is said to be {\em {\normalfont(SH)}-noetherian} if
 condition (SH) is satisfied for any coherent $\mc{O}_X|_U$-modules
 for any open subset $U$ of $X$. If, moreover, $\mc{O}_X$ is noetherian
 with respect to $\mf{B}$, we say that $\mc{O}_X$ is {\em strictly
 noetherian with respect to $\mf{B}$}.
\end{dfn}

\begin{lem}
\label{takinggraded}
 Let $(X,\mc{E})$ be a ringed space satisfying {\normalfont(FT)}, and
 let $(\mc{E},\mc{E}_i)$ be a filtration on $\mc{E}$.
 Suppose that the filtration is pointwise Zariskian (cf.\ Definition
 {\normalfont\ref{defnoethrigsh}}). If $\mr{gr}(\mc{E})$ is
 {\normalfont (SH)}-noetherian, then so is $\mc{E}$.
\end{lem}
\begin{proof}
 Let $\mc{M}$ be a coherent $\mc{E}$-module. Since the verification is
 local by Lemma \ref{SHilocal}, we may suppose that there exists a good
 filtration $(\mc{M},\mc{M}_i)$. By Lemma \ref{Zarisep}, the filtration
 is separated. Now by Lemma \ref{filtration}, the corollary follows.
\end{proof}

\begin{lem}
 \label{fgisstric}
 Let $\mc{A}$ be a strictly noetherian sheaf with respect to $\mf{B}$
 (resp.\ {\normalfont (SH)}-noetherian sheaf) on a topological space
 $X$, then so is $\mc{A}[T]$.
\end{lem}
\begin{proof}
 It is easy to check that $\mc{A}$ is noetherian if and only if
 $\mc{A}[T]$ is a noetherian ring since
 $\Gamma(U,\mc{A}[T])\cong\Gamma(U,\mc{A})[T]$ for any open set
 $U$. Assume $\mc{A}$ to be (SH)-noetherian, and let $\mc{M}$ be a
 coherent $\mc{A}[T]$-module. It suffices to show that $\mc{M}$
 satisfies (SH).
 Since the verification is local by Lemma \ref{SHilocal}, we may assume
 that there exists integers $a,b\geq0$ and a presentation
 \begin{equation*}
  \mc{A}[T]^{\oplus a}\xrightarrow{\psi}\mc{A}[T]^{\oplus b}
   \xrightarrow{\phi}\mc{M}\rightarrow0.
 \end{equation*}
 Let $\mc{A}_n:=\bigoplus_{i\leq n}\mc{A}\cdot T^i$ and
 $\mc{K}'_n:=\psi(\mc{A}_n)$. Let
 $\mc{K}_{m,n}:=\mc{K}'_m\cap\mc{A}_n^{\oplus b}$ in
 $\mc{A}[T]^{\oplus b}$, which is a coherent sub-$\mc{A}$-module of
 $\mc{A}_n^{\oplus b}\subset\mc{A}[T]^{\oplus b}$. Since $\mc{A}$ is
 strictly noetherian,
 $\mc{K}_n:=\bigcup_{m\geq0}\mc{K}_{m,n}=\mr{Ker}(\phi)\cap
 \mc{A}_n$ is a coherent $\mc{A}$-module by Proposition
 \ref{stationary}. We define a coherent $\mc{A}$-module $\mc{M}_n$ by
 $\mc{A}^{\oplus b}_n/\mc{K}_n\subset\mc{M}$. By construction,
 $\bigcup_n\mc{M}_n=\mc{M}$, and $(\mc{M},\{\mc{M}_n\}_{n\in\mb{Z}})$ is
 a filtered $(\mc{A}[T],\{\mc{A}_n\}_{n\in\mb{Z}})$-module.

 It suffices to show that $\mr{gr}(\mc{M})$ satisfies (SH) by Lemma
 \ref{filtration}. We have checked $\mr{gr}_i(\mc{M})$ is a coherent
 $\mc{A}$-module for
 any $i$.
 By construction, the homomorphism
 $T\colon\mr{gr}_i(\mc{M})\rightarrow\mr{gr}_{i+1}(\mc{M})$ is
 surjective for any $i$, and we have the following sequence of
 surjections.
 \begin{equation*}
  \mc{M}_0=\mr{gr}_0(\mc{M})\twoheadrightarrow\dots
   \twoheadrightarrow\mr{gr}_i(\mc{M})\twoheadrightarrow
   \mr{gr}_{i+1}(\mc{M})\twoheadrightarrow\dots
 \end{equation*}
 Since $\mc{M}_0$ is coherent and $\mc{A}$ is strictly
 noetherian, this sequence is stationary, and there exists an integer
 $N$ such that $T^i\colon\mr{gr}_N(\mc{M})\rightarrow\mr{gr}_{N+i}
 (\mc{M})$ is an isomorphism for any $i\geq0$. Since $\bigoplus_{0\leq
 i\leq N}\mr{gr}_i(\mc{M})$ is a coherent $\mc{A}$-module, it
 satisfies (SH) with respect to a family $\mf{Z}$. Then
 $\mr{gr}(\mc{M})$ satisfies (SH) with respect to $\mf{Z}$.
\end{proof}

\begin{cor}
 \label{tensQsnoe}
 Let $X$ be a topological space satisfying {\normalfont(FT)}, and
 $\mc{A}$ be a strictly noetherian $R$-module on $X$ with respect to
 $\mf{B}$ (resp.\ {\normalfont (SH)}-noetherian). Then so is
 $\mc{A}\otimes\mb{Q}$.
\end{cor}
\begin{proof}
 We only prove the strictly noetherian case.
 It is easily verified that $\mc{A}\otimes\mb{Q}$ is noetherian with
 respect to $\mf{B}$.
 Let $\mc{A}_n:=\mr{Ker}(\mc{A}\xrightarrow{p^n}\mc{A})$. Since $\mc{A}$
 is strictly noetherian, the sequence
 $\mc{A}_0\subset\mc{A}_1\subset\dots\subset\mc{A}$ is stationary. Let
 $\mc{A}_\infty:=\indlim_n\mc{A}_n$. Since $\mc{A}/\mc{A}_\infty$ is a
 coherent $\mc{A}$-algebra, it is strictly noetherian, and we may assume
 that $\mc{A}$ is a flat $R$-module in the corollary. Let
 $F_i(\mc{A}\otimes\mb{Q}):=\pi^{-i}\mc{A}$ for $i\geq0$ and
 $F_i(\mc{A}\otimes\mb{Q})=0$ for $i<0$. Then it suffices to show that
 $\mr{gr}^F(\mc{A}\otimes\mb{Q})$ is strictly noetherian by Lemma
 \ref{takinggraded}. Since $\mr{gr}^F(\mc{A}\otimes\mb{Q})$ is a
 coherent $\mc{A}[T]$-algebra where the action of $T$ is the
 multiplication by $\pi^{-1}\in\mr{gr}^F_1(\mc{A}\otimes\mb{Q})$, it is
 reduced to showing that $\mc{A}[T]$ is strictly noetherian, which
 follows from Lemma \ref{fgisstric}.
\end{proof}

\begin{lem}
\label{spacescondition}
 Let $X$ be a noetherian scheme, and let $\mf{B}$ be the open basis
 consisting of open affine subschemes of $X$. Then $\mc{O}_X$ is
 strictly noetherian with respect to $\mf{B}$.
\end{lem}
\begin{proof}
 Condition (FT) on $(X,\mc{O}_X)$ is a basic property of
 noetherian schemes. Let $\mc{M}$ be a coherent $\mc{O}_X$-module, and
 let us check (SH) for this $\mc{M}$.
 Since the
 statement is local by Lemma \ref{SHilocal}, we may suppose that
 $X=\mr{Spec}(A)$ for a noetherian ring $A$. There exists a finite
 decreasing filtration $\{\mc{M}_i\}_{0\leq i\leq n}$ on $\mc{M}$ such
 that $\mc{M}_0=\mc{M}$, $\mc{M}_n=0$, the quotient
 $\mc{M}_i/\mc{M}_{i+1}$ is irredondant for any $0\leq i<n$, and
 \begin{equation*}
  \mr{Ass}(\mc{M}_i/\mc{M}_{i+1})\subset\mr{Ass}(\mc{M})
 \end{equation*}
 by \cite[IV 3.2.8]{EGA}. By Lemma \ref{extension*}, it suffices to show
 the lemma for irredondant modules, but in this case, it follows by
 definition.
\end{proof}

Using Lemma \ref{fgisstric}, we have the following corollary.
\begin{cor*}
 \label{fgisstric2}
 Let $X$ be a noetherian scheme and $\mc{A}$ be a quasi-coherent
 $\mc{O}_X$-algebra of finite type. Then $\mc{A}$ is strictly
 noetherian.
\end{cor*}

\begin{lem}
 \label{stnoethopen}
 Let $(Y,\mc{A})$ be a ringed space satisfying {\normalfont(FT)}, and
 assume that $\mc{A}$ is strictly noetherian with respect to
 $\mf{B}$. Let $X$ be a sober and noetherian topological space, and
 $f\colon X\rightarrow Y$ be an open
 continuous map of topological spaces. Let $\mf{C}$ be an open basis of
 $X$ consisting of $V$ such that $f(V)\in\mf{B}$.
 Then the ringed space $(X,f^{-1}\mc{A})$ is strictly noetherian with
 respect to $\mf{C}$.
\end{lem}
\begin{proof}
 Let $\mc{N}$ be a coherent $\mc{A}$-module satisfying (SH) with respect
 to $\mf{Z}$. Then $f^{-1}\mc{N}$ satisfies (SH) with respect to
 $f^{-1}(\mf{Z})$. We claim that for any coherent $f^{-1}\mc{A}$-module
 $\mc{M}$, there exists a coherent $\mc{A}$-module $\mc{N}$ such that
 $\mc{M}\cong f^{-1}(\mc{N})$. Now, the functor $f_*$ is exact since $f$
 is open. Thus, for a coherent $f^{-1}\mc{A}$-module $\mc{M}$, the
 canonical homomorphism $f^{-1}f_*\mc{M}\rightarrow\mc{M}$ is an
 isomorphism, and $f_*\mc{M}|_{f(X)}$ is a coherent
 $\mc{A}|_{f(X)}$-module. Thus $f^{-1}\mc{A}$ is (SH)-noetherian. We
 leave the details to check that it is noetherian with respect to
 $\mf{C}$.
\end{proof}

\begin{thm}
 \label{stnoeththem}
 Let $\ms{X}$ be a smooth formal scheme {\em of finite type} over
 $\mr{Spf}(R)$. Then $\mc{O}_{X_i}$, $\mc{O}_{\ms{X}}$,
 $\mc{O}_{\ms{X},\mb{Q}}$, $\Dmod{m}{X_i}$,
 $\Dmod{m}{\ms{X}}$, $\Dmod{m}{\ms{X},\mb{Q}}$,
 $\Dcomp{m}{\ms{X}}$, $\DcompQ{m}{\ms{X}}$
 are strictly noetherian sheaves on
 $\ms{X}$. Moreover, $\Emod{m,m'}{X_i}$, $\Emod{m,m'}{\ms{X}}$,
 $\Ecomp{m,m'}{\ms{X}}$, $\EcompQ{m,m'}{\ms{X}}$ are strictly noetherian
 on $\mathring{T}^*\ms{X}$.
\end{thm}
\begin{proof}
 Note first that $\ms{X}$ and $\mathring{T}^*\ms{X}$ are noetherian
 spaces. The ring $\mc{O}_{X_i}$
 is strictly noetherian by Lemma \ref{spacescondition}. To check that
 $\mc{O}_{\ms{X}}$ is strictly noetherian, we consider the $\pi$-adic
 filtration. Since $\mc{O}_{\ms{X}}$ is pointwise Zariskian with respect
 to the $\pi$-adic filtration by \cite[3.3.6]{Ber1} and \cite[Ch.II, 2.2
 (4)]{HO} (or we can use Lemma \ref{noetherianlemma}), it suffices to
 show that $\mr{gr}_{\pi}(\mc{O}_{\ms{X}})\cong \mc{O}_{X_0}[T]$ is
 strictly noetherian by Lemma \ref{takinggraded} where
 $\mr{gr}_{\pi}$ denotes the $\mr{gr}$ with respect to the $\pi$-adic
 filtration and $T$ denotes the class of $\pi$. This follows from Lemma
 \ref{fgisstric}. For $\mc{O}_{\ms{X},\mb{Q}}$ use Corollary
 \ref{tensQsnoe}.

 Let $X$ be either $\ms{X}$ or $X_i$ for some $i\geq0$.
 Let us prove that $\Dmod{m}{X}$ is strictly noetherian. We consider the
 filtration by order. Since the filtration is positive, it suffices to
 show that $\mr{gr}(\Dmod{m}{X})$ is strictly noetherian by Lemma
 \ref{takinggraded}. Since
 $\mr{gr}(\Dmod{m}{X})$ is of finite type over
 $\mr{gr}_0(\Dmod{m}{X})$, and $\mr{gr}_i(\Dmod{m}{X})$ is
 coherent $\mr{gr}_0(\Dmod{m}{X})$-module for any $i$,
 $\mr{gr}(\Dmod{m}{X})$ can be seen as a coherent
 $\mc{O}_{X}[T_1,\dots,T_n]$-algebra for some $n$, and the claim follows
 by using Corollary \ref{fgisstric}.

 Let us prove that $\Dcomp{m}{\ms{X}}$ is strictly noetherian. Consider
 the $\pi$-adic filtration. Then $\Dcomp{m}{\ms{X}}$ is pointwise
 Zariskian filtered by \cite[3.3.6]{Ber1} and \cite[Ch.II, 2.2
 (4)]{HO}. By Lemma \ref{takinggraded}, it suffices to show that
 $\mr{gr}_{\pi}(\Dcomp{m}{\ms{X}})\cong(\Dmod{m}{X_0})[T]$ is strictly
 noetherian where $\mr{gr}_{\pi}$ denotes the $\mr{gr}$ with
 respect to the $\pi$-adic filtration and $T$ denotes the class of
 $\pi$. Since we checked that $\Dmod{m}{X_0}$ is strictly noetherian,
 $(\Dmod{m}{X_0})[T]$ is strictly noetherian by Lemma
 \ref{fgisstric}, and thus $\Dcomp{m}{\ms{X}}$ is strictly noetherian.

 Let $X$ be either $\ms{X}$ or $X_i$ for some $i$.
 Let us check that $\ms{E}:=\Emod{m,m'}{X}$ is strictly noetherian on
 $\mathring{T}^*\ms{X}$. Consider the filtration by order. It suffices
 to show that $\mr{gr}(\ms{E})$ is strictly noetherian by Lemma
 \ref{takinggraded} and Remark \ref{noetherianintermediate}.
 Let $q:\mathring{T}^*X\rightarrow P^*X$ be the canonical
 surjection. Then, it suffices to
 show that $q_*(\mr{gr}(\ms{E}))$ is strictly noetherian by
 (\ref{noprimecalcfromprima}). By Lemma \ref{intersectioncoherent}, we
 know that $q_*(\mr{gr}(\ms{E}))$ is an $\mc{O}_{P^{(m)*}X}$-algebra
 of finite type and $q_*(\mr{gr}_i(\ms{E}))$ is
 a coherent $\mc{O}_{P^{(m)*}X}$-module for any $i$, and we get the
 claim by using Lemma \ref{fgisstric}.

 For $\Ecomp{m,m'}{\ms{X}}$ and $\EcompQ{m,m'}{\ms{X}}$ the
 verifications are the same as those of $\Dcomp{m}{\ms{X}}$ and
 $\DcompQ{m}{\ms{X}}$, we leave the details to the reader.
\end{proof}

\section{Application: Stability theorem for curves}
In this section, we focus on the relation between the support of
the microlocalization and the characteristic variety. We formulate a
conjecture on the relation, and prove the conjecture in the curve case.

\subsection{}
\label{countexchsupp}
Recall the setting \ref{limitnaive}, and let $\ms{X}$ be a {\em
quasi-compact} smooth formal scheme over $R$.
One might expect that, for a coherent $\DcompQ{m}{\ms{X}}$-module
$\ms{M}$,
\begin{equation*}
 \mr{Char}^{(m)}(\ms{M})=\mr{Supp}(\Emod{m,\dag}{\ms{X,\mb{Q}}}
  \otimes_{\pi^{-1}\DcompQ{m}{\ms{X}}}\pi^{-1}\ms{M}).
\end{equation*}
For the definition of the characteristic varieties, see
\ref{charvardef}. However, this does not hold in general. Indeed,
suppose this were true. Then since
$\mr{Supp}(\Emod{m,\dag}{\ms{X,\mb{Q}}}\otimes\ms{M})
\supset\mr{Supp}(\Emod{m+1,\dag}{\ms{X,\mb{Q}}}\otimes\ms{M})$, we would
get $\mr{Char}^{(m)}(\ms{M})\supset\mr{Char}^{(m+1)}(\DcompQ{m+1}
{\ms{X}}\otimes\ms{M})$. However, this does not hold by Example
\ref{countcharnotstable}. Considering these, we conjecture the
following.

\begin{conj*}
 Let $\ms{X}$ be a {\em quasi-compact} smooth formal scheme over $R$,
 and $\ms{M}$ be a coherent $\DcompQ{m}{\ms{X}}$-module. Then there
 exists an integer $N\geq m$ such that for any $m'\geq N$,
 \begin{equation*}
  \mr{Char}^{(m')}(\DcompQ{m'}{\ms{X}}\otimes_{\DcompQ{m}{\ms{X}}}
   \ms{M})=\mr{Supp}(\Emod{m',\dag}{\ms{X},\mb{Q}}\otimes_{\pi^{-1}
   \DcompQ{m}{\ms{X}}}\pi^{-1}\ms{M}).
 \end{equation*}
\end{conj*}
We prove this conjecture in the case where $\ms{X}$ is a formal curve
({\it i.e.\ }dimension $1$ connected smooth formal scheme of finite
type over $R$). Namely,

\begin{thm}
 \label{main}
 Let $\ms{X}$ be a smooth formal {\em curve} over $R$.
 Let $\ms{M}$ be a coherent $\DcompQ{m}{\ms{X}}$-module.
 Then there exists an integer $N\geq m$ such that we have
 \begin{equation*}
  \mr{Char}^{(m')}(\DcompQ{m'}{\ms{X}}\otimes\ms{M})=
   \mr{Supp}(\EdagQ{\ms{X}}\otimes\ms{M})
 \end{equation*}
 for $m'\geq N$.
\end{thm}
This theorem is proven in the last part of this section.

\begin{rem}
 The conjecture and Theorem \ref{main} may seem to be different since we
 used $\Emod{m',\dag}{\ms{X},\mb{Q}}$ in the conjecture and
 $\EdagQ{\ms{X}}$ in the theorem. However, these are equivalent. Indeed,
 since there exists a homomorphism
 $\Emod{m',\dag}{\ms{X},\mb{Q}}\rightarrow
 \Emod{m'+1,\dag}{\ms{X},\mb{Q}}$ and the topological space $T^*\ms{X}$
 is noetherian, there exists an integer $a$ such that
 for any $m'\geq a$,
 \begin{equation*}
  \mr{Supp}(\EdagQ{\ms{X}}\otimes\ms{M})\subset
   \mr{Supp}(\Emod{m',\dag}{\ms{X},\mb{Q}}\otimes\ms{M})=
   \mr{Supp}(\Emod{a,\dag}{\ms{X},\mb{Q}}\otimes\ms{M}).
 \end{equation*}
 We remind that these supports are closed by \cite[$0_{\mr{I}}$,
 5.2.2]{EGA}.
 Let us show that this inclusion is in fact an equality. Since the
 problem is local, we may assume that $\ms{X}$ is affine, and take
 global generators $m_1,\dots,m_n\in\Gamma(\ms{X},\ms{M})$ of
 $\ms{M}$ over $\DcompQ{m}{\ms{X}}$. Suppose that the inclusion is not
 an equality, and take a
 point $x\in\mr{Supp}(\Emod{a,\dag}{\ms{X},\mb{Q}}\otimes\ms{M})$ which
 is not contained in $\mr{Supp}(\EdagQ{\ms{X}}\otimes\ms{M})$.
 This means that $(\EdagQ{\ms{X}}\otimes\ms{M})_x=0$.
 Now, we know that
 \begin{equation*}
  (\EdagQ{\ms{X}}\otimes\ms{M})_x\cong\indlim_{m'}~
   (\Emod{m',\dag}{\ms{X},\mb{Q}}\otimes\ms{M})_x
 \end{equation*}
 by \cite[Ch.II, 1.11]{Go}. Thus there exists an integer $m'\geq a$ such
 that the images of $m_1,\dots,m_n$ in
 $(\Emod{m',\dag}{\ms{X},\mb{Q}}\otimes\ms{M})_x$ are $0$.
 Since the latter module is generated by these elements over
 $(\Emod{m',\dag}{\ms{X},\mb{Q}})_x$, we would have
 $(\Emod{m',\dag}{\ms{X},\mb{Q}}\otimes\ms{M})_x=0$, which contradicts
 with the assumption. Summing up, we obtain
 \begin{equation*}
  \mr{Supp}(\EdagQ{\ms{X}}\otimes\ms{M})=\bigcap_{m'\geq m}
   \mr{Supp}(\Emod{m',\dag}{\ms{X},\mb{Q}}\otimes\ms{M}).
 \end{equation*}
\end{rem}

\subsection{}
Before we start proving the theorem, let us see some consequences of the
theorem.

\begin{dfn*}
 Let $\ms{X}$ be a smooth formal scheme of dimension $1$ over $R$ (not
 necessarily quasi-compact), and
 let $\ms{M}$ be a coherent $\DdagQ{\ms{X}}$-module. We define
 \begin{equation*}
  \mr{Char}(\ms{M}):=\mr{Supp}(\EdagQ{\ms{X}}\otimes\ms{M}).
 \end{equation*}
\end{dfn*}

\begin{cor}
 Suppose that $k$ is perfect and there exists a lifting
 $\sigma:R\xrightarrow{\sim}R$ of
 the absolute Frobenius automorphism of $k$, and fix one. Let $\ms{X}$
 be a smooth formal scheme of dimension $1$, and let $\ms{M}$ be a
 coherent $F$-$\DdagQ{X}$-module. Then
 \begin{equation*}
  \mr{Car}(\ms{M})=\mr{Char}(\ms{M}),
 \end{equation*}
 where $\mr{Car}$ denotes the characteristic variety defined by
 Berthelot {\normalfont(cf.\ \cite[5.2.7]{BerInt})}.
\end{cor}
\begin{proof}
 The problem being local, we may assume that $\ms{X}$ is a curve.
 Since this follows immediately from the definition of $\mr{Car}$, we
 recall the definition briefly. For a large enough integer $m$,
 we can take the Frobenius descent of level $m$ denoted by
 $\ms{M}^{(m)}$ by \cite[4.5.4]{Ber2}, which is a coherent
 $\DcompQ{m}{\ms{X}}$-module with an isomorphism
 $\DcompQ{m+1}{\ms{X}}\otimes\ms{M}^{(m)}
 \xrightarrow{\sim}F^*\ms{M}^{(m)}$. The characteristic variety of
 $\ms{M}$ is by definition $\mr{Char}^{(m)}(\ms{M}^{(m)})$. Since we
 have $\ms{M}^{(m+k)}:=\DcompQ{m+k}{\ms{X}}\otimes\ms{M}^{(m)}\cong
 F^{k*}\ms{M}^{(m)}$,
 $\mr{Char}^{(m)}(\ms{M}^{(m)})=\mr{Char}^{(m+k)}(\ms{M}^{(m+k)})$ by
 \cite[5.2.4 (iii)]{BerInt}. Thus the corollary follows by applying
 Theorem \ref{main} to $\ms{M}^{(m)}$.
\end{proof}

\subsection{}
Let us prove the theorem.
To prove the theorem, we show the following proposition first.
\begin{prop*}
\label{mainprop}
 Let $\ms{X}$ be a smooth formal curve over $R$, and $\ms{M}$ be a
 coherent $\DcompQ{m}{\ms{X}}$-module.
 Suppose that there exists an integer $N'\geq m$ such that 
 \begin{equation*}
  \mr{Char}^{(m')}(\DcompQ{m'}{\ms{X}}\otimes\ms{M})=
   \mr{Char}^{(N')}(\DcompQ{N'}{\ms{X}}\otimes\ms{M})
 \end{equation*}
 for any $m'\geq N'$. Then the conclusion of Theorem
 {\normalfont\ref{main}} holds.
\end{prop*}
The proof is given in \ref{proofmainprop}. For an interval
$I\subset\mb{R}$, we denote $I\cap\mb{Z}$ by $I_{\mb{Z}}$ in the
following.

\begin{lem}
\label{surjectionlemma}
 Let $\ms{X}$ be a smooth formal scheme. (We do not need to assume that
 $\ms{X}$ is a curve in this lemma.)
 Let $\ms{M}$ be a $\Dcomp{m}{\ms{X},\mb{Q}}$-module and $x\in
 T^*\ms{X}$. Suppose there exist integers $b\geq a\geq m$ such that
 \begin{equation*}
  (\EcompQ{m'}{\ms{X}}\otimes\ms{M})_x=0
 \end{equation*}
 for any integer $b\geq m'\geq a$. Then the canonical homomorphism
 \begin{equation}
  \label{surjhom}
   (\EcompQ{l,\bullet}{\ms{X}}\otimes
   \ms{M})_x\rightarrow(\EcompQ{m',\bullet}{\ms{X}}\otimes
   \ms{M})_x
 \end{equation}
 is split surjective for any integers $b\geq m'\geq a$, $m'\geq l\geq
 m$, and $\bullet\in\bigl\{\dag,\left[m',\infty\right[_{\mb{Z}}\bigr\}$. Here
 $\EcompQ{l,\dag}{\ms{X}}$ means $\Emod{l,\dag}{\ms{X},\mb{Q}}$ by
 abuse of language. Moreover, we get
 \begin{align*}
  \mr{Ker}\bigl((\EcompQ{l,\bullet}{\ms{X}}\otimes
  \ms{M})_x\rightarrow(\EcompQ{m',\bullet}{\ms{X}}&\otimes
  \ms{M})_x\bigr)
  \cong\mr{Tor}_1^{\DcompQ{m}{\ms{X}}}(\EcompQ{m',\bullet}
  {\ms{X}}/\EcompQ{l,\bullet}{\ms{X}},\ms{M})_x\\
  &\cong
  \begin{cases}
   (\EcompQ{l,a}{\ms{X}}\otimes\ms{M})_x\oplus
   \bigoplus_{i=a}^{m'-1}(\EcompQ{i,i+1}{\ms{X}}\otimes\ms{M})_x
   &\text{if $l<a$} \\
   \bigoplus_{i=l}^{m'-1}(\EcompQ{i,i+1}{\ms{X}}\otimes\ms{M})_x
   &\text{if $l\geq a$.}
  \end{cases}  
 \end{align*}
\end{lem}
\begin{proof}
 For $\bullet\in\bigl\{\dag,\left[m',\infty\right[_{\mb{Z}}\bigr\}$,
 \begin{equation*}
 \EcompQ{m',\bullet}{\ms{X}}/\EcompQ{l,\bullet}{\ms{X}}
  \xrightarrow{\sim}\EcompQ{m'}{\ms{X}}/\EcompQ{l,m'}{\ms{X}}
 \end{equation*}
 by Corollary \ref{keyisomdag}.
 We denote this quotient by $\ms{Q}$. We get the following diagram whose rows
 are exact:
 \begin{equation*}
 \xymatrix{
  \mr{Tor}_1(\EcompQ{m',\bullet}{\ms{X}},\ms{M})\ar[r]\ar[d]&
  \mr{Tor}_1(\ms{Q},\ms{M})\ar[r]^\beta\ar[d]^{\sim}&
  \EcompQ{l,\bullet}{\ms{X}}\otimes\ms{M}\ar[r]\ar[d]&
  \EcompQ{m',\bullet}{\ms{X}}\otimes \ms{M}\ar[r]\ar[d]&
  \ms{Q}\otimes \ms{M}\ar[d]^{\sim}\ar[r]&0\\
 \mr{Tor}_1(\EcompQ{m'}{\ms{X}},\ms{M})\ar[r]&
  \mr{Tor}_1(\ms{Q},\ms{M})\ar[r]
  ^<>(.5){\alpha}&\EcompQ{l,m'}{\ms{X}}\otimes\ms{M}\ar[r]&
  \EcompQ{m'}{\ms{X}}\otimes \ms{M}\ar[r]&\ms{Q}\otimes\ms{M}\ar[r]&0
  }
 \end{equation*}
 where the $\otimes$ and $\mr{Tor}_1$ are taken over
 $\Dcomp{m}{\ms{X},\mb{Q}}$, and we omit the pull-back of sheaves
 $\pi^{-1}$ since it is obvious where to put them. Now, since
 $\mr{Tor}^{\Dcomp{m}{\ms{X},\mb{Q}}}_1(\EcompQ{m'}
 {\ms{X}},\ms{M})=0$ by the flatness of $\EcompQ{m'}{\ms{X}}$ over
 $\DcompQ{m'}{\ms{X}}$ and $\DcompQ{m'}{\ms{X}}$ over
 $\DcompQ{m}{\ms{X}}$ (cf.\ Proposition \ref{noetheriansheafofmic}
 (ii) and \cite[3.5.3]{Ber1}), the homomorphism $\alpha$ is
 injective. Moreover, since $(\EcompQ{m'}{\ms{X}}\otimes \ms{M})_x=0$
 by the hypothesis, the homomorphism $\alpha_x$ is an
 isomorphism, and $(\ms{Q}\otimes \ms{M})_x=0$. Since $\alpha$ is
 injective, $\beta$ is injective as well. Thus, the homomorphism
 (\ref{surjhom}) is split surjective and
 \begin{equation*}
 \mr{Ker}(\EcompQ{l,\bullet}{\ms{X}}\otimes
  \ms{M}\rightarrow\EcompQ{m',\bullet}{\ms{X}}\otimes
  \ms{M})_x\cong\mr{Tor}_1(\ms{Q},\ms{M})_x\xrightarrow[\alpha_x]
  {\sim}(\EcompQ{l,m'}{\ms{X}}\otimes\ms{M})_x.
 \end{equation*}
 Let us calculate $\mr{Tor}_1(\ms{Q},\ms{M})_x$. We only treat the case
 where $l<a$, and since the proof is similar, the other case is left to
 the reader. To calculate this, it suffices to show
 \begin{equation*}
  (\EcompQ{l,m'}{\ms{X}}\otimes\ms{M})_x\cong
   (\EcompQ{l,a}{\ms{X}}\otimes\ms{M})_x\oplus
   \bigoplus_{i=a}^{m'-1}(\EcompQ{i,i+1}{\ms{X}}\otimes\ms{M})_x.
 \end{equation*}
 We use the induction on $k:=m'-a$. For $k=0$, the claim
 is redundant. Suppose that the statement holds to be true for
 $k=k_0-1\geq0$. Then it suffices to show the following isomorphism for
 $m''=a+k_0$:
 \begin{equation}
  \label{splittoproviso}
  (\EcompQ{l,m''}{\ms{X}}\otimes\ms{M})_x\cong
   (\EcompQ{l,m''-1}{\ms{X}}\otimes\ms{M})_x\oplus
   (\EcompQ{m''-1,m''}{\ms{X}}\otimes\ms{M})_x.
 \end{equation}
 Indeed, we just apply the induction hypothesis to
 $(\EcompQ{l,m'-1}{\ms{X}}\otimes\ms{M})_x$ to get the conclusion.
 Let us show (\ref{splittoproviso}). Note that $m''-1\geq a$. This
 isomorphism can be shown using exactly the same method as before using
 the following diagram instead:
 \begin{equation*}
  \def\objectstyle{\scriptstyle}
   \xymatrix{
   \mr{Tor}_1(\EcompQ{l,m''-1}{\ms{X}},\ms{M})\ar[r]\ar[d]&
   \mr{Tor}_1(\ms{Q}',\ms{M})\ar[r]\ar[d]^{\sim}&
   \EcompQ{l,m''}{\ms{X}}\otimes \ms{M}\ar[r]\ar[d]&
   \EcompQ{l,m''-1}{\ms{X}}\otimes\ms{M}\ar[r]\ar[d]&
   \ms{Q}'\otimes \ms{M}\ar[d]^{\sim}\ar[r]&0\\
  \mr{Tor}_1(\EcompQ{m''-1}{\ms{X}},\ms{M})\ar[r]&
   \mr{Tor}_1(\ms{Q}',\ms{M})\ar[r]&\EcompQ{m''-1,m''}{\ms{X}}\otimes
   \ms{M}\ar[r]&
   \EcompQ{m''-1}{\ms{X}}\otimes \ms{M}\ar[r]&\ms{Q}'\otimes \ms{M}
   \ar[r]&0
   }
 \end{equation*}
 where
 \begin{equation*}
 \ms{Q}'=\EcompQ{l,m''-1}{\ms{X}}/\EcompQ{l,m''}{\ms{X}}
  \cong\EcompQ{m''-1}{\ms{X}}/\EcompQ{m''-1,m''}{\ms{X}}
 \end{equation*}
 using Lemma \ref{keyisom2}.
\end{proof}

\begin{lem}
 Let $\ms{X}$ be a {\em curve}. Then the canonical homomorphism
 \begin{equation*}
  \pi^{-1}\DcompQ{m+k}{\ms{X}}/\DcompQ{m}{\ms{X}}\rightarrow
   \Emod{m+k,\dag}{\ms{X},\mb{Q}}/\Emod{m,\dag}{\ms{X},\mb{Q}}
 \end{equation*}
 is an isomorphism.
\end{lem}
\begin{proof}
 It suffices to prove that the canonical homomorphism of sheaf of
 abelian groups $\pi^{-1}\DcompQ{m}{\ms{X}}\rightarrow
 \EcompQ{m}{\ms{X}}/(\Emod{m}{\ms{X},\mb{Q}})_{-1}$ is an
 isomorphism. To show this, it suffices to check that
 $\pi^{-1}\Dmod{m}{X_i}\rightarrow\Emod{m}{X_i}/\Emod{m}{X_i,-1}$ is an
 isomorphism, whose verification is straightforward.
\end{proof}

\subsection{Proof of Proposition \ref{mainprop}:}
\label{proofmainprop}
\quad
 Consider the following diagram of sheaves on $T^*\ms{X}$ for any
 integer $m'\geq m$.
 \begin{equation*}
  \xymatrix{
   \mr{Tor}_1(\DcompQ{m'}{\ms{X}},\ms{M})\ar[r]\ar[d]&
   \mr{Tor}_1(\DcompQ{m'}{\ms{X}}/\DcompQ{m}{\ms{X}},\ms{M})
   \ar[r]\ar[d]^{\sim}&\ms{M}\ar[d]\ar[r]&
   \DcompQ{m'}{\ms{X}}\otimes\ms{M}\ar[d]\\
  \mr{Tor}_1(\EcompQ{m'}{\ms{X}},\ms{M})\ar[r]&
   \mr{Tor}_1(\EcompQ{m'}{\ms{X}}/\EcompQ{m,m'}
   {\ms{X}},\ms{M})\ar[r]_<>(.5){\alpha}&\EcompQ{m,m'}{\ms{X}}
   \otimes\ms{M}\ar[r]&\EcompQ{m'}{\ms{X}}\otimes\ms{M}
   }
 \end{equation*}
 Here $\otimes$ and $\mr{Tor}$ are taken over $\DcompQ{m}{\ms{X}}$, and
 we omit the pull-back of sheaves $\pi^{-1}$.
 Since $\DcompQ{m'}{\ms{X}}$ and $\EcompQ{m'}{\ms{X}}$ are flat over
 $\DcompQ{m}{\ms{X}}$, the left vertical arrow of the diagram is just
 $0\rightarrow 0$. Thus the homomorphism $\alpha$ is injective. Consider
 the following commutative diagram
 \begin{equation*}
  \xymatrix{
   \mr{Tor}_1(\Emod{m',\dag}{\ms{X},\mb{Q}}/\Emod{m,\dag}{\ms{X},\mb{Q}}
   ,\ms{M})\ar[r]^<>(.5){\beta}\ar[d]_\sim&
   \Emod{m,\dag}{\ms{X},\mb{Q}}\otimes\ms{M}\ar[d]\\
   \mr{Tor}_1(\EcompQ{m'}{\ms{X}}/\EcompQ{m,m'}{\ms{X}},\ms{M})
    \ar[r]_<>(.5)\alpha&\EcompQ{m,m'}{\ms{X}}\otimes\ms{M}
   }
 \end{equation*}
 where the isomorphism is by Lemma \ref{keyisomdag}. Since $\alpha$ is
 injective, $\beta$ is also injective. This implies that
 \begin{align*}
  \tag{*}
  \ms{K}_{m'}:=\mr{Ker}\bigl(\varphi_{m'}\colon
  \Emod{m,\dag}{\ms{X},\mb{Q}}\otimes\pi^{-1}\ms{M}
  \rightarrow&\Emod{m',\dag}{\ms{X},\mb{Q}}\otimes\pi^{-1}\ms{M}\bigr)
  \cong\mr{Tor}_1(\Emod{m',\dag}{\ms{X},\mb{Q}}/
  \Emod{m,\dag}{\ms{X},\mb{Q}},\pi^{-1}\ms{M})\\
  &\cong\pi^{-1}\mr{Tor}_1(\DcompQ{m'}{\ms{X}}/\DcompQ{m}{\ms{X}},\ms{M})
  \hookrightarrow\pi^{-1}\ms{M}.
 \end{align*}
 Since $\pi^{-1}\DcompQ{m}{\ms{X}}$ is strictly noetherian by Lemma
 \ref{stnoethopen} and Theorem \ref{stnoeththem}, there exists an
 integer $N$ such that $\ms{K}_k=\ms{K}_N$ for $k\geq N$. So far we have
 {\em not} used the assumption on the characteristic varieties.

 By changing $m$ if necessarily, we may assume that $m=N'$. Now, by
 this assumption,
 \begin{equation*}
  Z:=\mr{Supp}(\EcompQ{m}{\ms{X}}\otimes\ms{M})=
   \mr{Supp}(\EcompQ{m'}{\ms{X}}\otimes\ms{M})
 \end{equation*}
 for any $m'\geq m$. Take $x\in\mathring{T}^*\ms{X}\setminus Z$. Then Lemma
 \ref{surjectionlemma} shows that $\varphi_{m',x}$ (see (*)) is split
 surjective. Thus, the homomorphism
 \begin{equation*}
  \psi_{m',x}\colon(\Emod{N,\dag}{\ms{X},\mb{Q}}
   \otimes\ms{M})_x\rightarrow(\Emod{m',\dag}{\ms{X},\mb{Q}}
   \otimes\ms{M})_x
 \end{equation*}
 is also surjective for any $m'\geq N$. Since the kernel is isomorphic
 to $(\ms{K}_{m'}/\ms{K}_N)_x$, the homomorphism $\psi_{m',x}$ is an
 isomorphism by the choice of $N$. Thus using Lemma
 \ref{surjectionlemma} again,
 \begin{equation*}
  (\EcompQ{m',m''}{\ms{X}}\otimes\ms{M})_x=0
 \end{equation*}
 for any integer $m'\geq N$ and $m''\geq m'$.

 Let $\ms{U}$ be the complement of $Z$. Let $\ms{V}\subset\ms{U}$ be a
 strictly affine open subscheme. By the above observation, we get
  $\Gamma(\ms{V},\EcompQ{m',m''}{\ms{X}}\otimes\ms{M})=0$.
 Since $\Gamma(\ms{V},\Emod{m',\dag}{\ms{X},\mb{Q}})$ is a
 Fr\'{e}chet-Stein algebra, we have
 $\Gamma(\ms{V},\Emod{m',\dag}{\ms{X},\mb{Q}}\otimes\ms{M})=0$.
 Thus the proposition follows.

\subsection{Proof of Theorem \ref{main}:}\quad
 We use the notation in the proof of Proposition \ref{mainprop}. There
 exists an integer $M$ such that $\ms{K}_k=\ms{K}_M$ for $k\geq
 M$. Note that to prove the existence of this $M$, we did not use the
 assumption of Proposition \ref{mainprop}. Let
 $x\in\mathring{T}^*\ms{X}$. Suppose there is an integer $m'>M$ such that
 $(\EcompQ{m'}{\ms{X}}\otimes\ms{M})_x=0$. Then by using Lemma
 \ref{surjectionlemma},
 \begin{equation*}
  0=(\EcompQ{m'}{\ms{X}}\otimes\ms{M})_x\cong
   (\EcompQ{m,m'}{\ms{X}}\otimes\ms{M})_x/\ms{K}_{m',x}=
   (\EcompQ{m,m'}{\ms{X}}\otimes\ms{M})_x/\ms{K}_{m'',x}
   \subset(\EcompQ{m'',m'}{\ms{X}}\otimes\ms{M})_x
 \end{equation*}
 for an integer $m'\geq m''\geq M$.
 However, since $(\EcompQ{m'',m'}{\ms{X}}\otimes\ms{M})_x$ is generated
 by the image of $\ms{M}$, we have
 $(\EcompQ{m'',m'}{\ms{X}}\otimes\ms{M})_x=0$. This implies that
 $(\EcompQ{m''}{\ms{X}}\otimes\ms{M})_x=0$ for $m'\geq m''\geq M$. Thus,
 \begin{equation*}
    \mr{Supp}(\EcompQ{m'}{\ms{X}}\otimes\ms{M})\subset
   \mr{Supp}(\EcompQ{m'+1}{\ms{X}}\otimes\ms{M})
 \end{equation*}
 for $m'\geq M$.

 Now suppose there exists $M'\geq M$ such that
 $\mr{Supp}(\EcompQ{M'}{\ms{X}}\otimes\ms{M})=T^*\ms{X}$. Then there is
 nothing to prove. Using \cite[5.2.4]{Ga} we may suppose that the
 dimension of $\mr{Supp}(\EcompQ{m'}{\ms{X}}\otimes\ms{M})$ is equal to
 $1$ for any $m'\geq M$.

 In this case, for any $m'\geq M$, there exists an open formal subscheme
 $\ms{U}_{m'}$ of $\ms{X}$ such that
 \begin{equation*}
  \mr{Char}^{(m')}(\DcompQ{m'}{\ms{X}}\otimes
   \ms{M}|_{\ms{U}_{m'}})\cap~T^*\ms{U}_{m'}=\ms{U}_{m'},
 \end{equation*}
 and $\ms{U}_{m'}\supset\ms{U}_{m'+1}$. Let
 $\ms{M}^{(m')}:=\DcompQ{m'}{\ms{X}}\otimes\ms{M}$. The module
 $\ms{M}^{(m')}|_{\ms{U}_{m'}}$ is a coherent
 $\mc{O}_{\ms{U}_{m'},\mb{Q}}$-module. Let $r_{m'}$ be the rank as a
 locally projective $\mc{O}_{\ms{U}_{m'},\mb{Q}}$-module. Then we know
 that $r_{m'}\geq r_{m'+1}$ for any $m'\geq M$. There exists an integer
 $N\geq M$ such that $r_N=r_{m'}$ for any $m'\geq N$. By the choice of
 $N$, the canonical homomorphism
 \begin{equation*}
  \Gamma(\ms{U}_{m'},\ms{M}^{(N)})\rightarrow
   \Gamma(\ms{U}_{m'},\ms{M}^{(m')})
 \end{equation*}
 is an isomorphism for $m'\geq N$. Indeed, since both sides are finite
 over $\Gamma(\ms{U}_{m'},\mc{O}_{\ms{X},\mb{Q}})$, and the image is
 dense, the homomorphism is surjective. However, since the rank of
 the both sides are the same by assumption, the homomorphism
 is an isomorphism. Now, the proof of \cite[2.16]{Og} is shows that for
 a smooth curve $\ms{Y}$ and a $\DcompQ{m}{\ms{Y}}$-module $\ms{M}$
 which is coherent as an $\mc{O}_{\ms{Y},\mb{Q}}$-module, if there
 exists an open formal subscheme $\ms{V}$ such that $\ms{M}|_{\ms{V}}$
 is a $\DcompQ{m+1}{\ms{V}}$-module, then $\ms{M}$ is a
 $\DcompQ{m+1}{\ms{Y}}$-module. Thus $\ms{M}^{(N)}|_{\ms{U}_{N}}$ is
 already a $\DcompQ{m'}{\ms{U}_N}$-module, and
 $\ms{M}^{(N)}|_{\ms{U}_N}\xrightarrow{\sim}
 \DcompQ{m'}{\ms{U}_N}\otimes\ms{M}^{(N)}|_{\ms{U}_N}$.
 This implies that the condition of Proposition \ref{mainprop} holds,
 and we obtain the theorem.
 \qed

Tomoyuki Abe:\\
Institute for the Physics and Mathematics of the Universe (WPI)\\
The University of Tokyo\\
5-1-5 Kashiwanoha,  
Kashiwa, Chiba, 277-8583, Japan\\
e-mail: {\tt tomoyuki.abe@ipmu.jp}

\begin{thebibliography}{AtMa}
 \bibitem[Ab]{AbeL}
	     Abe, T.: {\em Langlands program for p-adic coefficients and
	     the petits camarades conjecture}, available at
	     \url{arxiv.org/abs/1111.2479}.

 \bibitem[AM]{AM}
	     Abe, T., Marmora, A.: {\em Product formula for $p$-adic
	     epsilon factors}, available at
	     \url{arxiv.org/abs/1104.1563}.

 \bibitem[AtMa]{AtMac}
		Atiyah, M. F., MacDonald, I. G.: {\em introduction to
	       commutative algebra}, Addison-Wesley Publishing (1969).

 \bibitem[Be1]{Ber1}
	      Berthelot, P.: {\em {${\mathscr D}$}-modules
	      arithm\'etiques. {I}. {O}p\'erateurs diff\'erentiels de
	      niveau fini}, Ann.\ Sci.\ \'Ecole Norm.\ Sup.\ {\bf 29},
	      p.185--272 (1996).
	       
 \bibitem[Be2]{Ber2}
	      Berthelot, P.:
	      {\em{$\mathscr D$}-modules arithm\'etiques. {II}. Descente par
	      {F}robenius},
	      M\'em.\ Soc.\ Math.\ Fr.\ {\bf 81} (2000).
	       
 \bibitem[BeI]{BerInt}
	      Berthelot, P.:
	      {\em Introduction \`a la th\'eorie arithm\'etique des {$\mathscr
	      D$}-modules},
	      Ast\'erisque {\bf 279}, p.1--80 (2002).
	 
 \bibitem[BGR]{BGR}
	      Bosch, S., G\"{u}ntzer, U., Remmert, R.: {\em
	      Non-Archimedean Analysis}, Grundlehren der
	      math. Wissenschaften {\bf 261}, Springer (1984).

 \bibitem[Bj]{Bjo}
	     Bj\"{o}rk, J. E.: {\em Analytic $\mc{D}$-modules and
	     applications}, Mathematics and its applications {\bf 247},
	     Kluwer (1993).

 \bibitem[Bo]{Bour}
	     Bourbaki, N.: {\em Alg\`{e}bre Commutative}, Hermann.

 \bibitem[CM]{CM}
	     Christol, G., Mebkhout, Z.: {\em Equations
	     diff\'{e}rentielles $p$-adiques et coefficients $p$-adiques
	     sur les courbes}, Asterisque {\bf 279}, p.125--183 (2002).

 \bibitem[Cr]{Cr}
	     Crew R.: {\em $F$-isocrystals and their monodromy groups}, 
	     Ann.\ Sci.\ \'Ecole Norm.\ Sup.\ {\bf 25}, p.429--464
	     (1992).

 \bibitem[De]{De}
	      Deligne, P.: {\em La conjecture de Weil.II}, Publ.\ Math.\
	     IHES {\bf 52}, p.137--252 (1980).
	       
 \bibitem[EGA]{EGA}
	      Grothendieck, A.: {\em El\'{e}ments de G\'{e}om\'{e}trie
	      Alg\'{e}brique}, Publ.\ Math.\ IHES


 \bibitem[Ga]{Ga}
	     Garnier, L.: {\em Th\'{e}or\`{e}mes de division sur
	     $\widehat{\mc{D}}^{(0)}_{\mc{X}\mb{Q}}$ et applications},
	     Bull.\ Soc.\ Math.\ Fr.\ {\bf 123}, p.547--589 (1995).

 \bibitem[Go]{Go}
	     Godement, R.: {\em Topologie alg\'{e}brique et th\'{e}orie
	     des faisceaux}, Hermann.

 \bibitem[Ha]{Har}
	     Hartshorne, R.: {\it Algebraic geometry}, GTM {\bf 52},
	     Springer.

 \bibitem[HO]{HO}
	     Huishi, L., Oystaeyen, F.: {\em Zariskian filtrations},
	     $K$-Monographs in math.\ {\bf 2}, Kluwer (1996).

 \bibitem[KK]{KK}
	     Kashiwara, M., Kawai, T.: {\em On holonomic system of
	     micraodifferential equations. III},
	     Publ.\ RIMS {\bf 17}, p.813--979 (1981).

 \bibitem[Laf]{Laf}
	      Lafforgue, L.:{\em Chtoucas de Drinfeld et correspondance
	      de Langlands}, Invent.\ Math.\ {\bf 147}, p.1--241
	      (2002).

 \bibitem[Lam]{Lam}
	      Lam, T. Y.: {\em Lectures on Modules and Rings}, GTM {\bf
	      189}, Springer.

 \bibitem[Lau]{Lau}
	     Laumon, G.: {\em Transformations canoniques et
	     sp\'{e}cialisation pour les $\ms{D}$-modules filtr\'{e}s},
	     Asterisque {\bf 130}, p.56--129 (1985).

 \bibitem[Lo]{Lop}
	     L\'{o}pez, R. G.: {\em Microlocalization and stationary
	     phase}, Asian J. Math.\ {\bf 8}, p.747--768 (2004).

 \bibitem[Ma]{Mar}
	      Marmora, A.: {\em Microdiff\'{e}rentielles
	      arithm\'{e}tiques sur une courbe}, in preparation.

 \bibitem[Og]{Og}
	     Ogus, A.: {\em $F$-isocrystals and de Rham cohomology II},
	     Duke Math.\ J.\ {\bf 51}, no.\ 4, p.765--850 (1984).

 \bibitem[SGA4]{SGA}
	      Grothendieck, A., {\it et al}.: {\em Th\'{e}orie des
	       topos et cohomologie \'{e}tale des sch\'{e}mas (SGA 4)}.
      
 \bibitem[ST]{ST}
	     Schneider, P., Teitelbaum, J.: {\em Algebras of p-adic
	     distributions and admissible representations},
	     Invent.\ Math.\ {\bf 153}, p.145--196 (2003).
\end{thebibliography}
\end{document}